\documentclass{amsart}
\usepackage[utf8]{inputenc}
\usepackage{amsmath}
\usepackage{amsthm}
\usepackage{amsfonts}
\usepackage{amssymb}
\usepackage{mathrsfs}
\usepackage{xcolor}
\definecolor{darkgreen}{rgb}{0,0.45,0}
\definecolor{darkred}{rgb}{0.75,0,0}
\definecolor{darkblue}{rgb}{0,0,0.6}
\usepackage[colorlinks,citecolor=darkgreen,linkcolor=darkblue,urlcolor=darkred]{hyperref}
\usepackage{mathtools}
\usepackage{physics}
\usepackage{tikz}
\usepackage{tikz-cd}

\usepackage{bbm}
\newcommand{\one}{\mathbbm{1}}

\theoremstyle{plain}
\newtheorem{theorem}{Theorem}[section]
\newtheorem{lemma}[theorem]{Lemma}
\newtheorem{prop}[theorem]{Proposition}
\newtheorem{corollary}[theorem]{Corollary}
\newtheorem{thm}[theorem]{Theorem}
\newtheorem{conj}[theorem]{Conjecture}

\newtheorem{thmx}{Theorem}

\theoremstyle{remark}
\newtheorem{example}[theorem]{Example}
\newtheorem{remark}[theorem]{Remark}

\theoremstyle{definition}
\newtheorem{definition}[theorem]{Definition}

\numberwithin{equation}{section}

\newcommand{\Z}{\mathbb Z}
\newcommand{\A}{\mathbb A}
\newcommand{\C}{\mathbb C}
\newcommand{\G}{\mathbb G}
\newcommand{\RR}{\mathbb R}
\newcommand{\iso}{\cong}
\newcommand{\X}{\mathbb X}
\newcommand{\GL}{\operatorname{GL}}
\newcommand{\SL}{\operatorname{SL}}
\newcommand{\PU}{\operatorname{PU}}
\newcommand{\U}{\operatorname{U}}
\newcommand{\Hom}{\operatorname{Hom}}
\renewcommand{\O}{\mathcal O}
\newcommand{\g}{\mathfrak g}
\newcommand{\PGL}{\operatorname{PGL}}

\newcommand{\Stab}{\operatorname{Stab}}
\newcommand{\Aut}{\operatorname{Aut}}
\newcommand{\Nm}{\operatorname{Nm}}

\newcommand{\R}{\operatorname{R}}

\newcommand{\id}{\operatorname{id}}
\newcommand{\Spin}{\operatorname{Spin}}
\newcommand{\coker}{\operatorname{coker}}
\newcommand{\End}{\operatorname{End}}
\newcommand{\adj}{\operatorname{adj}}
\newcommand{\Ind}{\operatorname{Ind}}
\newcommand{\disc}{\operatorname{disc}}
\newcommand{\cusp}{\operatorname{cusp}}
\newcommand{\res}{\operatorname{res}}

\newcommand{\ra}{\ensuremath{\displaystyle\mathop{\rightarrow}}}
\newcommand{\ang}[1]{\langle #1\rangle}

\title{Global long root $A$-packets for $\mathsf{G}_2$: the dihedral case}

\author[P. Baki\'{c}]{Petar Baki\'{c}}
\address{Department of Mathematics\\University of Utah\\155 S 1400 E\\Salt Lake City, UT 84112}
\email{bakic@math.utah.edu}
\author[A. Horawa]{Aleksander Horawa}
\address{Mathematical Institute\\University of Bonn\\
Endenicher Allee 60 \\ 53115 Bonn, Germany}
\email{horawa@math.uni-bonn.de}
\author[S. D. Li-Huerta]{Siyan Daniel Li-Huerta}
\address{Department of Mathematics\\Massachusetts Institute of Technology\\77 Massachusetts Avenue\\Cambridge, MA 02139}
\email{sdlh@mit.edu}
\author[N. Sweeting]{Naomi Sweeting}
\address{Department of Mathematics\\Princeton University\\Fine Hall, Washington Road\\Princeton, NJ 08540}
\email{naomiss@math.princeton.edu}

\date{}

\begin{document}

\maketitle

\begin{abstract}
Cuspidal automorphic representations $\tau$ of $\PGL_2$ correspond to global long root $A$-parameters for $\mathsf{G}_2$. Using an exceptional theta lift between $\PU_3$ and $\mathsf{G}_2$, we construct the associated global $A$-packet and prove the Arthur multiplicity formula for these representations when $\tau$ is dihedral and satisfies some technical hypotheses. We also prove that this subspace of the discrete automorphic spectrum forms a full near equivalence class. Our construction yields new examples of quaternionic modular forms on $\mathsf{G}_2$.
\end{abstract}

\setcounter{tocdepth}{1}
\tableofcontents

\section{Introduction}

The Langlands philosophy predicts that automorphic representations should be classified in terms of Galois representations. More specifically, for a connected reductive group $G$ over a number field $F$, Arthur \cite{Arthur, Arthur1, Arthur2} has proposed a description of its discrete automorphic spectrum $\mathcal{A}_{\disc}(G)$ in terms of \emph{global $A$-parameters}
$$\psi \colon L_F \times \SL_2(\C) \to {^L} G,$$
where $L_F$ denotes the conjectural Langlands group of $F$, and ${^L} G$ denotes the Langlands $L$-group of $G$. For every place $v$ of $F$, one expects to associate with $\psi_v\coloneqq\psi|_{L_{F_v}\times\SL_2(\C)}$ a finite set $\Pi(\psi_v)$ of representations of $G(F_v)$, and Arthur gives a conjectural formula for how representations whose component at $v$ lies in $\Pi(\psi_v)$ contribute to the discrete automorphic spectrum. This is the associated \emph{global $A$-packet}.

We focus on the case where $G$ is the split simple group over $F$ of type $\mathsf{G}_2$. Then ${^L}G=\widehat{G}(\C)$, where $\widehat{G}$ is a simple group over $\C$ that is also of type $\mathsf{G}_2$. The possibilities for $\psi|_{\SL_2(\C)}$ correspond to nilpotent orbits in~$\widehat{G}$, so when $\psi|_{\SL_2(\C)}$ is nontrivial, there are four options:
\begin{itemize}
    \item \emph{regular orbit}: then $\psi|_{L_F}$ must be trivial, and the associated automorphic representations are trivial,
    \item \emph{subregular orbit}: in this case, the Arthur multiplicity formula is completely resolved in \cite{gan2002cubic} and \cite{Gan:Mult_formula_for_cubic},
    \item \emph{short root orbit}: in this case, the Arthur multiplicity formula is  completely resolved in \cite{GanGurevich:Non-tempered},
    \item \emph{long root orbit}: this case is the subject of our paper.
\end{itemize}
Assume now that $\psi|_{\SL_2(\C)}$ corresponds to the long root orbit. Then the centralizer of $\psi|_{\SL_2(\C)}$ in $\widehat{G}(\C)$ equals the subgroup $\SL_2(\C)$ associated with a short root in $\mathsf{G}_2$, so $\psi|_{L_F}$ corresponds to a tempered global $A$-parameter $L_F\rightarrow\SL_2(\C)$. Under the Tannakian definition of $L_F$, this corresponds to an irreducible cuspidal automorphic representation $\tau$ of $\PGL_2$ over $F$. We henceforth use $\tau$ to index long root global $A$-parameters, to avoid using the conjectural $L_F$.

In this paper, we address the case where $\tau$ is \emph{dihedral} (cf.~\S \ref{ss:dihedral}). Write $K/F$ for the associated quadratic extension and $\chi:K^\times\backslash\A_K^\times\rightarrow S^1$ for the associated character, where $\chi|_{\A_F^\times}$ must equal the quadratic character corresponding to $K/F$ in order for the automorphic induction $\tau$ of $\chi$ to descend to $\PGL_2$.

For every place $v$ of $F$, in \S\ref{sec:AMF_and_HPS} we associate with $\tau_v$ a finite set $\Pi(\tau_v)$ of representations of $G(F_v)$, using work of Alonso--He--Ray--Roset \cite{AHRR23} when $v$ is nonarchimedean. The set $\Pi(\tau_v)$ consists of one irreducible representation $\{\pi_v^+\}$ when $v$ splits in $K$ or $\chi_v^2=1$, and two irreducible representations $\{\pi^+_v,\pi^-_v\}$ otherwise. Write $\mathcal{A}_\tau(G)\subseteq\mathcal{A}_{\disc}(G)$ for the sum of all irreducible automorphic representations that are \emph{nearly equivalent} to $(\pi^+_v)_v$, i.e. whose component at $v$ is isomorphic to $\pi^+_v$ for cofinitely many $v$.

Our first main result is the Arthur multiplicity formula for these global $A$-packets.

\begin{thmx}[Theorem \ref{thm:actualA}.(3)]\label{thmA}
    Assume that
    \begin{enumerate}
        \item $L(\textstyle\frac12, \chi)$ is nonzero,
        \item $K/F$ is unramified at all places of $F$ above $2$,
        \item $F$ is totally real and, for all archimedean places $v$ of $F$ where $K_v=\RR\times\RR$, we have $\chi_v(-1) = 1$.
    \end{enumerate}
    Then we have an isomorphism of $G(\A_F)$-representations
    \begin{align}\label{AMF} \mathcal{A}_\tau(G)\cong\bigoplus\limits_{(\epsilon_v)_v}\bigotimes_v\!'\,\pi^{\epsilon_v}_v,
    \end{align}where $(\epsilon_v)_v$ runs over sequences in $\{\pm1\}$ indexed by places $v$ of $F$ such that
    \begin{itemize}
        \item $\epsilon_v=+1$ when $v$ splits in $K$ or $\chi_v^2=1$,
        \item $\prod_v\epsilon_v = \epsilon(\textstyle\frac12, \chi^3)$.
    \end{itemize}
Moreover, $\bigotimes'_v\pi_v^{\epsilon_v}$ is not cuspidal if and only if every $\epsilon_v$ equals $+1$ and $L(\textstyle\frac12,\chi^3)$ is nonzero.
\end{thmx}

\begin{remark} In general, a global $A$-packet is a subspace of a near equivalence class, and this inclusion can be strict \cite[Theorem B]{Bla94}. However, in our setting Theorem \ref{thmA} shows that this inclusion is an equality.
\end{remark}
\begin{remark}\label{rem:assumptions}
We expect Theorem \ref{thmA} to hold without assumptions (1)--(3). In our paper, we explain how to remove assumption (3) if we knew certain Howe duality results for \emph{exceptional theta lifts} over archimedean fields (Conjecture \ref{conj:G2HPSarch}). Removing assumption (2) amounts to computing certain local root numbers over $2$-adic fields (Remark \ref{rmk:assumption_at_2}). While we expect Theorem \ref{thmA} to hold without assumption (1) if we replace $\epsilon(\textstyle\frac12,\chi^3)$ with $\epsilon(\textstyle\frac12,\chi)\epsilon(\textstyle\frac12,\chi^3)$, our method crucially relies on assumption (1), as we explain below.
\end{remark}
Let us discuss the proof of Theorem \ref{thmA}. The non-cuspidal part follows from work of H. Kim \cite{kim1996residual} and \v{Z}ampera \cite{zampera1997residual}, so we focus on the cuspidal part. Since $\tau$ is dihedral, the associated global $A$-parameter $\psi$ factors through a subgroup $\SL_3(\C)\rtimes\Z/2\Z$ of $\widehat{G}(\C)$, which is the Langlands $L$-group of $\PU_3$. The resulting global $A$-parameter of $\PU_3$ is associated with certain cuspidal automorphic representations whose existence was first observed by Howe--Piatetski-Shapiro \cite[p.~315]{HPS79}. Hence Langlands functoriality suggests that we should construct our global $A$-packet on $G$ by lifting these \emph{global Howe--Piatetski-Shapiro $A$-packets} on $\PU_3$. Indeed, there is a dual reductive pair $G\times\PU_3\hookrightarrow\widetilde{G}$, where $\widetilde{G}$ is the quasi-split adjoint group of type $\mathsf{E}_6$ with respect to $K/F$. By using an analogue of theta functions for $\widetilde{G}$, one obtains an \emph{exceptional theta lift} $\theta$ between $G$ and $\PU_3$.

We now describe the Howe--Piatetski-Shapiro $A$-packets in more detail. For every place $v$ of $F$, one can associate with $\chi_v$ a finite set $\Pi(\chi_v)$ of representations of $\PU_3(F_v)$: it consists of one irreducible representation $\{\sigma_v^+\}$ when $v$ splits in $K$, one irreducible representation $\{\sigma_v^-\}$ when $\PU_3(F_v)$ is compact, and two irreducible representations $\{\sigma_v^+,\sigma_v^-\}$ otherwise. By appropriately choosing the unitary group $\PU_3$, any sequence $(\epsilon_v)_v$ as in Theorem \ref{thmA} yields an irreducible cuspidal automorphic representation $\sigma\coloneqq\bigotimes'_v\sigma_v^{\epsilon_v}$ appearing in $\mathcal{A}_{\disc}(\PU_3)$ with multiplicity one. There is flexibility in choosing $\PU_3$, and for our arguments it will be important to ensure that, when $v$ is archimedean, $\PU_3(F_v)$ is compact exactly when $\epsilon_v=-1$; otherwise, the lift $\theta(\sigma)$ of $\sigma$ to $G$ will vanish for local reasons.

Our next main result is that $\theta(\sigma)$ is isomorphic to $\bigotimes'_v\pi_v^{\epsilon_v}$.
\begin{thmx}[Corollary \ref{cor:thetacuspidal}, Theorem \ref{cor:non-vanishing}]\label{thmB}
Assume that
\begin{enumerate}
    \item $L(\frac12,\chi)$ is nonzero,
    \item $K/F$ is unramified at all places of $F$ above $2$.
    \item $F$ is totally real and, for all archimedean places $v$ of $F$ where $K_v=\RR\times\RR$, we have $\chi_v(-1) = 1$.
\end{enumerate}
Then $\theta(\sigma)$ is cuspidal and isomorphic to $\bigotimes'_v\pi_v^{\epsilon_v}$.
\end{thmx}

To prove Theorem \ref{thmB}, we start by proving that $\theta(\sigma)$ is cuspidal. We relate the constant terms of $\theta(\sigma)$ to the theta lift for a smaller dual reductive pair $\PU_3\times\GL_2\hookrightarrow\widetilde{M}$, where $\widetilde{M}$ is the Levi subgroup of the Heisenberg parabolic of $\widetilde{G}$. At places of $F$ that split in $K$, this becomes a classical type II dual pair, so the resulting \emph{mini-theta lift} is well-understood by work of M\'{\i}nguez \cite{Min08}.

Next, we turn to showing that $\theta(\sigma)$ is nonzero. We prove that generic Fourier coefficients of $\theta(\sigma)$ with respect to the Heisenberg parabolic of $G$ vanish if and only if certain \emph{torus periods} of $\sigma$ vanish. Because $\sigma$ can be described as a classical theta lift from $\U_1$ to $\U_3$, we can use a seesaw argument to express this torus period\footnote{A generalization of this torus period problem, both locally and globally, is studied in work of Borade--Franzel--Girsch--Yao--Yu--Zelingher \cite{groupA}.} as a classical theta lift between two other $1$-dimensional unitary groups. By work of T. Yang \cite{Yan97}, the latter is nonzero if and only if the associated local theta lifts are nonzero and a certain $L$-value is nonzero. We find such a torus period by carefully analyzing certain local root numbers, using assumption (1), and appealing to work of Friedberg--Hoffstein \cite{FH95}. Altogether, this shows that $\theta(\sigma)$ is nonzero.

Finally, to prove that $\theta(\sigma)\cong\bigotimes'_v\pi_v^{\epsilon_v}$, we use local-global compatibility to reduce this to Howe duality results for local exceptional theta lifts and identifying $\theta(\sigma_v^{\epsilon_v})\cong\pi_v^{\epsilon_v}$ for all places $v$ of $F$. When $v$ is nonarchimedean, the Howe duality results are due to Gan--Savin \cite{gan2021howe} and Baki\'{c}--Savin \cite{bakic2021howe}, and the identification $\theta(\sigma_v^{\epsilon_v})=\pi_v^{\epsilon_v}$ holds by construction. When $v$ is archimedean, we use results of Huang--Pand\v{z}i\'{c}--Savin \cite{NDPC}, Loke--Savin \cite{LokeSavin}, and Baki\'{c}--Loke--Savin \cite{BLS25}, but they do not cover all cases; this is why we need assumption (3).

By appropriately varying the unitary group $\PU_3$ and using Theorem \ref{thmB}, we show that the left hand side of \eqref{AMF} contains the right hand side. Our last main result implies the reverse containment, which concludes the proof of Theorem \ref{thmA}.
\begin{thmx}[Proposition \ref{prop:thetabackfull}]\label{thmC}
Assume that
\begin{enumerate}
    \item $L(\textstyle\frac12,\chi)$ is nonzero.
\end{enumerate}
Then the cuspidal part of $\mathcal{A}_\tau(G)$ lies in $\sum_{\PU_3}\theta(\sigma)$, where $\PU_3$ runs over unitary groups associated with $3$-dimensional Hermitian spaces for $K/F$, and $\sigma$ runs over irreducible cuspidal representations in the global Howe--Piatetski-Shapiro $A$-packet.
\end{thmx}
To prove Theorem \ref{thmC}, we study the lift $\theta(\pi)$ to $\PU_3$ of irreducible cuspidal subrepresentations $\pi$ of $\mathcal{A}_\tau(G)$. First, we prove that $\theta(\pi)$ is cuspidal by relating its constant terms to the theta lift for the dual reductive pair $S_3\times G\hookrightarrow S_3\ltimes\Spin_8$ considered by Gan--Gurevich--Jiang \cite{gan2002cubic}. By work of Gan \cite{Gan:Mult_formula_for_cubic}, the latter yields a near equivalence class not equal to $\mathcal{A}_\tau(G)$, and we prove that this implies that $\theta(\pi)$ is cuspidal.

Next, we turn to showing that $\theta(\pi)$ is nonzero for some appropriately chosen $\PU_3$. It suffices to show that certain torus periods of $\theta(\pi)$ are nonzero. We use a seesaw argument to express these torus periods as an integral of $\pi$ with $\theta_E(\one)$, where $\theta_E$ denotes the theta lift for the dual reductive pair $T_E\times\Spin^E_8\hookrightarrow\widetilde{G}$ considered by Gan--Savin \cite{gan2021twisted}, and $\one$ denotes the constant functions on $T_E$. Then, we prove a \emph{Siegel--Weil formula} for $\theta_E(\one)$, which identifies it with the residue of a certain Eisenstein series on $\Spin^E_8$. By work of A. Segal \cite{segal2017family}, integrating the latter with $\pi$ yields the residue of a twisted standard $L$-function for $\pi$. We use assumption (1) to show that this residue is nonzero. Altogether, this shows that $\theta(\pi)$ is nonzero.

Finally, local-global compatibility indicates that every irreducible subrepresentation of $\theta(\pi)$ is nearly equivalent to $\bigotimes'_v\sigma_v^+$. Therefore the description of the discrete automorphic spectrum $\mathcal{A}_{\disc}(\PU_3)$ given by Rogawski \cite{rogawski1990automorphic, rogawski1992multiplicity} implies that they lie in the global Howe--Piatetski-Shapiro $A$-packet. From this, we deduce that $\pi$ lies in the image of said global $A$-packet under $\theta$.

\subsection*{Outline}
In \S\ref{s:lifts}, we introduce our $G\times\PU_3$ theta lift. In \S\ref{s:PU3}, we gather facts about automorphic representations of $\PU_3$ and compute certain torus periods of them. In \S\ref{sec:AMF_and_HPS}, we state Arthur's multiplicity formula for long root $A$-packets on $G$, specialize to the case of Theorem \ref{thmA}, and define the relevant local $A$-packets. In \S\ref{sec:cuspidality}, we prove that the images of Howe--Piatetski-Shapiro representations under $\theta$ are cuspidal. In \S\ref{sec:non-vanishing}, we prove that they are nonzero, which completes the proof of Theorem \ref{thmB}. In \S\ref{s:liftback}, we prove Theorem \ref{thmC}, and we conclude by putting everything together to prove Theorem \ref{thmA}.

\subsection*{Notation}
For an affine algebraic group $G$ over a number field $F$, write $[G]$ for the quotient space $G(F)\backslash G(\A_F)$. When $G$ is connected reductive, write $\mathcal{A}(G)$ for the space of automorphic forms of $G$ as in \cite[p.~228]{gan2002cubic}; in particular, $\mathcal{A}(G)$ is actually a representation of $G(\A_F)$. By an \emph{automorphic representation of $G(\A_F)$}, we mean a subrepresentation of $\mathcal{A}(G)$. When $G$ is connected semisimple and $?$ lies in $\{\disc,\cusp,\res\}$, write $\mathcal{A}_?(G)$ for the automorphic representation of $G(\A_F)$ given by the smooth vectors of $L^2_?([G])$.

For nonarchimedean local fields $F$, we write $\varpi_F$ for a choice of uniformizer. We normalize class field theory by sending uniformizers to geometric Frobenii.

\subsection*{Acknowledgments}
This work began at the 2022 Arizona Winter School, and we thank the organizers for coordinating this event and providing a wonderful work environment. We are extremely thankful to Wee Teck Gan for suggesting this problem, as well as for comments on an earlier draft. We also thank Wee Teck Gan and Gordan Savin for many helpful conversations about this project; our intellectual debt to them should be clear. Finally, we thank Aaron Landesman for some useful discussions.

The first-named author was partially supported by the Croatian Science Foundation under the project number HRZZ-IP-2022-10-4615, the second-named author was partially supported by UK Research and Innovation grant MR/V021931/1, and the third- and fourth-named authors were partially supported by NSF Grants \#DMS2303195 and \#DGE1745303, respectively.

\section{The $\mathsf{G}_2\times\PU_3$ theta lifts}\label{s:lifts}
In this section, we introduce the main players of our paper: the dual pair $\mathsf{G}_2\times\PU_3$ inside the quasi-split adjoint form of $\mathsf{E}_6$, as well as the associated (exceptional) theta lift. We set up the necessary notation and recall the basic facts we will need. Nothing in this section is new.

\subsection{Freudenthal--Jordan algebras}

Let $F$ be a field of characteristic 0, and let $(J,\circ)$ be a Freudenthal algebra over $F$ in the sense of \cite[\S37.C]{KMRT}. Write $N_J:J\rightarrow F$ for the norm form, $(-,-,-):J^3\rightarrow F$ for its associated symmetric trilinear form, and $T_J:J\rightarrow F$ for the trace form. We use the non-degenerate bilinear form $(x,y)\mapsto T_J(x\circ y)$ to identify $J\cong J^*$. For any $x$ in $J$, write $x^\#$ in $J$ for its adjoint in the sense of \cite[\S38]{KMRT}.

Let $K$ be a quadratic \'etale $F$-algebra.
\begin{example}\label{exmp:J}
Let $B$ be a $9$-dimensional central simple algebra over $K$, and let $\iota:B\rightarrow B$ be an involution of the second kind for $K/F$. Then the subspace $J\subseteq B$ of $\iota$-fixed points is naturally a Freudenthal algebra over $F$ \cite[p.~513]{KMRT}. Write $G'_J\coloneqq\underline{\Aut}(J)^\circ$ for the connected automorphism group of $J$ over $F$.

When $B$ is isomorphic to $\mathrm{M}_3(K)$, our $\iota$ is of the form $x\mapsto e\prescript{t}{}{\overline{x}}e^{-1}$ for some invertible Hermitian matrix $e$ in $\mathrm{M}_3(K)$. Note that $x\mapsto xe$ identifies $J$ with the subspace $\mathfrak{h}_3\subseteq\mathrm{M}_3(K)$ of Hermitian matrices, and under this identification,
\begin{itemize}
    \item the identity in $J$ is $e$,
    \item $x\circ y$ is given by $\frac12(xe^{-1}y+ye^{-1}x)$,
    \item $N_J(x)$ is given by $\det(xe^{-1})$,
    \item $T_J(x)$ is given by $\tr(xe^{-1})$,
    \item $x^\#$ equals $\det(e^{-1})e\adj(x)e$, where $\adj(x)$ denotes the adjugate matrix of $x$.
\end{itemize}
Moreover, $G'_J$ corresponds to the projective unitary group $\PU_3$ over $F$ associated with the Hermitian space induced by $e$, where $g$ in $\PU_3$ acts on $\mathfrak{h}_3$ via $x\mapsto gx\prescript{t}{}{\overline{g}}$.

Without changing the Hermitian space induced by $e$, we may assume that $e$ lies in $\mathrm{M}_3(F)$. Then $\prescript{t}{}{e}=e$, and $x\mapsto\overline{x}$ induces automorphisms $c$ of $G'_J$ and $J$ over $F$ such that the full automorphism group $\underline{\Aut}(J)$ of $J$ over $F$ is given by $G'_J\rtimes c^{\Z/2}$ \cite[p.~515]{KMRT}.
\end{example}

\subsection{A quasi-split adjoint form of $\mathsf{E}_6$}\label{ss:E6} Using $J$, one can construct a Lie algebra $\widetilde\g_J$ over $F$ as follows. As an $F$-vector space, set $\widetilde\g_J \coloneqq (\mathfrak{sl}_3 \oplus \mathfrak{l}_J) \oplus (V \otimes J) \oplus (V^\ast \otimes J^*)$, where $V$ denotes the standard representation of $\mathfrak{sl}_3$ over $F$, and $\mathfrak{l}_J$ denotes the Lie subalgebra
$$\mathfrak{l}_J \coloneqq \{a \in \End(J) \ | \ (a x, y, z) + (x, ay, z) + (x, y, az) = 0 \text{ for all }x,y,z \in J \}\subseteq\End(J).$$
For a complete description of the Lie bracket in $\widetilde\g_J$, consult \cite[p.~138]{rumelhart1997minimal}. Recall from \cite[p.~178]{rumelhart1997minimal} that
\begin{itemize}
    \item When $J$ is $1$-dimensional over $F$ (i.e. isomorphic to $F$), $\widetilde\g_J=\widetilde\g_F$ is split of type $\mathsf{G}_2$,
    \item When $J$ is $3$-dimensional over $F$ (i.e. isomorphic to a cubic \'etale $F$-algebra $E$), $\widetilde\g_J=\widetilde\g_E$ is the quasi-split form of $\mathsf{D}_4$ with respect to $E$ over $F$.
\end{itemize}
Henceforth, assume that $J$ is of the form considered at the end of Example \ref{exmp:J}. Then $\widetilde\g\coloneqq\widetilde\g_J$ is the quasi-split form of $\mathsf{E}_6$ with respect to $K$ over $F$ \cite[Proposition 7]{rumelhart1997minimal}. In particular, $\widetilde\g$ depends only on $K$, and its connected automorphism group $\widetilde{G}\coloneqq\underline{\Aut}(\widetilde\g)^\circ$ is the quasi-split adjoint form of $\mathsf{E}_6$ with respect to $K$ over $F$.

\subsection{Dual pairs in $\widetilde{G}$}\label{ss:dualpairs}We will consider the following semisimple subgroups of $\widetilde{G}$. The natural action of $\underline{\Aut}(J)$ on $\widetilde\g$ induces an injective morphism $G'_J\hookrightarrow\widetilde{G}$ of groups over $F$, and the image of the Lie algebra $\g'_J$ of $G'_J$ under this map equals $\{a\in\mathfrak{l}_J\mid a(e)=0\}\subseteq\widetilde\g$.

Write $\g$ for $\widetilde\g_F$, and write $G\coloneqq\underline{\Aut}(\g)$ for its automorphism group over $F$. By \S\ref{ss:E6}, $G$ is the connected split simple group of type $\mathsf{G}_2$ over $F$. Now \cite[Proposition 6.2]{gan2021twisted} shows that the inclusion $\g\subseteq\widetilde\g$ induces an injective morphism $G\hookrightarrow\widetilde{G}$ of groups over $F$.

By looking at the Lie algebras, one checks that the images of $G$ and $G'_J$ in $\widetilde{G}$ are mutual connected centralizers. Since $G$ and $G'_J$ are both adjoint, this yields an injective morphism $G\times G'_J\hookrightarrow\widetilde{G}$ of groups over $F$. This is our main dual pair of interest.

\subsection{Root spaces and simple roots}\label{ss:roots} The inclusion $\mathfrak{sl}_3\subseteq\g$ induces an injective morphism $\SL_3\hookrightarrow G$ of groups over $F$. Write $T$ for the diagonal maximal subtorus of $\SL_3$, whose image in 
$G$ is also a maximal subtorus. For any root $\delta$ of $G$ with respect to $T$, write $\mathfrak{n}_\delta$ (respectively $\widetilde{\mathfrak{n}}_\delta$) for the associated subspace of $\g$ (respectively $\widetilde\g$). Similarly, write $N_\delta$ (respectively $\widetilde{N}_\delta$) for the associated unipotent subgroup of $G$ (respectively $\widetilde{G}$) over $F$, which we identify with $\mathfrak{n}_\delta$ (respectively $\widetilde{\mathfrak{n}}_\delta$) as varieties over $F$ via the exponential morphism. Because $G'_J$ commutes with $G$, we see that conjugation by $G'_J$ preserves $\widetilde{N}_\delta$.

These root spaces are computed as follows \cite[\S6.3]{gan2021twisted}. When $\delta$ is a long root, we have $\mathfrak{n}_\delta = \widetilde{\mathfrak{n}}_\delta=\ang{e_{ij}}$, where $1\leq i\neq j\leq 3$ and $e_{ij}$ denotes the standard basis vector in $\mathfrak{sl}_3$. When $\delta$ is a short root, then $\mathfrak{n}_\delta$ equals either $\ang{e_i\otimes e}$ or $\ang{e_i^*\otimes e^*}$, where $1\leq i\leq 3$ and $e_i$ denotes the standard basis vector in $V$. Moreover, $\widetilde{\mathfrak{n}}_\delta$ equals the corresponding $e_i\otimes J$ or $e_i^*\otimes J^*$. In particular, we identify $\widetilde{\mathfrak{n}}_\delta$ with $J$ when $\delta$ is a short root, and for any root $\delta$ we identify $\mathfrak{n}_\delta$ with $F$.

We fix simple roots $\{\alpha,\beta\}$ of $G$ with respect to $T$ such that $\mathfrak{n}_\alpha=\ang{e_2\otimes e}$ and $\mathfrak{n}_\beta=\ang{e_{12}}$.

\subsection{Heisenberg parabolics}\label{ss:heisenberg}
We describe the Heisenberg parabolic of $\widetilde{G}$ as follows. Write $\mathfrak{n}_3\subseteq\mathfrak{sl}_3$ for the subalgebra of strictly upper-triangular matrices over $F$. Write $\widetilde P$ for the parabolic subgroup of $\widetilde{G}$ that admits a Levi factor $\widetilde{M}$ with Lie algebra equal to $(\mathfrak{t}\oplus\mathfrak{l}_J)\oplus(e_2\otimes J)\oplus(e_2^*\otimes J^*)$ and whose unipotent radical $\widetilde{N}$ has Lie algebra $\widetilde{\mathfrak{n}}$ equal to 
\begin{align*}
\mathfrak{n}_3\oplus(e_1\otimes J)\oplus(e_3^*\otimes J^*) = \widetilde{\mathfrak{n}}_\beta\oplus\widetilde{\mathfrak{n}}_{\alpha+\beta}\oplus\widetilde{\mathfrak{n}}_{2\alpha+\beta}\oplus\widetilde{\mathfrak{n}}_{3\alpha+\beta}\oplus\widetilde{\mathfrak{n}}_{3\alpha+2\beta}
\end{align*}
\cite[\S6.4]{gan2021twisted}. Note that the center $Z$ of $\widetilde{N}$ has Lie algebra $\mathfrak{z}$ equal to $\widetilde{\mathfrak{n}}_{3\alpha+2\beta}$, and the quotient $\widetilde{N}/Z$ is also abelian. Hence the exponential morphism $\widetilde{N}/Z\cong\widetilde{\mathfrak{n}}_\beta\oplus\widetilde{\mathfrak{n}}_{\alpha+\beta}\oplus\widetilde{\mathfrak{n}}_{2\alpha+\beta}\oplus\widetilde{\mathfrak{n}}_{3\alpha+\beta}$ is an isomorphism of groups over $F$. We use the Killing form to identify $(\widetilde{\mathfrak{n}}/\mathfrak{z})^*$ with $\widetilde\X\coloneqq\widetilde{\mathfrak{n}}_{-3\alpha-\beta}\oplus\widetilde{\mathfrak{n}}_{-2\alpha-\beta}\oplus\widetilde{\mathfrak{n}}_{-\alpha-\beta}\oplus\widetilde{\mathfrak{n}}_{-\beta}$.

Under our identifications from \S\ref{ss:roots}, $\widetilde\X$ corresponds to $F\times J\times J\times F$. Note that $\widetilde{M}$ naturally acts on $\widetilde\X$. Write $\O_{\min}$ for the $\widetilde{M}$-orbit of $(0,0,0,1)$ in $\widetilde\X$, which is a locally closed subvariety of $\widetilde\X$ over $F$.

\begin{prop}[{\cite[Proposition 8.1]{gan2021twisted}}]\label{prop:min_orbit}
The orbit $\O_{\min}$ equals the locus of $(a,x,y,d)$ in $\widetilde\X$ such that
\begin{align*}
(a,x,y,d)\neq0,\,x^\#=ay,\,y^\#=dx,\mbox{ and }l(x)\circ l^*(y)=ad\mbox{ for all }l\mbox{ in }L_J(F),
\end{align*}
where $L_J\subseteq\GL_J$ is the subgroup of linear maps that preserve $N_J$, and $l^*$ denotes the dual action.
\end{prop}

We similarly describe the Heisenberg parabolic of $G$ as follows. Write $P$ for $\widetilde{P}\cap G$, which is a parabolic subgroup of $G$ that admits a Levi factor $M\coloneqq\widetilde{M}\cap G$. Note that we have a unique isomorphism $\GL_2\ra^\sim M$ such that the standard coordinates on strictly upper-triangular matrices and $N_\alpha$ agree, which we use to identify $M$ with $\GL_2$. The unipotent radical $N$ of $P$ equals $\widetilde{N}\cap G$ and hence has Lie algebra $\mathfrak{n}$ equal to
\begin{align*}
\mathfrak{n}_3\oplus(\ang{e_1\otimes e})\oplus(\ang{e_3^*\otimes e^*}) = \mathfrak{n}_\beta\oplus\mathfrak{n}_{\alpha+\beta}\oplus\mathfrak{n}_{2\alpha+\beta}\oplus\mathfrak{n}_{3\alpha+\beta}\oplus\mathfrak{n}_{3\alpha+2\beta}.
\end{align*}
Note that the center of $N$ equals $Z$. As above, the exponential morphism  $N/Z\cong\mathfrak{n}_\beta\oplus\mathfrak{n}_{\alpha+\beta}\oplus\mathfrak{n}_{2\alpha+\beta}\oplus\mathfrak{n}_{3\alpha+\beta}$ is an isomorphism of groups over $F$, and the Killing form identifies $(\mathfrak{n}/\mathfrak{z})^*$ with
\begin{align*}
\X\coloneqq\mathfrak{n}_{-3\alpha-\beta}\oplus\mathfrak{n}_{-2\alpha-\beta}\oplus\mathfrak{n}_{-\alpha-\beta}\oplus\mathfrak{n}_{-\beta}.
\end{align*}

Write $p:\widetilde\X\rightarrow\X$ for the map obtained by dualizing $\mathfrak{n}/\mathfrak{z}\hookrightarrow\widetilde{\mathfrak{n}}/\mathfrak{z}$.
\begin{lemma}\label{lem:p}
Under our identifications in \S\ref{ss:roots}, $p$ corresponds to the map
\begin{align*}
\id\times T_J\times T_J\times \id:F\times J\times J\times F\rightarrow F\times F\times F\times F.
\end{align*}
\end{lemma}
\begin{proof}
When $K$ is split, this is \cite[(4.7)]{RealAndGlobal}. The general case follows by Galois descent.
\end{proof}
Finally, note that $\widetilde{P}\cap(G\times G'_J)=P\times G'_J$.

\subsection{Other parabolics}\label{ss:parabolics}

Identify $\mathfrak{gl}_2$ with the subalgebra of $\mathfrak{sl}_3$ consisting of matrices concentrated in the upper-left $2\times2$ and lower-right $1\times1$ entries. Write $\widetilde{Q}$ for the parabolic subgroup of $\widetilde{G}$ that admits a Levi factor $\widetilde{L}$ with Lie algebra equal to $\mathfrak{gl}_2\oplus\mathfrak{l}_J$ and whose unipotent radical $\widetilde{U}$ has Lie algebra
\begin{align*}
(\ang{e_{13},e_{23}})\oplus(\ang{e_1,e_2}\otimes J)\oplus(e_3^*\otimes J^*) = \widetilde{\mathfrak{n}}_{\alpha}\oplus\widetilde{\mathfrak{n}}_{\alpha+\beta}\oplus\widetilde{\mathfrak{n}}_{2\alpha+\beta}\oplus\widetilde{\mathfrak{n}}_{3\alpha+\beta}\oplus\widetilde{\mathfrak{n}}_{3\alpha+2\beta}
\end{align*}
\cite[\S6.5]{gan2021twisted}. Write $Q$ for $\widetilde{Q}\cap G$, which is a parabolic subgroup of $G$ that admits a Levi factor $L\coloneqq\widetilde{L}\cap G$. Note that the injective morphism of groups $\GL_2\hookrightarrow\SL_3$ induced by $\mathfrak{gl}_2\subseteq\mathfrak{sl}_3$ yields an isomorphism $\GL_2\ra^\sim L$. The unipotent radical $U$ of $Q$ equals $\widetilde{U}\cap G$ and hence has Lie algebra
\begin{align*}
(\ang{e_{13},e_{23}})\oplus(\ang{e_1,e_2}\otimes e)\oplus(\ang{e_3^*\otimes e^*}) = \mathfrak{n}_\alpha\oplus\mathfrak{n}_{\alpha+\beta}\oplus\mathfrak{n}_{2\alpha+\beta}\oplus\mathfrak{n}_{3\alpha+\beta}\oplus\mathfrak{n}_{3\alpha+2\beta}.
\end{align*}
Note that $\widetilde{Q}\cap (G\times G'_J)=Q\times G'_J$.

Write $B$ for $P\cap Q$, which is a Borel subgroup of $G$ containing $T$. Write $V$ for the unipotent radical of $B$.

\subsection{Local theta lifts}\label{ss:localtheta}
In this subsection, assume that $F$ is a local field. Write $\Omega$ for the minimal representation of $\widetilde{G}(F)$ in the sense of \cite[Definition 3.6]{gan2005minimal} or \cite[Definition 4.6]{gan2005minimal}, which is an irreducible smooth representation of $\widetilde{G}(F)$. When $F$ is nonarchimedean and $K/F$ is unramified, note that $\widetilde{G}$ and hence $\Omega$ is unramified \cite[Corollary 7.4]{gan2005minimal}.

Let $\psi:F\rightarrow S^1$ be a nontrivial unitary character. For any $X$ in $\widetilde\X$, write $\psi_X:\widetilde{N}(F)\rightarrow S^1$ for the unitary character given by $\widetilde{n}\mapsto\psi(\langle X,\widetilde{n}\rangle)$, where $\langle-,-\rangle$ denotes the Killing form.
\begin{prop}\label{prop:minorbit}
Assume that $F$ is nonarchimedean. Then there exists a natural $\widetilde{M}(F)$-equivariant injection $\Omega_{Z(F)}\hookrightarrow C^\infty(\O_{\min}(F))$ such that
\begin{enumerate}
    \item the image of $\Omega_{Z(F)}$ contains $C^\infty_c(\O_{\min}(F))$,
    \item for all $X$ in $\O_{\min}(F)$, postcomposing with evaluation at $X$ induces the unique nonzero $\C$-linear map $\Omega_{\widetilde{N}(F),\psi_X}\rightarrow\C$ up to scaling,
    \item when $K/F$ is unramified, the image of any unramified vector in $\Omega$ is supported on $\O_{\min}(F)\cap\widetilde\X(\O)$,
    \item for nonzero $X$ in $\widetilde\X(F)-\O_{\min}(F)$, we have $\Omega_{N(F),\psi_X}=0$.
\end{enumerate}
\end{prop}
\begin{proof}
Note that $\O_{\min}(F)\cap\widetilde\X(\O_F)=\bigcup\limits_{n=0}^\infty\varpi_F^n\cdot\O_{\min}(\O_F)$, so parts (1)--(3) follow from \cite[p.~2057-2058]{gan2021twisted}, except for the uniqueness in part (2). Part (4) and the uniqueness in part (2) are \cite[Proposition 11.5]{gan2005minimal}.
\end{proof}

We use $\Omega$ to define our local theta lift.
\begin{definition}\label{defn:localtheta}
Let $\sigma$ be an irreducible smooth representation of $G'_J(F)$. The maximal $\sigma$-isotypic quotient of $\Omega|_{G'_J(F)}$ can be written as $\sigma\otimes\Theta(\sigma)$ for some $\C$-vector space $\Theta(\sigma)$, and since $G'_J$ commutes with $G$, our $\Theta(\sigma)$ is naturally a smooth representation of $G(F)$. Write $\theta(\sigma)$ for the maximal semisimple unitarizable quotient of $\Theta(\sigma)$.

One can swap the roles of $G$ and $G'_J$: for any irreducible smooth representation $\pi$ of $G(F)$, write $\Theta(\pi)$ and $\theta(\pi)$ for the analogous smooth representations of $G'_J(F)$.
\end{definition}

\subsection{Global theta lifts}\label{ss:globaltheta}
In this subsection, assume that $F$ is a number field. For every place $v$ of $F$, write $\Omega_v$ for the minimal representation of $\widetilde{G}(F_v)$ as in \S\ref{ss:localtheta}, and write $\Omega$ for the restricted tensor product $\bigotimes'_v\Omega_v$. There exists a canonical $\widetilde{G}(\A_F)$-equivariant embedding $\theta:\Omega\hookrightarrow \mathcal{A}_{\disc}(\widetilde{G})$ \cite[\S14.3]{gan2021twisted}.

Let $\psi:F\backslash\A_F\rightarrow S^1$ be a nontrivial unitary character. For any $X$ in $\widetilde\X$, write $\psi_X:[\widetilde{N}/Z]\rightarrow S^1$ for the unitary character given by $\widetilde{n}\mapsto\psi(\langle X,\widetilde{n}\rangle)$.
\begin{prop}\label{prop:FE_of_min}
For any $\varphi$ in $\Omega$, we have $\theta(\varphi)_Z=\theta(\varphi)_{\widetilde{N}}+\displaystyle\sum_{X\in\O_{\min}(F)}\theta(\varphi)_{\widetilde{N},\psi_X}$.
\end{prop}
\begin{proof}
This follows from applying Proposition \ref{prop:minorbit}.(4) to any nonarchimedean place of $F$.
\end{proof}

Note that $\varphi\mapsto\theta(\varphi)_{\widetilde{N},\psi_{X}}(1)$ yields an element $\ell$ of $\Hom_{\widetilde{N}(\A_F)}(\Omega,\psi_{X})$.
\begin{lemma}\label{lem:stabilizerinvariance}
If $\Stab_{G'_J}X$ satisfies weak approximation, then $(\Stab_{G'_J}X)(\A_F)$ acts trivially on $\ell$.
\end{lemma}
\begin{proof}
Proposition \ref{prop:minorbit}.(2) and Proposition \ref{prop:minorbit}.(4) show that $\Hom_{\widetilde{N}(F_v)}(\Omega_v,\psi_{v,X})$ is at most $1$-dimensional for every nonarchimedean place $v$ of $F$, and the action of $(\Stab_{\widetilde{M}}X)(F_v)$ on $\Hom_{\widetilde{N}(F_v)}(\Omega_v,\psi_{v,X})$ factors through an algebraic character of $\widetilde{M}$ over $F_v$ \cite[Proposition 11.7.(ii)]{gan2005minimal}. Since $G'_J$ has no nontrivial algebraic characters, this implies that $(\Stab_{G'_J}X)(F_v)$ acts trivially on $\Hom_{\widetilde{N}(F_v)}(\Omega_v,\psi_{v,X})$.

For any $g'$ in $(\Stab_{G'_J}X)(\A_F)$, note that $g'\cdot\ell$ equals the linear functional $\varphi\mapsto\theta(\varphi)_{\widetilde{N},\psi_{X}}(g'^{-1})$. So fixing $\varphi$ and varying $g'$ yields a continuous function $\ell_\varphi:(\Stab_{G'_J}X)(\A_F)\ra\mathbb{C}$. The above shows that $\ell_\varphi$ is invariant under $(\Stab_{G'_J})(\A_F^\infty)$, and we see that $\ell_\varphi$ is also invariant under $(\Stab_{G'_J}X)(F)$. Since $\Stab_{G'_J}X$ satisfies weak approximation, this implies that $\ell_\varphi$ is constant, which yields the desired result.
\end{proof}

We now define global theta lifts. For any $\varphi$ in $\Omega$ and $f$ in $\mathcal{A}_{\cusp}(G'_J)$, consider the function 
\begin{align*}
    \theta(\varphi,f):[G]\rightarrow\C\mbox{ given by }g\mapsto \int_{[G'_J]} \theta(\varphi)(gg') \overline{f(g')}\dd{g'}.
\end{align*}
The above integral converges because $f$ is rapidly decreasing and $\theta(\varphi)$ is of moderate growth. Since $\theta(\varphi)$ and $f$ are automorphic forms, we see that $\theta(\varphi,f)$ is also an automorphic form.
\begin{definition}\label{defn:globaltheta}
Let $\sigma$ be an irreducible cuspidal automorphic representation of $G'_J(\A_F)$. Write $\theta(\sigma)$ for the space of functions on $[G]$ spanned by $\theta(\varphi,f)$ as $\varphi$ runs over $\Omega$ and $f$ runs over $\sigma$. Note that $\theta(\sigma)$ is an automorphic representation of $G(\A_F)$.

One can swap the roles of $G$ and $G'_J$: for any irreducible cuspidal automorphic representation $\pi$ of $G(\A_F)$, write $\theta(\pi)$ for the analogous automorphic representation of $G'_J(\A_F)$.
\end{definition}

We always have the following general compatibility between local and global theta lifts.
\begin{prop}\label{prop:local_global_comp} Let $\pi$ be an irreducible smooth quotient of $\theta(\sigma)$. Then for every place $v$ of $F$, the local component $\pi_v$ is a quotient of $\theta(\sigma_v)$ as smooth representations of $G(F_v)$.
\end{prop}
\begin{proof}
Note that $\varphi\otimes f\mapsto\theta(\varphi,f)$ yields a $(G\times G_J')(\A_F)$-equivariant surjection $\Omega\otimes\overline{\sigma}\twoheadrightarrow\theta(\sigma)$, where $G'_J(\A_F)$ acts diagonally on $\Omega\otimes\overline{\sigma}$ and trivially on $\theta(\sigma)$. The Petersson inner product identifies $\overline{\sigma}$ with $\sigma^\vee$, so altogether we get a $(G\times G'_J)(\A_F)$-equivariant surjection $\Omega\otimes\sigma^\vee\twoheadrightarrow\pi$. This induces a $(G\times G'_J)(\A_F)$-equivariant surjection $\Omega\twoheadrightarrow\sigma\otimes\pi$ and hence a nonzero element of
\begin{align*}
\Hom_{(G\times G'_J)(\A_F)}(\Omega,\sigma\otimes\pi) = \textstyle\bigotimes_v\Hom_{(G\times G'_J)(F_v)}(\Omega_v,\sigma_v\otimes\pi_v).
\end{align*}
In particular, $\sigma_v\otimes\pi_v$ is a $\sigma_v$-isotypic quotient of $\Omega_v|_{G'_J(F_v)}$. Therefore $\pi_v$ is a quotient of $\Theta(\sigma_v)$, and because $\pi_v$ is irreducible, this factors through $\theta(\sigma_v)$, as desired.
\end{proof}

\section{Automorphic representations of $\PU_3$}\label{s:PU3}
In this section, we provide various results on $\mathcal{A}_{\disc}(\PU_3)$ to feed into our exceptional theta lift. While a complete description of $\mathcal{A}_{\disc}(\PU_3)$ is given by Rogawski \cite{rogawski1990automorphic, rogawski1992multiplicity}, he describes Howe--Piatetski-Shapiro $A$-packets in terms of \emph{endoscopy}. To compute certain torus periods of Howe--Piatetski-Shapiro $A$-packets, we instead need the description originally suggested by Howe--Piatetski-Shapiro \cite[p.~315]{HPS79} in terms of \emph{unitary group theta lifts}. When $\PU_3$ is quasi-split, this was carried out by Gelbart--Rogawski \cite{GR91}. In the literature, we could not find the general case that we need, so we provide a proof here.

Using this description of Howe--Piatetski-Shapiro $A$-packets, we prove a criterion for certain torus periods of them to not vanish. For a generalization of this result, see work of Borade--Franzel--Girsch--Yao--Yu--Zelingher \cite{groupA}.

\subsection{The relation with $\U_3$}\label{ss:U3}
Let $e$ in $\mathrm{M}_n(K)$ be an invertible Hermitian matrix, and write $\U_n$ for the unitary group over $F$ associated with the Hermitian space induced by $e$. Identify the center of $\U_n$ with the norm-$1$ torus $\R^1_{K/F}\mathbb{G}_m$ via the diagonal morphism, and write $\PU_n$ for the quotient $\U_n/\R^1_{K/F}\G_m$.
\begin{lemma}\label{lem:oddunitary}
Assume that $n$ is odd. Then the map $\U_n(F)\rightarrow\PU_n(F)$ is surjective.
\end{lemma}
\begin{proof}
When restricted to $\R^1_{K/F}\G_m\subseteq\U_n$, the determinant morphism $\U_n\rightarrow\R^1_{K/F}\mathbb{G}_m$ equals the $n$-th power morphism. Since $n$ is odd and $H^1(F,\R^1_{K/F}\G_m)\cong F^\times/\Nm_{K/F}(K^\times)$ is $2$-torsion, the composition
\begin{align*}
H^1(F,\R^1_{K/F}\G_m)\rightarrow H^1(F,\U_n)\rightarrow H^1(F,\R^1_{K/F}\G_m)
\end{align*}
equals the identity, so $H^1(F,\R^1_{K/F}\G_m)\rightarrow H^1(F,\U_n)$ is injective. From here, the exactness of
\begin{align*}
\U_n(F)\ra\PU_n(F)\ra H^1(F,\R^1_{K/F}\G_m)\rightarrow H^1(F,\U_n)
\end{align*}
yields the desired result.
\end{proof}
Henceforth, assume that $n$ is odd. When $F$ is a local field, Lemma \ref{lem:oddunitary} lets us identify smooth representations of $\PU_n(F)$ with smooth representations of $\U_n(F)$ on which the center acts trivially. Similarly, when $F$ is a number field, Lang's lemma combined with Lemma \ref{lem:oddunitary} lets us identify automorphic representations of $\PU_n(\A_F)$ with automorphic representations of $\U_n(\A_F)$ on which the center acts trivially.

\subsection{Subgroups of $\PU_3$}\label{ss:subtori}
Henceforth, assume that $n=3$. When the Hermitian space induced by $e$ is isotropic, we may assume that $e$ is anti-diagonal. Write $B'$ for the upper-triangular Borel subgroup of $\U_3$ over $F$, and write $T'$ for the diagonal maximal subtorus of $\U_3$ over $F$. Identify $(\R_{K/F}\mathbb{G}_m)\times(\R_{K/F}^1\mathbb{G}_m)$ with $T'$ via the morphism $(\alpha,\beta)\mapsto\mathrm{diag}(\alpha,\beta,\overline{\alpha}^{-1})$.

Next, we describe some more maximal subtori of $\U_3$ over $F$. Let $E$ be a cubic \'etale $F$-algebra, and write $L$ for $E\otimes K$. Let $\lambda$ be in $E^\times/\Nm_{L/E}(L^\times)$, and consider the $1$-dimensional Hermitian space $L_\lambda$ for $L/E$ induced by $\lambda$. By postcomposing the Hermitian pairing with $\tr_{L/K}$, we view $L_\lambda$ as a $3$-dimensional Hermitian space for $K/F$.

Assume that $L_\lambda$ is isomorphic to the Hermitian space induced by $e$. Fix such an isomorphism, which induces an embedding $L\hookrightarrow\mathrm{M}_3(K)$ of $K$-algebras with involution (equivalently, an embedding $i:E\hookrightarrow J$ of $F$-algebras). Write $T_E$ for the torus $\coker(\R^1_{K/F}\mathbb{G}_m\ra\R_{E/F}(\R^1_{L/E}\mathbb{G}_m))$, and note that any such $i:E\hookrightarrow J$ induces an injective morphism $i:T_E\hookrightarrow\PU_3$ of groups over $F$.

We interpret the above construction in terms of Galois cohomology as follows.
\begin{lemma}\label{lem:embeddingsGaloiscohom}
The disjoint union
\begin{align*}
\coprod_{(B,\iota)}G'_J(F)\backslash\{F\mbox{-algebra embeddings }i:E\hookrightarrow J\},
\end{align*}
where $(B,\iota)$ runs over $9$-dimensional central simple algebras over $K$ equipped with an involution of the second kind for $K/F$ and $J=B^{\iota=1}$, is naturally in bijection with $H^1(F,T_E)$. Under this identification, the map
\begin{align*}
E^\times/\Nm_{L/E}(L^\times)\cong H^1(E,\R^1_{L/E}\G_m)\rightarrow H^1(F,T_E)
\end{align*}
corresponds to the assignment $\lambda\mapsto(i:E\hookrightarrow J)$ described above, and the boundary map
\begin{align*}
H^1(F,T_E)\rightarrow H^2(F,\R^1_{K/F}\G_m)\cong\ker(\Nm_{K/F}:H^2(K,\G_m)\rightarrow H^2(F,\G_m))
\end{align*}
corresponds to the assignment $(i:E\hookrightarrow B^{\iota=1})\mapsto B$.
\end{lemma}
\begin{proof}
Consider the torus
\begin{align*}
\ker(\Nm_{K/F}:(\R_{L/F}\G_m)/(\R_{K/F}\G_m)\rightarrow(\R_{E/F}\G_m)/\G_m)
\end{align*}
over $F$. The natural inclusion $\R_{E/F}(\R^1_{L/E}\G_m)\subseteq\R_{L/F}\G_m$ induces an isomorphism from $T_E$ to the above torus, so \cite[Lemma 4.5]{gan2021twisted} shows that $H^1(F,T_E)$ is naturally in bijection with $L$-isomorphism classes of rank-$2$ $E$-twisted composition algebras with quadratic invariant $K$. The Springer construction \cite[Theorem (38.6)]{KMRT} shows that the latter are naturally in bijection with the above disjoint union, as desired. Finally, tracing through these identifications yields the last sentence.
\end{proof}

\subsection{The $\U_1\times\U_3$ theta lifts}\label{ss:unitarytheta}
In this subsection, assume that $F$ is a local field or a number field. Write
\begin{align*}
C_F\coloneqq\begin{cases}
F^\times & \mbox{if }F\mbox{ is local,}\\
F^\times\backslash\A_F^\times & \mbox{if }F\mbox{ is a number field.}
\end{cases}
\end{align*}
Let $\chi:C_K\ra S^1$ be a \emph{conjugate-symplectic} unitary character, i.e. $\chi|_{C_F}$ equals the character $\omega:C_F\ra\{\pm1\}$ associated with $K/F$. In particular, $\chi(\overline{k})=\chi(k)^{-1}$ for all $k$ in $C_K$. Let $\epsilon_1$ be in $F^\times/\Nm_{K/F}(K^\times)$, and write $\U_1$ for the unitary group over $F$ associated with the Hermitian space induced by $\epsilon_1$. Let $\delta$ in $K^\times$ satisfy $\tr_{K/F}\delta=0$. 

Consider the $F$-algebra
\begin{align*}
R_F\coloneqq\begin{cases}
F & \mbox{if }F\mbox{ is local,} \\
\A_F & \mbox{if }F\mbox{ is a number field.}
\end{cases}
\end{align*}
For any irreducible smooth representation $\rho$ of $\U_1(R_F)$ (cuspidal automorphic when $F$ is a number field), write $\theta(\rho)$ for the theta lift of $\rho$ with respect to
\begin{itemize}
    \item the symplectic space $W$ over $F$ associated with $\delta$ and the Hermitian spaces induced by $\epsilon_1$ and $e$,
    \item the Weil representation of $\mathrm{Mp}_W(R_F)$ associated with $\psi$,
    \item the lifting of $(\U_1\times\U_3)(R_F)\rightarrow\mathrm{Sp}_W(R_F)$ to $\mathrm{Mp}_W(R_F)$ associated with $\psi$ and $(\chi,\chi^3)$ \cite[Theorem 3.1]{Kud87}.
\end{itemize}
For unitary group theta lifts, we have the following strengthening of Proposition \ref{prop:local_global_comp}.
\begin{prop}\label{prop:unitarythetalocalglobal}
Assume that $F$ is a number field, and assume that $K$ is a field. If $\theta(\rho_v)$ is nonzero for every place $v$ of $F$ such that $\U_{3,F_v}$ is anisotropic, then $\theta(\rho)$ is an irreducible discrete automorphic representation of $\U_3(\A_F)$. Moreover, $\theta(\rho)_v$ is isomorphic to $\theta(\rho_v)$.
\end{prop}
\begin{proof}
If we can prove that $\theta(\rho)$ is nontrivial discrete, then the desired results would follow \cite[Proposition 1.2]{GRS93} from local Howe duality \cite[Theorem 1.1(iii)]{GT16}, \cite[Theorem 1]{How89}. So let us focus on proving that $\theta(\rho)$ is nontrivial discrete.

For quasi-split $\U_3$, this is \cite[Proposition 3.4.1]{GR91}. 
For anisotropic $\U_3$, our $\theta(\rho)$ is automatically cuspidal, and we will use the Rallis inner product formula as proved by Yamana \cite{Yam14} to show non-vanishing. More precisely, by \cite[Corollary 10.1(1)]{Yam14}, \cite[Lemma 10.1(2)]{Yam14}, and \cite[(9.1)]{Yam14}, our $\theta(\rho)$ is nonzero if and only if, for every place $v$ of $F$, the corresponding local zeta integral at $v$ (as in \cite[p.~671]{Yam14}) is nonzero.

When $\U_{3,F_v}$ is quasi-split, we can globalize $\U_{3,F_v}$ and $\rho_v$ to a quasi-split $\U_3$ as in the proof of Proposition \ref{prop:HPStheta}, and applying the above discussion to this globalization shows that the local zeta integral at $v$ is nonzero. When $\U_{3,F_v}$ is anisotropic, the local zeta integral at $v$ is taken over the compact group $S^1\times S^1$. Hence the local seesaw identity used to prove \cite[Lemma 8.6.(1)]{Yam14} also proves its converse, i.e. the non-vanishing of $\theta(\rho_v)$ implies that the local zeta integral at $v$ is nonzero. This concludes the desired result.
\end{proof}
\begin{remark}
When $v$ is nonarchimedean, $\U_{3,F_v}$ is quasi-split. Hence the assumption in Proposition \ref{prop:unitarythetalocalglobal} that $\theta(\rho_v)$ is non-zero only needs to be checked when $v$ is archimedean. In our cases of interest, we give explicit conditions for this in Proposition \ref{prop:archHPStheta} below.
\end{remark}

Write $\epsilon_3$ in $F^\times/\Nm_{K/F}(K^\times)$ for the discriminant of the $3$-dimensional Hermitian space induced by $e$. When $F$ is local, we identify $F^\times/\Nm_{K/F}(K^\times)$ with its image under $\omega$.

\subsection{Local $A$-packets over $p$-adic fields}\label{ss:U3HPSpadic} In this subsection, assume that $F$ is a nonarchimedean local field. Then the Hermitian space induced by $e$ is isotropic. Consider the character $\varrho':T'(F)\rightarrow\C^\times$ given by
\begin{align*}
(\alpha,\beta)\mapsto\chi(\alpha)\chi(\beta)^{-1}\lvert\alpha\rvert^{1/2},
\end{align*}
and write $\sigma^+$ for the unique irreducible quotient of the normalized parabolic induction $i_{B'}^{\U_3}\varrho'$. Since $\varrho'$ is trivial on the center of $\U_3(F)$, we see that $i^{\U_3}_{B'}\varrho'$ and hence $\sigma^+$ is too, so $\sigma^+$ descends to an irreducible smooth representation of $\PU_3(F)$. Note that $\sigma^+$ is not tempered.

When $K$ is a field, write $\xi:(\U_2\times\U_1)(F)\ra S^1$ for the unitary character given by $(h_2,h_1)\mapsto\chi(h_1)^{-1}$, and write $\sigma^-$ for the irreducible cuspidal representation of $\U_3(F)$ defined as $\pi^s(\xi)$ in \cite[Proposition 13.13.(d)]{rogawski1990automorphic}.

We describe $\sigma^\pm$ in terms of theta lifts as follows.

\begin{prop}\label{prop:HPStheta}
The theta lift $\theta(\one)$ of the trivial representation $\one$ is isomorphic to $\sigma^\epsilon$, where $\epsilon$ in $\{\pm1\}$ is
\begin{align*}
\epsilon_1\cdot\epsilon_3\cdot\epsilon(\textstyle\frac12,\chi^3,\psi(\tr_{K/F}(\delta-))).
\end{align*}
\end{prop}
\begin{proof}
When $K$ is split, then $\epsilon=+1$, and this follows from \cite[Th\'eor\`eme 1]{Min08}. So assume that $K$ is a field. Using Krasner's lemma, one can construct a quadratic extension $\mathbf{K}/\mathbf{F}$ of number fields with a place $v$ of $\mathbf{F}$ such that $\mathbf{K}_v/\mathbf{F}_v$ is identified with $K/F$. By the Hasse principle for Hermitian spaces, there exists $\boldsymbol{\epsilon}_1$ in $\mathbf{F}^\times/\Nm_{\mathbf{K}/\mathbf{F}}(\mathbf{K}^\times)$ and an invertible Hermitian matrix $\mathbf{e}$ in $\mathrm{M}_3(\mathbf{K})$ such that the resulting Hermitian spaces for $\mathbf{K}/\mathbf{F}$ are isomorphic at $v$ to the ones induced by $\epsilon_1$ and $e$, respectively, and that the Hermitian space induced by $\mathbf{e}$ is isotropic.

Let $\boldsymbol{\chi}:\mathbf{K}^\times\backslash\A_{\mathbf{K}}^\times\ra S^1$ be as in Lemma \ref{lem:globalizecharacter} below, let $\boldsymbol\delta$ in $\mathbf{K}^\times$ satisfy $\tr_{\mathbf{K}/\mathbf{F}}(\boldsymbol{\delta})=0$, and let $\boldsymbol\psi:\mathbf{F}\backslash\A_{\mathbf{F}}\ra S^1$ be a nontrivial unitary character. For now, assume that $\delta$ is the image of $\boldsymbol\delta$ in $K$ and that $\psi$ is $\boldsymbol\psi_v$. Since $\U_1(F)$ is compact, the injective homomorphism $\U_1(F)\hookrightarrow\U_1(\mathbf{F})\backslash\U_1(\A_{\mathbf{F}})$ is a closed embedding, so Pontryagin duality yields a unitary character $\boldsymbol\rho:\U_1(\mathbf{F})\backslash\U_1(\A_{\mathbf{F}})\ra S^1$ such that $\boldsymbol\rho_v$ is trivial. 

Consider the global theta lift $\theta(\boldsymbol{\rho})$ as in \S\ref{ss:unitarytheta} with respect to our globalization. Its local component $\theta(\boldsymbol{\rho})_v$ is isomorphic to $\theta(\boldsymbol\rho_v)=\theta(\one)$ by Proposition \ref{prop:unitarythetalocalglobal}, but this local component is also isomorphic to $\sigma^\epsilon$ for some $\epsilon$ in $\{\pm1\}$ by \cite[Theorem 3.4(a)]{GR91}. Therefore $\theta(\one)\cong\sigma^\epsilon$.

To determine $\epsilon$, note that it is equivalent to determine whether $\theta(\one)$ is cuspidal. Now $\theta(\one)$ is cuspidal if and only if the local theta lift of $\one$ to the unitary group associated with the $1$-dimensional Hermitian space for $K/F$ induced by $\epsilon_3$ is zero \cite[(ch.3,IV,4)]{MVW87}, and the latter occurs if and only if $\epsilon_1\cdot\epsilon_3\cdot\epsilon(\textstyle\frac12,\chi^3,\psi(\tr_{K/F}(\delta-)))=-1$ \cite[Prop. 3.4]{rogawski1992multiplicity}. This yields the desired result for this specific $\delta$ and $\psi$.

In general, $\delta$ and $\psi$ are $F$-multiples of ones arising above. Hence the equivariance of the Weil representation of $(\U_1\times\U_3)(F)$ and $\epsilon(\textstyle\frac12,\chi^3,\psi(\tr_{K/F}(\delta-)))$ under scaling implies the result for general $\delta$ and $\psi$.
\end{proof}

\begin{lemma}\label{lem:globalizecharacter}
There exists a conjugate-symplectic unitary character $\boldsymbol{\chi}:\mathbf{K}^\times\backslash\A_{\mathbf{K}}^\times\rightarrow S^1$ satisfying $\boldsymbol\chi_v=\chi$.
\end{lemma}
\begin{proof}
Because $\mathbf{F}^\times\backslash\A_{\mathbf{F}}^\times\hookrightarrow\mathbf{K}^\times\backslash\A_{\mathbf{K}}^\times$ is a closed embedding, Pontryagin duality yields a conjugate-symplectic unitary character $\boldsymbol{\phi}:\mathbf{K}^\times\backslash\A_{\mathbf{K}}^\times\rightarrow S^1$. Therefore $\boldsymbol\phi_v\chi^{-1}$ is trivial on $F^\times\subseteq K^\times$. Now the homomorphism $K^\times/F^\times\ra \mathbf{K}^\times\backslash\A_{\mathbf{K}}^\times/\A_{\mathbf{F}}^\times$ is injective, and since $K^\times/F^\times$ is compact, it is a closed embedding. Pontryagin duality again yields a unitary character $\boldsymbol\gamma:\mathbf{K}^\times\backslash\A_{\mathbf{K}}^\times/\A_{\mathbf{F}}^\times\ra S^1$ such that $\boldsymbol\gamma_v$ equals $\boldsymbol{\phi}_v\chi^{-1}$. Hence $\boldsymbol{\chi}\coloneqq\boldsymbol{\phi\gamma}^{-1}$ satisfies the desired properties.
\end{proof}

\begin{corollary}
When $K$ is a field, $\sigma^-$ has trivial central character. Consequently, $\sigma^-$ descends to an irreducible smooth representation of $\PU_3(F)$.
\end{corollary}

\begin{proof}
Choose $\epsilon_1$ such that $\epsilon_1\cdot\epsilon_3\cdot\epsilon(\textstyle\frac12,\chi^3,\psi(\tr_{K/F}(\delta-)))=-1$. Then Proposition \ref{prop:HPStheta} shows that $\sigma^-$ is isomorphic to $\theta(\one)$. Our choice of lifting characters $(\chi,\chi^3)$ ensures that $\theta(-)$ preserves central characters.
\end{proof}

We also use Proposition \ref{prop:HPStheta} to compute the following torus periods of $\sigma^\pm$.
\begin{prop}\label{prop:HPStorus}
Let $i:E\hookrightarrow J$ be a $\PU_3(F)$-conjugacy class of $F$-algebra embeddings, and let $\epsilon$ be in $F^\times/\Nm_{K/F}(K^\times)$. Suppose that $\Hom_{i(T_E)(F)}(\sigma^\epsilon,\one)$ is nonzero. Then under the identifications of Lemma \ref{lem:embeddingsGaloiscohom}, $i$ corresponds to the image of
\begin{align*}
\lambda=\epsilon(\textstyle\frac12,\chi\circ\Nm_{L/K},\psi(\tr_{L/F}(\delta-)))\in E^\times/\Nm_{L/E}(L^\times)
\end{align*}
under the natural map $E^\times/\Nm_{L/E}(L^\times)\rightarrow H^1(F,T_E)$, and $\epsilon$ equals the image of
\begin{align*}
\Nm_{E/F}(\lambda)\cdot\Delta_{E/F}\cdot[-1]\cdot\epsilon(\textstyle\frac12,\chi^3,\psi(\tr_{K/F}(\delta-)))\in F^\times/\Nm_{K/F}(K^\times),
\end{align*}
where $\Delta_{E/F}$ in $F^\times/(F^\times)^2$ denotes the discriminant of $E/F$, and $[-]$ denotes the map $F^\times\ra F^\times/\Nm_{K/F}(K^\times)$. Moreover, $\Hom_{i(T_E)(F)}(\sigma^\epsilon,\one)$ is $1$-dimensional for this choice of $i$ and $\epsilon$.
\end{prop}

\begin{proof}
By Lemma \ref{lem:embeddingsGaloiscohom} and the exactness of
\begin{align*}
H^1(F,\R^1_{K/F}\G_m)\ra H^1(F,\R_{E/F}\R^1_{L/E}\G_m)\ra H^1(F,T_E)\ra H^2(F,\R^1_{K/F}\G_m),
\end{align*} we see that our $\PU_3(F)$-conjugacy class of $F$-algebra embeddings $i: E\hookrightarrow J$ is induced by some $\lambda$ in $E^\times/\Nm_{L/E}(L^\times)$, and $\lambda$ is unique up to the image of $F^\times/\Nm_{K/F}(K^\times)$ in $E^\times/\Nm_{L/E}(L^\times)$.

Fix such a $\lambda$ for our $i$, and 
choose $\epsilon_1$ such that $\epsilon_1\cdot\epsilon_3\cdot\epsilon(\textstyle\frac12,\chi^3,\psi(\tr_{K/F}(\delta-)))=\epsilon$.
The analogue of \cite[(2.17)]{kudla1984seesaw} for unitary groups yields a seesaw of dual pairs
\begin{center}
    \begin{tikzcd}
     \R_{E/F}\R^1_{L/E}\G_m \ar[rd, no head] \ar[d, no head] & \U_3 \ar[ld, no head] \ar[d, no head]\\
    \U_1 & \R_{E/F}\R^1_{L/E}\G_m
    \end{tikzcd}
\end{center}
in $\mathrm{Sp}_W$ over $F$, where the top-left corresponds to the $1$-dimensional Hermitian space $L_{\epsilon_1}$ for $L/E$ induced by the image of $\epsilon_1$ in $E^\times/\Nm_{L/E}(L^\times)$, the bottom-right corresponds to the $1$-dimensional Hermitian space $L_\lambda$ for $L/E$, and the right morphism recovers $i:T_E\hookrightarrow\PU_3$ after quotienting by $\R^1_{K/F}\G_m$.

Consider the lifting of
\begin{align*}
(\R_{E/F}\R^1_{L/E}\G_m\times\R_{E/F}\R^1_{L/E}\G_m)(F)\rightarrow\mathrm{Sp}_W(F)
\end{align*}
to $\mathrm{Mp}_W(F)$ associated with $\psi\circ\tr_{E/F}$ and $(\chi\circ\Nm_{L/K},\chi\circ\Nm_{L/K})$ \cite[Theorem 3.1]{Kud87}, and for any irreducible smooth representation $\phi$ of $(\R^1_{L/E}\G_m)(E)$, write $\theta(\phi)$ for the resulting theta lift. Since $\Nm_{L/K}(k)=k^3$ for all $k$ in $K^\times$, we see that this lifting is compatible with our aforementioned lifting of $(\U_1\times\U_3)(F)\ra\mathrm{Sp}_W(F)$ to $\mathrm{Mp}_W(F)$. Therefore the above diagram yields the seesaw identity
\begin{align*}
\Hom_{(\R^1_{L/E}\G_m)(E)}(\theta(\rho),\phi) = \Hom_{\U_1(F)}(\theta(\phi),\rho),
\end{align*}
where we can use $\theta$ instead of $\Theta$ since $\rho$ and $\phi$ are cuspidal \cite[(ch.3,IV,4)]{MVW87}.

 By specializing the seesaw identity to $\rho=\one$ and $\phi=\one$, Proposition \ref{prop:HPStheta} yields $\Hom_{(\R^1_{L/E}\G_m)(E)}(\sigma^\epsilon,\one) = \Hom_{\U_1(F)}(\theta(\phi),\one)$. Since this is nonzero, $\theta(\phi)$ is nonzero, which implies that $\lambda$ equals
 \begin{align*}
\epsilon_1\cdot\epsilon(\textstyle\frac12,\chi\circ\Nm_{L/K},\psi(\tr_{L/F}(\delta-)))
 \end{align*}
 \cite[Prop. 3.4]{rogawski1992multiplicity}. Note that the corresponding $i$ also equals the image of $\epsilon(\textstyle\frac12,\chi\circ\Nm_{L/K},\psi(\tr_{L/F}(\delta-)))$ in $H^1(F,T_E)$, as desired.

By choosing an $F$-basis of $E$, one shows that the discriminant of $L_\lambda$ as a $3$-dimensional Hermitian space for $K/F$ equals the image of $\Nm_{E/F}(\lambda)\cdot\Delta_{E/F}\cdot[-1]$ in $F^\times/\Nm_{K/F}(K^\times)$. Hence $\epsilon_3=\Nm_{E/F}(\lambda)\cdot\Delta_{E/F}\cdot[-1]$, so $\epsilon$ has the desired form. Finally, our choice of lifting characters $(\chi\circ\Nm_{L/K},\chi\circ\Nm_{L/K})$ ensures that $\theta(\phi)=\phi$, so $\Hom_{\U_1(F)}(\theta(\phi),\one)$ is indeed $1$-dimensional.
\end{proof}

\subsection{Local $A$-packets over $\RR$ and $\C$}\label{ss:U3HPSarch} In this subsection, assume that $F$ is an archimedean local field. When the Hermitian space induced by $e$ is isotropic, proceeding as in the start of \S\ref{ss:U3HPSpadic} yields an irreducible smooth representation $\sigma^+$ of $\PU_3(F)$, via the Langlands quotient of $\varrho':T'(F)\ra\C^\times$ with respect to $B'$. Note that $\sigma^+$ is not tempered.

When $K$ is a field, note that there exists an odd integer $N$ such that $\chi(z)=(z/\sqrt{z\overline{z}})^N$ for all $z$ in $\C^\times$. For quasi-split $\U_3$, write $\sigma^-$ for the irreducible discrete series representation of $\U_3(F)$ defined as $\pi^s(\xi)$ on \cite[p.~178]{rogawski1990automorphic}, where $\xi:(\U_2\times\U_1)(F)\ra S^1$ denotes the unitary character $(h_2,h_1)\mapsto\chi(h_1)^{-1}$, and $\U_2$ denotes the quasi-split unitary group with respect to $K$ over $F$. Our $\sigma^-$ is a Jordan--H\"older constituent of $i^{\U_3}_{B'}\varrho'$ \cite[p.~176]{rogawski1990automorphic}, so $\sigma^-$ descends to an irreducible discrete series representation of $\PU_3(F)$.

For anisotropic $\U_3$, write $\sigma^-$ for the irreducible discrete series representation of $\U_3(F)$ with highest weight $(\textstyle\frac{N-1}2,\frac{N-1}2,-N+1)$ when $N$ is positive and $(-N-1,\textstyle\frac{N+1}2,\frac{N+1}2)$ when $N$ is negative. Note that $\sigma^-$ descends to an irreducible discrete series representation of $\PU_3(F)$, and recall that $\sigma^-$ is isomorphic to the $F_\varphi$ defined on \cite[p.~243]{rogawski1990automorphic}, where $\varphi$ is associated with $\xi$ as in \cite[p.~177]{rogawski1990automorphic}.

We describe $\sigma^\pm$ in terms of theta lifts as follows.
\begin{prop}\label{prop:archHPStheta}
Write $\epsilon$ for $\epsilon_1\cdot\epsilon_3\cdot\epsilon(\textstyle\frac12,\chi^3,\psi(\tr_{K/F}(\delta-)))$ in $\{\pm1\}$. If $\epsilon=+1$ and $\U_3$ is anisotropic, then the theta lift $\theta(\one)$ of the trivial representation $\one$ is zero. Otherwise, $\theta(\one)$ is isomorphic to $\sigma^\epsilon$.
\end{prop}
\begin{proof}
When $K$ is split, we have $\epsilon=+1$. For $F=\RR$ the result follows from \cite[Proposition III.9]{Moe89}, and for $F=\C$ it follows from \cite[Proposition 1.4.(1)]{AB95}. So assume that $K$ is a field. The end of the proof of Proposition \ref{prop:HPStheta} shows that it suffices to prove the desired statement for a single choice of $\delta$ and $\psi$, so take $\delta=i$ and $\psi(x)=e^{-2\pi ix}$. Then $\epsilon(\textstyle\frac12,\chi^3,\psi(\tr_{K/F}(\delta-)))$ equals the sign of $N$ \cite[Proposition 2.1]{GGP12}.

 For quasi-split $\U_3$, the globalization argument from the proof of Proposition \ref{prop:HPStheta} shows that $\theta(\one)$ is isomorphic to $\sigma^\epsilon$ for some $\epsilon$ in $\{\pm1\}$. To determine $\epsilon$, note that it is equivalent to determine whether $\theta(\one)$ is discrete series. By using \cite[Theorem 4.1]{Ich22} to write $\theta(\one)$ as a cohomological induction and applying the criterion from \cite[p.~58]{VZ84}, we see that $\theta(\one)$ is discrete series if and only if $\epsilon_1\cdot\epsilon_3\cdot\epsilon(\textstyle\frac12,\chi^3,\psi(\tr_{K/F}(\delta-)))=-1$. This yields the desired result.

 For anisotropic $\U_3$, the vanishing result follows from \cite[Theorem 6.1]{Ich22}. The non-vanishing result follows from using \cite[Theorem 4.1]{Ich22} to explicitly compute the infinitesimal character of $\theta(\one)$.
\end{proof}

We use Proposition \ref{prop:archHPStheta} to compute the following torus periods of $\sigma^\pm$.
\begin{prop}\label{prop:archHPStorus}
Let $i:E\hookrightarrow J$ be a $\PU_3(F)$-conjugacy class of $F$-algebra embeddings, and let $\epsilon$ be in $F^\times/\Nm_{K/F}(K^\times)$. If $\U_3$ is anisotropic, assume that $\epsilon=-1$. Suppose that $\Hom_{i(T_E)(F)}(\sigma^\epsilon,\one)$ is nonzero. Then under the identifications of Lemma \ref{lem:embeddingsGaloiscohom}, $i$ corresponds to the image of
\begin{align*}
\lambda=\epsilon(\textstyle\frac12,\chi\circ\Nm_{L/K},\psi(\tr_{L/F}(\delta-)))\in E^\times/\Nm_{L/E}(L^\times)
\end{align*}
under the natural map $E^\times/\Nm_{L/E}(L^\times)\rightarrow H^1(F,T_E)$, and $\epsilon$ equals the image of
\begin{align*}
\Nm_{E/F}(\lambda)\cdot\Delta_{E/F}\cdot[-1]\cdot\epsilon(\textstyle\frac12,\chi^3,\psi(\tr_{K/F}(\delta-)))\in F^\times/\Nm_{K/F}(K^\times).
\end{align*}
Moreover, $\Hom_{i(T_E)(F)}(\sigma^\epsilon,\one)$ is $1$-dimensional for this choice of $i$ and $\epsilon$.
\end{prop}
\begin{proof}
Use the proof of Proposition \ref{prop:HPStorus} with the following modifications:
\begin{itemize}
    \item in the seesaw identity, we can use $\theta$ instead of $\Theta$ since our unitary group theta lifts are either in the stable range \cite[Theorem A]{LM15} or only involve anisotropic groups,
    \item use Proposition \ref{prop:archHPStheta} instead of Proposition \ref{prop:HPStheta},
    \item use \cite[Theorem 6.1]{Ich22} instead of \cite[Prop. 3.4]{rogawski1992multiplicity}.\qedhere
\end{itemize}
\end{proof}

\subsection{Automorphic representations}\label{ss:PU3HPS} In this subsection, assume that $F$ is a number field, and assume that $K$ is a field. For every place $v$ of $F$ and $\epsilon$ in $\{\pm1\}$, write $\sigma^\epsilon_v$ for the irreducible smooth representation of $\PU_3(F_v)$ associated with $\chi_v$ as in \S\ref{ss:U3HPSpadic} or \S\ref{ss:U3HPSarch} whenever it is defined.

Note that $\sigma^+_v$ is defined for cofinitely many $v$. Write $\mathcal{A}_\chi(\PU_3)\subseteq\mathcal{A}_{\disc}(\PU_3)$ for the sum of all irreducible subrepresentations of $\mathcal{A}_{\disc}(\PU_3)$ that are nearly equivalent to $(\sigma^+_v)_v$.
\begin{theorem}\label{thm:globalHPSU3}\hfill
\begin{enumerate}
    \item We have an isomorphism of $\PU_3(\A_F)$-representations
\begin{align*}
\mathcal{A}_\chi(\PU_3) \cong \bigoplus_{(\epsilon_v)_v}\bigotimes_v\!'\,\sigma^{\epsilon_v}_v,
\end{align*}
where $(\epsilon_v)_v$ runs over sequences in $\{\pm1\}$ indexed by places $v$ of $F$ such that
\begin{itemize}
    \item $\sigma^{\epsilon_v}_v$ is defined for all $v$,
    \item $\epsilon_v=+1$ for cofinitely many $v$,
    \item $\prod_v\epsilon_v=\epsilon(\textstyle\frac12,\chi^3)$.
\end{itemize}
 \item Let $\sigma$ be an irreducible discrete automorphic representation of $\PU_3(\A_F)$, and assume that there exists a nonarchimedean place $v$ of $F$ such that
\begin{itemize}
    \item $v$ splits in $K$ (and hence $\PU_{3,F_v}\cong\PGL_3$),
    \item $\sigma_v$ is unramified,
    \item the Satake parameter $S_v$ of $\sigma_v$ has an eigenvalue of norm $\geq q_v^{1/2}$.
\end{itemize}
Then $\sigma$ lies in $\mathcal{A}_{\chi'}(\PU_3)$ for some conjugate-symplectic unitary character $\chi':K^\times\backslash\A_K^\times\ra S^1$.
\item Let $G'_{\mathrm{in}}$ be an inner form of $\PU_3$ over $F$, and assume there exists an irreducible discrete automorphic representation of $G'_{\mathrm{in}}(\A_F)$ satisfying the condition in part (2). Then $G'_{\mathrm{in}}$ is associated with a $3$-dimensional Hermitian space for $K/F$.
\item Let $(\epsilon_v)_v$ be a sequence as in part (1). Then $\bigotimes'_v\sigma_v^{\epsilon_v}$ is not cuspidal if and only if every $\epsilon_v$ equals $+1$ and $L(\textstyle\frac12,\chi^3)$ is nonzero.
\end{enumerate}
\end{theorem}
\begin{proof}
First, we tackle part (1) and part (2). Now $\mathcal{A}_\chi(\PU_3)$ contains $\bigoplus_{(\epsilon_v)_v}\bigotimes_v'\sigma^{\epsilon_v}_v$ for quasi-split $\PU_3$ by \cite[Theorem 1.1]{rogawski1992multiplicity} and for anisotropic $\PU_3$ by \cite[Theorem 1.2]{rogawski1992multiplicity}. Rogawski's description \cite[p.~202]{rogawski1990automorphic} of $\mathcal{A}_{\disc}(\PU_3)$ shows that this containment is an equality, and combined with \cite[Corollary (2.5)]{JS81}, this description also implies part (2). Finally, part (3) follows from part (2) and \cite[Theorem 14.6.3]{rogawski1990automorphic}, and part (4) is \cite[p.~396--397]{rogawski1992multiplicity}.
\end{proof}

Let $i:E\hookrightarrow J$ be a $\PU_3(F)$-conjugacy class of $F$-algebra embeddings. For any irreducible cuspidal automorphic representation $\sigma$ of $\PU_3(\A_F)$, consider the $\C$-linear map $\mathcal{P}_i:\sigma\ra\C$ given by
\begin{align*}
f\mapsto\int_{[i(T_E)]}f(t')\dd{t'},
\end{align*}
which converges because $f$ is rapidly decreasing.

\begin{prop}\label{prop:globalHPStorus}
Let $(\epsilon_v)_v$ be a sequence as in Theorem \ref{thm:globalHPSU3}.(1), and write $\sigma$ for $\bigotimes'_v\sigma_v^{\epsilon_v}$. Assume that $\sigma$ is cuspidal and $\R^1_{L/E}\G_m$ is anisotropic. Then $\mathcal{P}_i$ is nonzero if and only if $L(\textstyle\frac12,\Ind_K^F\chi\otimes\Ind_E^F\one)$ is nonzero and, for every place $v$ of $F$, our $i_v$ and $\epsilon_v$ satisfy the conditions in Proposition \ref{prop:HPStorus} or Proposition \ref{prop:archHPStorus}.
\end{prop}

\begin{proof}
For every place $v$ of $F$, choose $\epsilon_{1,v}$ to be $\epsilon_{3,v}\cdot\epsilon(\textstyle\frac12,\chi_v^3,\psi_v(\tr_{K_v/F_v}(\delta-)))\cdot\epsilon_v$. Our assumptions imply
\begin{align*}
\textstyle\prod_v\epsilon_{1,v} = (\prod_v\epsilon_{3,v})\epsilon(\frac12,\chi^3)(\prod_v\epsilon_v) = +1\cdot\epsilon(\frac12,\chi^3)^2=+1,
\end{align*}
so the Hasse principle yields an $\epsilon_1$ in $F^\times/\Nm_{K/F}(K^\times)$ whose image in $F_v^\times/\Nm_{K_v/F_v}(K^\times_v)$ equals $\epsilon_{1,v}$ for all $v$. Then Proposition \ref{prop:unitarythetalocalglobal} along with Proposition \ref{prop:HPStheta} or Proposition \ref{prop:archHPStheta} show that the global theta lift $\theta(\one)$ as in \S\ref{ss:unitarytheta} with respect to $\epsilon_1$ equals $\sigma$.

Consider the global analogue of the seesaw from the proof of Proposition \ref{prop:HPStorus}, as well as the lifting of
\begin{align*}
(\R_{E/F}\R^1_{L/E}\G_m\times\R_{E/F}\R^1_{L/E}\G_m)(\A_F)\rightarrow\mathrm{Sp}_W(\A_F)
\end{align*}
to $\mathrm{Mp}_W(\A_F)$ associated with $\psi\circ\tr_{E/F}$ and $(\chi\circ\Nm_{L/K},\chi\circ\Nm_{L/K})$ \cite[Theorem 3.1]{Kud87}. For any irreducible cuspidal automorphic representation $\phi$ of $(\R^1_{L/E}\G_m)(\A_E)$, write $\theta(\phi)$ for the resulting theta lift. Note that $\theta(\phi)$ converges because we assume $\R^1_{L/E}\G_m$ is anisotropic.

For the constant functions $1$ on $[\U_1]$ and $[\R^1_{L/E}\G_m]$ and any $\varphi$ in the Weil representation, our seesaw yields
\begin{align*}
\int_{[i(\R^1_{L/E}\G_m)]}\theta(\varphi,1)(t')\dd{t'} = \int_{[\U_1]}\theta(\varphi,1)(u)\dd{u}.
\end{align*}
Write $\rho=\one$, and take $\phi=\one$. By \cite[Theorem 0.4]{Yan97}, our $\theta(\phi)$ is nonzero if and only if $L(\textstyle\frac12,\Ind^F_K\chi\otimes\Ind^F_E\one)$ is nonzero and, for every place $v$ of $F$, our $i_v$ and $\epsilon_v$ satisfy the conditions in Proposition \ref{prop:HPStorus} or Proposition \ref{prop:archHPStorus}. Our choice of lifting characters ensures that $\theta(\phi)=\phi$, and this corresponds to the right hand side of the seesaw identity. Because $\theta(\rho)=\sigma$, the left hand side of the seesaw identity corresponds to $\mathcal{P}_i$, as desired.
\end{proof}
For a generalization of Proposition \ref{prop:globalHPStorus}, see \cite[Theorem 5.3]{groupA}.

\section{Automorphic representations of $\mathsf{G}_2$}\label{sec:AMF_and_HPS}
We begin this section by stating Arthur's multiplicity formula \cite[Conjecture 8.1]{Arthur} in the setting of global long root $A$-parameters $\tau$ for $\mathsf{G}_2$. Next, we specialize to the case where $\tau$ is dihedral, and we state Theorem \ref{thmA}. We then define the relevant local $A$-packets over $p$-adic fields, which are defined by applying our exceptional theta lift to the local $A$-packets for $\PU_3$ considered in \S\ref{ss:U3HPSpadic}; this amounts to work of Alonso--He--Ray--Roset \cite{AHRR23}.
 
Afterwards, we define the relevant local $A$-packets over archimedean fields \emph{without} using theta lifts. We instead formulate a conjecture for how our exceptional theta lift relates them to the local $A$-packets for $\PU_3$ considered in \S\ref{ss:U3HPSarch}. Finally, we give evidence for our conjecture stemming from the Vogan philosophy as well as stemming from global considerations, and we verify our conjecture in some cases.
 
\subsection {Global long root $A$-parameters for $\mathsf{G}_2$}\label{ss:global_long_root} In this subsection, assume that $F$ is a number field. Recall from \S\ref{ss:dualpairs} that $G$ is the connected split group of type $\mathsf{G}_2$ over $F$. Then ${^L}G=\widehat{G}(\C)$, where $\widehat{G}$ is a connected simple group over $\C$ that is also of type $\mathsf{G}_2$.

Fix a maximal subtorus $\widehat{T}$ of $\widehat{G}$ over $\C$. Since there exists a unique pair of orthogonal short and long roots in $\mathsf{G}_2$ up to the Weyl action, we obtain a well-defined morphism $\SL_{2,\operatorname{short}}\times\SL_{2,\operatorname{long}}\ra\widehat{G}$ of groups over $\C$ up to $\widehat{G}(\C)$-conjugation.

Let $\tau$ be an irreducible cuspidal automorphic representation of $\PGL_2(\A_F)$. For all nonarchimedean places $v$ of $F$ where $\tau_v$ is unramified, write $T_v$ for the Satake parameter of $\tau_v$, i.e. the image of geometric Frobenius under the corresponding unramified representation $W_{F_v}\ra\SL_2(\C)$. Consider the semisimple conjugacy class in $\widehat{G}(\C)$ given by the image of $(T_v,\mathrm{diag}(q_v^{-1/2},q_v^{1/2}))$ under $\SL_{2,\operatorname{short}}\times\SL_{2,\operatorname{long}}\ra\widehat{G}$, and write $\pi_v^+$ for the corresponding irreducible unramified representation of $G(F_v)$.

Note that $\pi_v^+$ is defined for cofinitely many $v$. Write $\mathcal{A}_\tau(G)\subseteq\mathcal{A}_{\disc}(G)$ for the sum of all irreducible $G(\A_F)$-subrepresentations $\pi$ of $\mathcal{A}_{\disc}(G)$ that are nearly equivalent to $(\pi^+_v)_v$.

In this setting, Arthur's multiplicity formula \cite[Conjecture 8.1]{Arthur} is equivalent to the following statement. 
\begin{conj}[Arthur's multiplicity formula]\label{conj:AMF_long_root_G2}
For every place $v$, there exists a finite set $\Pi(\tau_v)$ of irreducible smooth representations of $G(F_v)$ depending only on $\tau_v$ such that
\begin{enumerate}
\item $\Pi(\tau_v)$ consists of two elements $\{\pi^+_v,\pi^-_v\}$ when $\tau_v$ is discrete series and one element $\{\pi^+_v\}$ otherwise,
\item For all nonarchimedean places $v$ of $F$ where $\tau_v$ is unramified, $\pi^+_v$ agrees with the irreducible unramified representation of $G(F_v)$ defined above,
\item We have an isomorphism of $G(\A_F)$-representations
\begin{align*}
\mathcal{A}_\tau(G)\cong\bigoplus_{(\epsilon_v)_v}\bigotimes_v\!'\,\pi^{\epsilon_v}_v,
\end{align*}
where $(\epsilon_v)_v$ runs over sequences in $\{\pm1\}$ indexed by places $v$ of $F$ such that
\begin{itemize}
    \item $\epsilon_v=+1$ when $\tau_v$ is not discrete series,
    \item $\prod_v\epsilon_v = \epsilon(\textstyle\frac12,\tau,\mathrm{Sym}^3)$.
\end{itemize}
\end{enumerate}
\end{conj}
For a more detailed explanation of why this is equivalent to Arthur's formulation, consult \cite[Section 4.4]{AHRR23}.

\subsection{The dihedral case}\label{ss:dihedral}
In this subsection, assume that $F$ is a number field. Henceforth, assume that $\tau$ is \emph{dihedral}, i.e. that $\tau\otimes\omega=\tau$ for the character $\omega:F^\times\backslash\A_F^\times\ra\{\pm1\}$ associated with some quadratic extension $K/F$. Then $\tau$ equals the automorphic induction of a unitary character $\chi:K^\times\backslash\A_K^\times\ra S^1$ \cite[Th\'eor\`eme 3]{Hen12}, the condition that $\tau$ is cuspidal is equivalent to $\chi$ not being self-conjugate, and the condition that $\tau$ has trivial central character is equivalent to $\chi$ being conjugate-symplectic with respect to $K/F$. One computes that $\epsilon(\textstyle\frac12,\tau,\mathrm{Sym}^3)=\epsilon(\frac12,\chi)\cdot\epsilon(\frac12,\chi^3)$.

Let $v$ be a place of $F$. Since $\tau_v$ is the local automorphic induction of $\chi_v$ \cite[Th\'eor\`eme 4]{Hen12}, we see that $\tau_v$ is not discrete series if and only if $v$ splits in $K$ or $\chi_v$ is self-conjugate. Because $\chi_v$ is conjugate-symplectic with respect to $K_v/F_v$, $\chi_v$ being self-conjugate is equivalent to $\chi_v^2=1$.

We will prove the following cases of Conjecture \ref{conj:AMF_long_root_G2} in the dihedral setting; this is our Theorem \ref{thmA}.
\begin{theorem}\label{thm:actualA}
For every place $v$, we define a finite set $\Pi(\tau_v)$ of irreducible smooth representations of $G(F_v)$ depending only on $\tau_v$ such that 
\begin{enumerate}
\item $\Pi(\tau_v)$ consists of one element $\{\pi^+_v\}$ when $v$ splits in $K$ or $\chi_v^2=1$ and two $\{\pi^+_v,\pi^-_v\}$ otherwise,
\item For all nonarchimedean places $v$ of $F$ where $\tau_v$ is unramified, $\pi^+_v$ agrees with the irreducible unramified representation of $G(F_v)$ defined in \S\ref{ss:global_long_root},
\item Assume that $L(\textstyle\frac12,\chi)$ is nonzero, $K_v/F_v$ is unramified at every place $v$ of $F$ above $2$, $K$ is totally real, and $\chi_v:\RR^\times\rightarrow S^1$ satisfies $\chi_v(-1)=1$ at every archimedean place $v$ of $F$. Then we have an isomorphism of $G(\A_F)$-representations
\begin{align*}
\mathcal{A}_\tau(G)\cong\bigoplus_{(\epsilon_v)_v}\bigotimes_v\!'\,\pi^{\epsilon_v}_v,
\end{align*}
where $(\epsilon_v)_v$ runs over sequences in $\{\pm1\}$ indexed by places $v$ of $F$ such that
\begin{itemize}
    \item $\epsilon_v=+1$ when $v$ splits in $K$ or $\chi_v^2=1$,
    \item $\prod_v\epsilon_v = \epsilon(\frac12,\chi^3)$.
\end{itemize}
Moreover, $\bigotimes'_v\pi_v^{\epsilon_v}$ is not cuspidal if and only if every $\epsilon_v$ equals $+1$ and $L(\textstyle\frac12,\chi^3)$ is nonzero.
\end{enumerate}
\end{theorem}
\begin{proof}[Proof outline]
We define $\Pi(\tau_v)$ in \S\ref{ss:G2_packets_nonarch} and \S\ref{ss:G2_packets_arch}. They depend only on $\tau_v$ by Lemma \ref{lem:packetindependence}, and they satisfy part (1) by construction. Part (2) follows from Theorem \ref{thm:groupB}. Finally, part (3) is proved at the end of \S\ref{ss:AMF}.
\end{proof}

\subsection{Local $A$-packets over $p$-adic fields} 
\label{ss:G2_packets_nonarch}
In this subsection, assume that $F$ is a nonarchimedean local field. Recall from \S\ref{ss:unitarytheta} that $\chi:K^\times\ra S^1$ is a unitary character that is conjugate-symplectic, and recall from Example \ref{exmp:J} that $G'_J$ is the connected automorphism group of the Freudenthal algebra $J$ over $F$ associated with $(\mathrm{M}_3(K),\iota)$. Write $\tau$ for the automorphic induction of $\chi$ as in \cite[Definition 3.7]{HH95}, which descends to an irreducible smooth representation of $\PGL_2(F)$ because $\chi$ is conjugate-symplectic.

Recall from \S\ref{ss:U3HPSpadic} the irreducible smooth representations $\sigma^{\pm}$ of $G'_J(F)$ associated with $\chi$ whenever they are defined, and recall from Definition \ref{defn:localtheta} the exceptional theta lift $\theta(-)$. Write $\pi^+$ for $\theta(\sigma^+)$. When $K$ is a field and $\chi^2\neq1$, write $\pi^-$ for $\theta(\sigma^-)$.
\begin{lemma}\label{lem:packetindependence}
Our $\pi^+$ depends only on $\tau$. When $K$ is a field and $\chi^2\neq1$, the same is true for $\pi^-$.
\end{lemma}
\begin{proof}
Let $\epsilon$ be in $\{\pm1\}$. Since $\sigma^\epsilon$ depends on $\chi$, a priori $\pi^\epsilon=\theta(\sigma^\epsilon)$ depends on $\chi$. The only character that also automorphically induces to $\tau$ is the conjugate of $\chi$ \cite[Corollary 5.2]{HH95}, so we must show that replacing $\chi$ by its conjugate leaves $\pi^\epsilon$ unchanged.

Recall from Example \ref{exmp:J} that $\underline{\Aut}(J)$ equals $G'_J\rtimes c^{\Z/2}$. By extending our exceptional theta lift to $G\times\underline{\Aut}(J)$, we see that $\theta(\sigma\circ c)\cong\theta(\sigma)$ for all irreducible smooth representations $\sigma$ of $G'_J(F)$ \cite[Remark 3.2]{AHRR23}. Hence it suffices to show that $\sigma^\epsilon\circ c$ is isomorphic to the $\sigma^\epsilon$ associated with the conjugate of $\chi$.

Recall from Proposition \ref{prop:HPStheta} that $\sigma^\epsilon$ equals the unitary group theta lift of $\rho=\one$ as in \S\ref{ss:unitarytheta} using
\begin{align*}
\epsilon_1=\epsilon\cdot\epsilon_3\cdot\epsilon(\textstyle\frac12,\chi^3,\psi(\tr_{K/F}(\delta-))).
\end{align*} The explicit formula for the lifting of $(\U_1\times\U_3)(F)$ to $\mathrm{Mp}_W(F)$ \cite[Theorem 3.1]{Kud87} shows that applying $-\circ c$ to this unitary group theta lift is isomorphic to the unitary group theta lift as in in \S\ref{ss:unitarytheta} after replacing $\chi$ by its conjugate. Finally, since the conjugate of $\chi$ equals $\chi^{-1}$, this leaves $\epsilon(\textstyle\frac12,\chi^3,\psi(\tr_{K/F}(\delta-)))$\,---\,and hence $\epsilon_1$\,---\,unchanged. Applying Proposition \ref{prop:HPStheta} again, we conclude that $\sigma^\epsilon\circ c$ is isomorphic to the $\sigma^\epsilon$ associated with the conjugate of $\chi$, as desired.
\end{proof}

Recall from \S\ref{ss:parabolics} that $Q$ is a parabolic subgroup of $G$ with Levi subgroup $L$ identified with $\GL_2$.
\begin{thm}\label{thm:groupB}
The representation $\pi^+$ is isomorphic to the unique irreducible quotient of $i_Q^G(\tau\abs{-}^{1/2})$. When $K$ is a field and $\chi^2\neq1$, $\pi^-$ is irreducible and tempered. When $K$ is a field and $\chi^2=1$, $\theta(\sigma^-)$ is zero.
\end{thm}
\begin{proof}
When $K$ is split, this is \cite[Theorem 4.12]{AHRR23}. When $K$ is a field, this is \cite[Theorem 4.10]{AHRR23}.
\end{proof}

\subsection{Local $A$-packets over $\RR$ and $\C$} 
\label{ss:G2_packets_arch}

In this subsection, assume that $F$ is an archimedean local field. Recall from \S\ref{ss:unitarytheta} that $\chi:K^\times\ra S^1$ is a unitary character that is conjugate-symplectic. Write $\tau$ for the automorphic induction of $\chi$ as in \cite[3.5]{Hen10}, which descends to an irreducible smooth representation of $\PGL_2(F)$ because $\chi$ is conjugate-symplectic. Write $\pi^+$ for the unique irreducible quotient of $i^G_Q(\tau\abs{-}^{1/2})$, which evidently only depends on $\tau$.

When $K$ is a field, recall from \S\ref{ss:U3HPSarch} that $\chi(z)=(z/\sqrt{z\overline{z}})^N$ for some odd integer $N$. In particular, $\chi^2\neq1$. Use the standard realization of $\mathsf{G}_2$ in $\{(a,b,c)\in\RR^3\mid a+b+c=0\}$ with simple roots $(1,-1,0)$ and $(-1,2,-1)$, and write $\pi^-$ for the irreducible discrete series representation of $G(F)$ with Harish-Chandra parameter $(\textstyle\frac{\abs{N}+1}2,\frac{\abs{N}-1}2,-\abs{N})$. Since conjugating $\chi$ corresponds to negating $N$, we see that $\pi^-$ only depends on $\tau$.

Recall from \S\ref{ss:U3HPSarch} the irreducible smooth representations $\sigma^\pm$ of $G'_J(F)$ associated with $\chi$ whenever they are defined. Recall from \S\ref{ss:localtheta} the minimal representation $\Omega$ of $\widetilde{G}(F)$, and recall from Definition \ref{defn:localtheta} the exceptional theta lift $\theta(-)$.
\begin{conj}\label{conj:G2HPSarch}\hfill
\begin{enumerate}
    \item $\theta(\sigma^+)$ is isomorphic to $\pi^+$, and $\Hom_{(G\times G'_J)(F)}(\Omega,\sigma^+\otimes\pi^+)$ is $1$-dimensional,
    \item When $K$ is a field and $G'_J$ is anisotropic, $\theta(\sigma^-)$ is isomorphic to $\pi^-$, and $\Hom_{(G\times G'_J)(F)}(\Omega,\sigma^-\otimes\pi^-)$ is $1$-dimensional,
    \item When $K$ is a field and $G'_J$ is quasi-split, $\theta(\sigma^-)$ is zero.
\end{enumerate}
\end{conj}
\begin{remark}
Assume that $K$ is a field. One motivation for Conjecture \ref{conj:G2HPSarch} comes from the \emph{Vogan philosophy}, which suggests that one should consider disjoint unions of packets over pure inner forms (\emph{Vogan packets}). For example, unitary group theta lifts induce bijections between discrete series Vogan packets when the groups have the same rank \cite[\S4]{Pra00}.

The pure inner forms of $G'_J$ are $\PU_{3,0}$ and $\PU_{2,1}$, and the pure inner forms of $G$ are $G$ and its unique compact form $G_c$. Write $\sigma^\pm_{2,1}$ for the irreducible smooth representations of $\PU_{2,1}(\RR)$ associated with $\chi$ as in \S\ref{ss:U3HPSarch}, and write $\sigma^-_{3,0}$ for the irreducible smooth representation of $\PU_{3,0}(\RR)$ associated with $\chi$ as in \S\ref{ss:U3HPSarch}. Write $\pi^-_c$ for the irreducible finite-dimensional representation of $G_c(\RR)$ with highest weight $(\textstyle\frac{|N|-3}2,\frac{|N|-3}2,-|N|+3)$, where we assume $\abs{N}\geq3$ for simplicity.

With any inner form $G'_{\mathrm{in}}$ of $G'_J$ and $G_{\mathrm{in}}$ of $G$, one can associate an inner form $\widetilde{G}_{\mathrm{in}}$ of $\widetilde{G}$ along with a dual pair $G'_{\mathrm{in}}\times G_{\mathrm{in}}\hookrightarrow\widetilde{G}_{\mathrm{in}}$ that recovers our main dual pair of interest when $G_{\mathrm{in}}=G$ \cite[p.~270]{GanLokePaulSavin}. Using the minimal representation of $\widetilde{G}_{\mathrm{in}}(\RR)$, one can similarly form an exceptional theta lift $\theta(-)$ between $G'_{\mathrm{in}}$ and $G_{\mathrm{in}}$.

Since $G'_J$ and $G$ have the same rank, we expect the exceptional theta lifts to induce a bijection between the Vogan packets $\{\sigma^+_{2,1},\sigma^-_{2,1},\sigma^-_{3,0}\}$ and $\{\pi^+,\pi^-,\pi_c^-\}$. One always expects $\sigma^+_{2,1}$ to correspond to $\pi^+$. By combining the $K$-type decomposition of the minimal representation \cite[p.~278--279]{GanLokePaulSavin} with the branching laws in \cite{McGovern}, one can show that $\theta(\pi_c^-)=\sigma^-_{2,1}$.\footnote{We omit the details since this is logically unnecessary for our results.} Hence we expect $\theta(\sigma_{3,0}^-)=\pi^-$ and $\theta(\sigma_{2,1}^-)=0$, which is precisely Conjecture \ref{conj:G2HPSarch}. The situation is summarized in Figure \ref{fig:Howe_duality_C/R}.
\begin{figure}[h]
    \begin{tikzpicture}[scale = 1.4]
        \draw (0,0) node {\underline{$G'_J$}};
    
        \draw (-0.8, -1) node {$\PU_{2,1}$};
        \draw (-0.8, -2) node {$\PU_{3,0}$};
        \draw (0, -0.7) node {$\sigma^+_{2,1}$};
        \draw (0, -1.3) node {$\sigma^-_{2,1}$};
        \draw (0, -2) node {$\sigma^-_{3,0}$};

        \draw (4,0) node {\underline{$\mathsf{G}_2$}};
        
        \draw (4.8,-1) node {$G$};
        \draw (4.8,-2) node {$G_c$};
        \draw (4, -0.7) node {$\pi^+$};
        \draw (4, -1.3) node {$\pi^-$};
        \draw (4, -2) node {$\pi^-_c$};

        \draw[dotted, blue, ->] (0.5, -0.7) -- (3.5, -0.7);
        \draw[dotted, blue, ->] (0.5, -1.3) -- (1.5, -1.3);
        \draw (1.8, -1.3) node {$0$};

        \draw[dashed, orange, ->] (0.5, -1.4) to (3.5, -1.9);
        \draw (2.8, -2.1) node {};

        \draw[darkgreen, ->] (0.5, -1.9) to (3.5, -1.4);
        \draw (1.2, -2.1) node {};
    \end{tikzpicture}
    \caption{An illustration of the expected bijection between the Vogan packets $\{\sigma^+_{2,1},\sigma^-_{2,1},\sigma^-_{3,0}\}$ and $\{\pi^+,\pi^-,\pi^-_c\}$. The dotted blue arrows are the theta correspondence for $\PU_{2,1} \times G$ and indicate parts (1) and (3) of Conjecture \ref{conj:G2HPSarch}. The dashed orange arrow is the theta correspondence for $\PU_{2,1}\times G_c$; it would follow from Howe duality, since we know that $\theta(\pi^-_c)=\sigma^-_{2,1}$. The green arrow is the theta correspondence for $\PU_{3,0} \times G$ and indicates part (2) of Conjecture \ref{conj:G2HPSarch}, which holds by \cite[Theorem 5.2]{NDPC}.} 
    \label{fig:Howe_duality_C/R}
\end{figure}
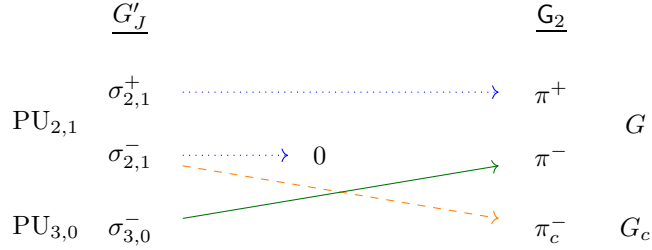
\end{remark}
We can verify Conjecture \ref{conj:G2HPSarch} in the following cases.
\begin{prop}\label{prop:G2HPSarchcases}\hfill
\begin{enumerate}
    \item When $F=\RR$ and $K=\RR\times\RR$, Conjecture \ref{conj:G2HPSarch} holds for $\chi:\RR^\times\ra S^1$ satisfying $\chi(-1)=1$,
    \item When $K$ is a field, Conjecture \ref{conj:G2HPSarch} holds.
\end{enumerate}
\end{prop}
\begin{proof}
In case (1), note that there exists $t$ in $\RR$ such that $\chi(x)=\abs{x}^{it}$ for all $x$ in $\RR$. Now $\sigma^+$ and $\pi^+$ are both spherical with infinitesimal character $(it+\textstyle\frac12,it-\frac12,-2it)$, so the result follows from \cite[p.~6359]{LokeSavin} and \cite[Theorem 8.1]{Li}. In case (2), part (2) of Conjecture \ref{conj:G2HPSarch} follows from \cite[Theorem 5.2]{NDPC}, and parts (1) and (3) of Conjecture \ref{conj:G2HPSarch} follow from \cite[Theorem 11.3.1]{BLS25}.
\end{proof}
%As global evidence for Conjecture \ref{conj:G2HPSarch}.(3), see Theorem \ref{cor:non-vanishing}. There we show that, for any global cuspidal $\sigma$ as in Theorem \ref{thm:globalHPSU3} with an archimedean place $v$ such that $G'_{J,F_v}$ is isomorphic to $\PU_{2,1}$ and $\sigma_v$ is isomorphic to $\sigma^-$, the global theta lift $\theta(\sigma)$ vanishes.

\section{From $\PU_3$ to $\mathsf{G}_2$: cuspidality}\label{sec:cuspidality}
In this section, we prove a criterion for our exceptional theta lift from $\PU_3$ to $\mathsf{G}_2$ to be cuspidal, in terms the \emph{mini-theta} lift arising from the constant term of the minimal representation along the Heisenberg parabolic. Using this criterion, we conclude that the exceptional theta lifts of cuspidal Howe--Piatetski-Shapiro representations remain cuspidal, which is part of Theorem \ref{thmB}.

\subsection{Freudenthal--Jordan algebra}

We start with some algebraic preliminaries about $J$.

\begin{lemma}\label{lemma: square zero matrices}
Let $x$ be in $J$. The following are equivalent:
\begin{enumerate}
\item the $K$-rank of $x$ in $\mathrm{M}_3(K)$ is at most $1$, and $T_J (x) = 0 $.
\item $x\circ x = 0 $.
\item $x$ equals $cv\prescript{t}{}{\overline{v}}$ for some $c$ in $F$ and $v$ in $K^3$ such that $\prescript{t}{}{\overline{v}}e^{-1} v = 0$.
\end {enumerate}
In particular, if $x$ is nonzero and satisfies any of {\rm (1)--(3)}, then $\Stab_{G'_J}x$ is isomorphic to $(\R^1_{K/F}\G_m)\ltimes V'$, where $V'$ equals the unipotent radical of a Borel subgroup of $G'_J$.
\end {lemma}
\begin {proof}
Note that (3) immediately implies (2), and the formulas in Example \ref{exmp:J} show that (2) implies (1). Let us prove that (1) implies (3). This is immediate if $x=0$, so suppose that $x\neq 0$. Then $x$ has $K$-rank $1$, so $x = v\prescript{t}{}{\overline{w}}$ for some nonzero $v$ and $w$ in $K^3$. Since $x$ is Hermitian, we have $v\prescript{t}{}{\overline{w}} = w\prescript{t}{}{\overline{v}}$, which implies that $w = cv $ for some $c$ in $K^\times$. Applying this relation again shows that $c$ lies in $F^\times$. Finally, $$0 = T_J (x) = \trace (cv\prescript{t}{}{\overline{v}}e^{-1}) = c\trace (v\cdot\prescript{t}{}{\overline{v}}e^{-1}) = c\trace(\prescript{t}{}{\overline{v}} e ^ {-1}\cdot v) = c\prescript{t}{}{\overline{v}}e^{-1} v,$$
so $\prescript{t}{}{\overline{v}}e^{-1} v=0$.

For the last statement, note that $v$ is isotropic for the non-degenerate Hermitian form on $K^3$ induced by $e$. Hence the stabilizer of $\ang{v}\subseteq\ang{v}^\perp$ yields a Borel subgroup $B'$ of $G'_J$, and (3) shows that $\Stab_{G'_J}x$ equals the preimage of $\R^1_{K/F}\G_m\subseteq\R_{K/F}\G_m$ under the Levi quotient $B'\twoheadrightarrow\R_{K/F}\G_m$.
\end{proof}

\begin{corollary}\label{cor:xcircy}
Let $x$ and $y$ in $J$ satisfy the equivalent conditions of Lemma
\ref {lemma: square zero matrices}. Then $x\circ y = 0 $
if and only if $x$ and $y$ are $K$-linearly dependent.
\end {corollary}
\begin {proof}
If either $x$ or $y$ vanishes, this is immediate, so suppose that both are nonzero. Write $x = cv\prescript{t}{}{\overline{v}}$ and $y = dw\prescript{t}{}{\overline{w}}$ for $c$ and $d$ in $F$ and $v$ and $w$ in $K^3$ as in Lemma \ref{lemma: square zero matrices}. Then $x $ and $y $ are $K$-linearly dependent if and only if $v $ and $w $ are, and we have $$x\circ y = \textstyle\frac12cd\left (v\prescript{t}{}{\overline{v}} e ^ {-1} w\prescript{t}{}{\overline{w}}+ w\prescript{t}{}{\overline{w}} e ^ {-1} v\prescript{t}{}{\overline{v}}\right). $$
If $v$ and $w $ are $K$-linearly dependent, the vanishing of $\prescript{t}{}{\overline{v}}e^{-1} v$ implies that the above vanishes.

Conversely, suppose that $x\circ y = 0$. Then $(x\circ y) (e ^ {-1} v) = 0 $, so the vanishing of $\prescript{t}{}{\overline{v}}e^{-1} v$ implies that $$0=(x\circ y) (e ^ {-1} v)=v\prescript{t}{}{\overline{v}} e ^ {-1} w\prescript{t}{}{\overline{w}} e ^ {-1} v = v(\prescript{t}{}{\overline{v}} e ^ {-1} w)(\prescript{t}{}{\overline{w}} e ^ {-1} v). $$
Therefore $\prescript{t}{}{\overline{v}} e ^ {-1} w = 0$. Now $e$ induces a non-degenerate Hermitian form on $K^3$, and this shows that the $K$-subspace generated by $v$ and $w$ is isotropic. Hence $v$ and $w$ must be $K$-linearly dependent, as desired.
\end {proof}

Recall from Proposition \ref{prop:min_orbit} the subgroup $L_J\subseteq\GL_J$ of linear maps that preserve $N_J$.
\begin{lemma}\label{lemma:l(x)circlstar(y)}
Let $x$ and $y$ in $J$ satisfy $xe ^ {-1} y = ye ^ {-1} x = 0 $. Then $l(x)\circ l^*(y) = 0$ for all $l$ in $L_J(F)$.
\end{lemma}

\begin{proof}
Recall from Example \ref{exmp:J} the automorphism $c$ of $J$. The group $L_J$ corresponds to
\begin{align*}
\big(\ker(\Nm_{K/F}\circ\det:\R_{K/F}\GL_3\ra\G_m)/(\R^1_{K/F}\G_m)\big)\rtimes c^{\Z/2},
\end{align*} where $g$ in $\R_{K/F}\GL_3$ acts on $J=\mathfrak{h}_3$ via $x\mapsto gx\prescript{t}{}{\overline{g}}$ \cite[p.~330]{bakic2021howe}. Under our identification $J\cong J^*$, our $c$ is self-dual, and the dual of $g$ in $\R_{K/F}\GL_3$ acts via $x\mapsto g'x\prescript{t}{}{\overline{g'}}$, where $g'\coloneqq\iota(g)^{-1}$.

We have $c(x)\circ c^*(y)=\textstyle\frac12(\overline{x}e^{-1}\overline{y}+\overline{y}e^{-1}\overline{x})=\frac12(\overline{xe^{-1}y+ye^{-1}x})=0$, so let us turn to $g$ in $\GL_3(K)$. Then
\begin{align*}
g(x)\circ g^*(y) = \textstyle\frac {1} {2}\left (gx e ^ {-1} ye^{-1}g^{-1} +e \prescript{t}{}{\overline{g}}^{-1} e ^ {-1} ye ^ {-1} x\prescript{t}{}{\overline{g}}e^{-1}\right),
\end{align*}
which vanishes when $xe ^ {-1} y = ye ^ {-1} x = 0 $.
\end {proof}

\subsection{The mini-theta lifts}\label{sec:minitheta}
Recall from \S\ref{ss:heisenberg} that $\widetilde{P}$ is a parabolic subgroup of $\widetilde{G}$ with Levi subgroup $\widetilde{M}$ and unipotent radical $\widetilde{N}$. We identify $\widetilde{M}$ with a quasi-split form of $(\G_m\times\GL_6)/\G_m$ that splits over $K$, where $\G_m\hookrightarrow\G_m\times\GL_6$ is given by $x\mapsto(x^3,x)$ \cite[\S7.2]{gan2021twisted}. Note that the algebraic character $(\G_m\times\GL_6)/\G_m\ra\G_m$ given by $(x,g)\mapsto\det(g)/x$ descends to $\widetilde{M}$ \cite[\S7.2]{gan2021twisted}, which we use to view homomorphisms $F^\times\ra\C^\times$ as homomorphisms $\widetilde{M}(F)\ra\C^\times$.

For the rest of this section, assume that $F$ is a number field, and assume that $K$ is a field. Recall from \S\ref{ss:globaltheta} the global minimal representation $\Omega$. We now study the \emph{mini-theta correspondence} arising from the constant term of $\Omega$ along $\widetilde{N}$. For every place $v$ of $F$, write $\Omega_{\widetilde{M},v}$ for the minimal representation of $\widetilde{M}(F_v)$ as in \cite[\S8.4]{gan2021twisted}, and write $\Omega_{\widetilde{M}}$ for $\bigotimes_v'\Omega_{\widetilde{M},v}$. There exists a canonical $\widetilde{M}(\A_F)$-equivariant injection $\theta:\Omega_{\widetilde{M}}\hookrightarrow\mathcal{A}_{\disc}(\widetilde{M})$.
\begin{prop}\label{prop:jacquet_functor_minimal}
    \leavevmode
    \begin{enumerate}
        \item For every nonarchimedean place $v$ of $F$, we have $(\Omega_v)_{\widetilde N(F_v)}\cong\omega_v \abs{-}^{-2}_v \oplus \Omega_{\widetilde M, v}\abs{-}^{-3/2}_v$ as representations of $\widetilde{M}(F_v)$.
        \item The composition of $\theta:\Omega\ra\mathcal{A}_{\disc}(\widetilde{G})$ with the constant term map $(-)_{\widetilde{N}}:\mathcal{A}(\widetilde{G})\ra\mathcal{A}(\widetilde{M})$ has image lying in $\omega\abs{-}^{-2}\oplus\theta(\Omega_{\widetilde{M}})\abs{-}^{-3/2}$.
    \end{enumerate}
\end{prop}
\begin{proof}
    Part (1) is \cite[\S 8.4]{gan2021twisted}. Part (2) follows from part (1) and \cite[Lemma 6.2]{KST20}.\end{proof}

We write $\phi \mapsto \theta_0(\phi)$ for the automorphic realization of $\Omega_{\widetilde M}$ introduced by Proposition~\ref{prop:jacquet_functor_minimal}.(2). %(Note that it is possible that $\theta_0(\phi) = 0$ for all $\phi \in \Omega_{\widetilde M}$.) \Naomi{Ask Wee Teck if this is actually possible.}

Since $G$ and $G'_J$ are mutual centralizers in $\widetilde{G}$, we see that $G\cap\widetilde{M}=M$ and $G'_J$ are mutual centralizers in $\widetilde M$. This is the dual pair for our mini-theta lift. For any $\phi$ in $\Omega_{\widetilde{M}}$ and $f$ in $\mathcal{A}_{\cusp}(G'_J)$, consider the function 
\begin{equation*}
\theta_0(\phi,f):[M]\ra\C\mbox{ given by }g\mapsto \int_{[G'_J]} \theta_0(\phi)(gg') \overline{f(g')} \dd{g'},
\end{equation*}
and note that $\theta_0(\phi,f)$ is a well-defined automorphic form of $M$. For any irreducible cuspidal automorphic representation $\sigma$ of $G'_J(\A_F)$, write $\theta_0(\sigma)$ for the subspace of $\mathcal{A}(M)$ spanned by $\theta_0(\phi,f)$ as $\phi$ runs over $\Omega_{\widetilde{M}}$ and $f$ runs over $\sigma$.

Recall from \S\ref{ss:roots} the subgroup $\widetilde{N}_\alpha$ of $\widetilde{G}$, and write $\widetilde{S}$ for the Siegel parabolic subgroup of $\widetilde{M}$ whose unipotent radical equals $\widetilde N_\alpha$ \cite[\S7.5]{gan2021twisted}. By checking on Lie algebras, we see that $\widetilde{S}\cap(G\times G'_J)=B\times G'_J$, where $B$ is the Borel subgroup of $G$ from \S\ref{ss:parabolics}. For any $X$ in $\widetilde{\mathfrak{n}}_{-\alpha}(F)$, write $\psi_X:[\widetilde{N}_\alpha]\ra S^1$ for the unitary character given by $\widetilde{n}\mapsto\psi(\ang{X,\widetilde{n}})$. Under our identifications from \S\ref{ss:roots}, $\widetilde{\mathfrak{n}}_\alpha\times\widetilde{\mathfrak{n}}_{-\alpha}$ corresponds to $J\times J$, and the Killing form on $\widetilde{\mathfrak{n}}_\alpha\times\widetilde{\mathfrak{n}}_{-\alpha}$ corresponds to $(x,y)\mapsto T_J(x\circ y)$.

\begin{lemma}
\label{lemma:mini_theta_Fourier}
For any $\phi$ in $\Omega_{\widetilde M}$, we have $\theta_0(\phi) = \displaystyle\sum_{\substack{X\in J\\\operatorname{rk}X \leq  1}}\theta_0(\phi)_{\widetilde N_\alpha,\psi_X}$.
\end{lemma}
\begin{proof}
Apply the short exact sequence in \cite[\S8.8]{gan2021twisted} to any nonarchimedean place of $F$.
\end{proof}

\begin{corollary}\label{cor:mini_theta_cuspidality}
For any $f$ in $\mathcal{A}_{\cusp}(G'_J)$ and $\phi$ in $\Omega_{\widetilde M}$, our $\theta_0 (\phi, f) $ lies in $\mathcal{A}_{\cusp}(M)$.
\end{corollary}
\begin{proof}
We compute the constant term of $\theta_0 (f,\phi)$ along $N_\alpha$ (the unipotent radical of the Borel of $M$):
\begin{align*}
\theta_0 (\phi,f)_{ N_\alpha} (1) & =\int_{[ N_\alpha]}\int_{[G'_J]}\theta_0 (\phi) (ng') \overline{f (g')}\dd g'\dd n\\
& =\int_{[G'_J]}\int_{[N_\alpha]}\sum_{\substack{X\in J\\\operatorname{rk}X\leq 1}}\theta_0 (\phi)_{\widetilde N_\alpha,\psi_X}(ng')\overline{f (g')}\dd n\dd g' \quad \text{by Lemma \ref{lemma:mini_theta_Fourier}}\\
& =\int_{[G'_J]}\sum_{\substack{X\in J\\\operatorname{rk} X\leq 1\\ T_J(X) = 0}}\theta_0 (\phi)_{\widetilde N_\alpha,\psi_X} (g')\overline{f (g')}\dd g'
\end{align*}
since $N_\alpha(F)\subseteq\widetilde{N}_\alpha(F)$ corresponds to $\ang{e}\subseteq J$. Write $T_0$ for the set $\{X\in J\mid\operatorname{rk}X\leq 1\mbox{ and }T_J(X)=0\}$. Because the Killing form is $G'_J$-invariant, we can unfold the above integral to obtain
\begin{align*}
\sum_{X\in G'_J(F)\backslash T_0}\int_{(\Stab_{G'_J}X)(\A_F)\backslash G'_J(\A_F)}\int_{[\Stab_{G'_J}X]}\theta_0(\phi)_{\widetilde N_\alpha,\psi_X}(u'g')\overline{f(u'g')}\dd{u'}\dd{g'}.
\end{align*}
Write $\Stab_{G'_J}X\cong(\R_{K/F}^1\G_m)\ltimes V'$ as in Lemma \ref{lemma: square zero matrices}. Since $\R_{K/F}^1\G_m$ and $V'$ satisfy weak approximation, $\Stab_{G'_J}X$ also satisfies weak approximation, so the argument of Lemma \ref{lem:stabilizerinvariance} implies that $\theta_0(\phi)_{\widetilde{N}_\alpha,\psi_X}(u'g')=\theta_0(\phi)_{\widetilde{N}_\alpha,\psi_X}(g')$. Therefore our expression becomes
\begin{align*}
\sum_{X\in G'_J(F)\backslash T_0}\int_{(\Stab_{G'_J}X)(\A_F)\backslash G'_J(\A_F)}\theta_0(\phi)_{\widetilde N_\alpha,\psi_X}( g')\int_{[\Stab_{G'_J}X]}\overline{f(u'g')}\dd{u'}\dd{g'},
\end{align*}
which vanishes because $f$ is cuspidal.
\end{proof}

For our purposes, we will need the following stronger result. Recall from \S\ref{ss:heisenberg} our identification $M\cong\GL_2$.
\begin{prop}\label{prop:mini_theta_Howe_PS}
Let $\sigma$ be an irreducible cuspidal automorphic representation of $G'_J(\A_F)$.
 Assume that there are infinitely many nonarchimedean places $v$ of $F$ such that
 \begin{itemize}
     \item $v$ splits in $K$ (and hence $G'_{J,F_v}\cong\PGL_3$),
     \item $\sigma_v$ is unramified,
     \item the Satake parameter $S_v$ of $\sigma_v$ has an eigenvalue of norm $\geq q_v^{1/2}$.
 \end{itemize}
Then $\theta_0(\sigma)=0$. In particular, this holds for the cuspidal $\sigma=\bigotimes'_v\sigma_v^{\epsilon_v}$ considered in Theorem \ref{thm:globalHPSU3}.(1).
\end{prop}

\begin{proof}
Corollary \ref{cor:mini_theta_cuspidality} shows that $\theta_0(\sigma)$ is cuspidal and hence semisimple. Therefore it suffices to show that $\theta_0(\sigma)$ has no irreducible subrepresentations $\pi$. For such a $\pi$, let $v$ be a nonarchimedean place of $F$ where $\pi_v$ is unramified and the above conditions are satisfied; such a $v$ exists by infinitude. Write $P_v$ for the Satake parameter of $\pi_v$. By the argument of Proposition \ref{prop:local_global_comp}, our $\pi_v$ is a quotient of $\theta_0(\sigma_v)$. 

For the rest of this proof, we work over $F_v$. Under our identifications, $M\times G'_J\hookrightarrow\widetilde{M}$ corresponds to the morphism $\GL_2\times\PGL_3\hookrightarrow(\G_m\times\GL_6)/\G_m$ of groups over $F_v$ given by $(m,g')\mapsto(\det m^2\det g',m\otimes g')$. Moreover, the construction of $\Omega_{\widetilde{M},v}$ shows that it equals the minimal representation of $\widetilde{M}(F_v)$ \cite[\S8.4]{gan2021twisted}. Therefore $\theta_0$ equals the theta lift associated with a type II dual pair, so \cite[Th\'eor\`eme 1]{Min08} implies that $S_v$ equals the image of $P_v$ under the natural morphism $\SL_2\hookrightarrow\SL_3$ of groups over $\C$.

Now $\pi$ is an irreducible cuspidal automorphic representation of $\GL_2(\A_F)$, which implies that $\pi_v$ is generic. Hence the eigenvalues of $P_v$ have norm $<q_v^{1/2}$ \cite[Corollary (2.5)]{JS81}, so the same holds for the eigenvalues of $S_v$. But by assumption, this cannot happen. Therefore we see that $\theta_0(\sigma)=0$.
\end{proof}

\subsection{Cuspidality}
Recall from Definition \ref{defn:globaltheta} the exceptional theta lift $\theta(-)$. The goal of this subsection is to prove the following criterion for $\theta(-)$ to be cuspidal, which is the main result of this section.

\begin{theorem}\label{thm:cuspidal}
Let $\sigma$ be an irreducible cuspidal automorphic representation of $G'_J(\A_F)$. Then $\theta(\sigma)$ is cuspidal if and only if $\theta_0(\sigma)=0$.
\end{theorem}

Note that Theorem \ref{thm:cuspidal} and Proposition \ref{prop:mini_theta_Howe_PS} immediately imply the following, which is part of Theorem \ref{thmB}.

\begin{corollary}\label{cor:thetacuspidal}
Let $(\epsilon_v)_v$ be a sequence as in Theorem \ref{thm:globalHPSU3}.(1), write $\sigma$ for $\bigotimes'_v\sigma^{\epsilon_v}_v$, and assume that $\sigma$ is cuspidal. Then $\theta(\sigma)$ is cuspidal.
\end{corollary}
To prove Theorem \ref{thm:cuspidal}, we must compute the constant terms of $\theta(\sigma)$ along the unipotent radicals of the two parabolic subgroups of $G$: the Heisenberg parabolic, $P=MN$, and the \emph{three-step} parabolic, $Q=LU$.

\subsubsection{Heisenberg parabolic}
We begin by computing $\theta(\sigma)_N$.

\begin{prop}\label{prop:Heisenberg_constant_term}
For any $f$ in $\mathcal{A}_{\cusp}(G'_J)$ and $\varphi$ in $\Omega$, we have
    $$\theta(\varphi, f)_N(g) = \int_{[G'_J]} \theta(\varphi)_{\widetilde N} (gg')\overline{f(g')} \dd{g'}.$$
    In particular, $\theta(\sigma)_{N}$ equals $
    \theta_0(\sigma)$.
\end{prop}

For this, we need to describe certain stabilizers for the $G'_J(F)$-action on the minimal orbit.

\begin{lemma}\label{lemma:stabilizer_of_v}
For any $X$ in $\O_0(F)\coloneqq\O_{\min}(F)\cap p^{-1}(0,0,0,0)$, its stabilizer in $G'_J$ is isomorphic to $(\R^1_{K/F}\G_m)\ltimes V'$, where $V'$ equals the unipotent radical of a Borel subgroup of $G'_J$.
\end{lemma}
\begin{proof}
By Lemma \ref{lem:p}, $X$ is nonzero and equals $(0,x,y,0)$ for some $x$ and $y$ in $J$ satisfying $T_J(x)=0$ and $T_J(y)=0$. By Proposition~\ref{prop:min_orbit}, we have $x\circ y=0$, $x^\#=0$, and $y^\#=0$. The formulas in Example \ref{exmp:J} then imply that $x$ and $y$ have $K$-rank at most $1$, so Lemma \ref{lemma: square zero matrices} and Corollary \ref{cor:xcircy} show that $x$ and $y$ are $K$-linearly dependent. Applying Lemma \ref{lemma: square zero matrices} again yields the desired result.
\end{proof}

\begin{remark}
If $G'_J$ is anisotropic, then Lemma \ref{lemma:stabilizer_of_v} shows that $\O_0(F)$ is empty.
\end{remark}

\begin{proof}[Proof of Proposition~\ref{prop:Heisenberg_constant_term}]
By replacing $\varphi$ with $g\cdot\varphi$, it suffices to consider the case $g=1$. We compute the constant term of $\theta(\varphi,f)$ along $N$:
    \begin{align*}
    \theta(\varphi, f)_{N}(1) 
         &= \int_{[N]} \int_{[G'_J]} \theta(\varphi)(ng') \overline{f(g')} \dd{g'}\dd{n} \\
         & = \int_{[G'_J]}\int_{[N/Z]} \left(\theta(\varphi)_{\widetilde N}(ng') + \sum_{X \in \mathcal{O}_{\min}(F)}\theta(\varphi)_{\widetilde{N},\psi_X}(ng')  \right)\overline{f(g')}\dd{n}\dd{g'}\quad\text{ by Proposition \ref{prop:FE_of_min}} \\ 
          & = \int_{[G'_J]} \theta(\varphi)_{\widetilde N}(g')\overline{f(g')}\dd{g'} + \int_{[G'_J]} 
           \sum_{X \in \mathcal{O}_0(F)} \theta(\varphi)_{\widetilde{N},\psi_X}(g')\overline{f(g')}\dd{g'}.
\end{align*}
We will show that the rightmost integral vanishes. Because the Killing form is $G'_J$-invariant, we can unfold this integral to obtain
\begin{align*}
\sum_{X\in G'_J(F)\backslash\O_0(F)}\int_{(\Stab_{G'_J}X)(\A_F)\backslash G'_J(\A_F)}\int_{[\Stab_{G'_J}X]}\theta(\varphi)_{\widetilde{N},\psi_X}(u'g')\overline{f(u'g')}\dd{u'}\dd{g'}.
\end{align*}
Write $\Stab_{G'_J}X\cong(\R_{K/F}^1\G_m)\ltimes V'$ as in Lemma \ref{lemma:stabilizer_of_v}. Since $\R_{K/F}^1\G_m$ and $V'$ satisfy weak approximation, $\Stab_{G'_J}X$ also satisfies weak approximation, so Lemma \ref{lem:stabilizerinvariance} implies that $\theta(\varphi)_{\widetilde{N},\psi_X}(u'g')=\theta(\varphi)_{\widetilde{N},\psi_X}(g')$. Therefore our expression becomes
\begin{align*}
\sum_{X\in G'_J(F)\backslash\O_0(F)}\int_{(\Stab_{G'_J}X)(\A_F)\backslash G'_J(\A_F)}\theta(\varphi)_{\widetilde{N},\psi_X}(g')\int_{[\Stab_{G'_J}X]}\overline{f(u'g')}\dd{u'}\dd{g'},
\end{align*}
which vanishes because $f$ is cuspidal.
\end{proof}
\subsubsection{Three-step parabolic}

Next, we want to compute $\theta(\sigma)_U$.
\begin{prop}\label{prop:three-step_constant_term}
For any $f$ in $\mathcal{A}_{\cusp}(G'_J)$ and $\varphi$ in $\Omega$, we have $\theta(\varphi, f)_{U} = 0.$
\end{prop}
To prove Proposition \ref{prop:three-step_constant_term}, we adapt the argument from the split case in \cite{RealAndGlobal}.

\begin{lemma}\label{lemma:X^3=0orbits}
For any $t$ in $F ^\times $, write $$\phi_t: J\to\widetilde {\mathbb X} $$ for the map of algebraic varieties given by $x\mapsto (t, x, t ^ {-1} x ^\#, 0) $. Then
\begin{enumerate}
\item $\phi_t $ restricts to a $G'_J$-equivariant isomorphism of varieties over $F$ $$\left\{x\in J\mid x\circ x\circ x = 0\right\}\xrightarrow {\sim}\O_t \coloneqq \O_{\min}\cap p ^ {-1} (t, 0, 0, 0). $$
\item for any $x$ in $J$ satisfying $x\circ x\circ x = 0 $ and $x\circ x\neq 0 $, consider the Borel subgroup of $G'_J$ associated with $\ker(x)\subseteq\ker(x\circ x) $, and write $V'$ for its unipotent radical. Then there exists a surjective morphism $n: V'\twoheadrightarrow N_\alpha$ of groups over $F$ such that $$\phi_t (u'\cdot x) = n (u')\cdot\phi_t (x)\;\; \text{ for all $u'$ in $V'$.} $$
Moreover, $\phi_t$ induces a bijection $V'(F)\cdot x\xrightarrow{\sim} N_\alpha(F)\cdot\phi_t (x) $.
\end{enumerate}
\end{lemma}

\begin{proof}
Let $x$ in $J$ satisfy $x\circ x\circ x=0$. The formulas in Example \ref{exmp:J} imply that $T_J(x)=0$ and $N_J(x)=0$, so Lemma \ref{lem:p} shows that $\phi_t(x)$ lies in $p^{-1}(t,0,0,0)$. Proposition \ref{prop:min_orbit} and Lemma \ref{lemma:l(x)circlstar(y)} show $\phi_t(x)$ lies in $\O_{\min}$. 

Conversely, Proposition \ref{prop:min_orbit} implies that any $(t,x,y,0)$ in $\O_t$ satisfies $y=t^{-1}x^\#$ and hence
\begin{align*}
0=x\circ y=t^{-1}x\circ x^\# = t^{-1}N_J(x),
\end{align*}
so $N_J(x)=0$. Lemma \ref{lem:p} shows that $T_J(x)=0$ and $T_J(x^\#)=0$ also, so the characteristic polynomial of $x$ becomes $0=x\circ x\circ x-T_J(x)x\circ x+T_J(x^\#)x-N_J(x)=x\circ x\circ x$. This proves part (1).

For part (2), we first observe that since the groups are split over $K$, such a morphism $n$ exists over $K$ by \cite[Lemma 6.1]{RealAndGlobal}. Therefore it suffices to show that $n$ descends to $F$. Indeed, for $u'$ in $V'(K)$ we have $$\phi_t (\overline{u'}\cdot x) =\overline {\phi_t (u'\cdot x)} =\overline {n (u')\cdot\phi_t (x)} =\overline {n (u')}\cdot\phi_t (x), $$
so $n (\overline{u'})\overline {n (u')} ^ {-1} $ acts trivially on $\phi_t (x) $. Now \cite[(6.12)]{RealAndGlobal} implies that $\Stab_{N_\alpha}\phi_t(x)$ is trivial, letting us conclude that $n (\overline u) =\overline {n (u)} $. For the last claim, note that $H^1 (F, \ker n) = 0 $ since $\ker n\subseteq V'$ is unipotent.
\end{proof}

Recall from \S\ref{ss:parabolics} that $V$ is a maximal unipotent subgroup of $G$, and that $L\cong\GL_2$.
\begin{proof}[Proof of Proposition~\ref{prop:three-step_constant_term}]
View $\theta(\varphi,f)_U$ as an element of $\mathcal{A}(L)$, and note that $N_\beta$ is a maximal unipotent subgroup of $L$. Since $UN_\beta=V=NN_\alpha$, we have
\begin{align*}
(\theta(\varphi, f)_{U})_{N_\beta} = \theta(\varphi, f)_V = (\theta(\varphi, f)_{N})_{N_\alpha} = \theta_0(\varphi,f)_{N_\alpha}
\end{align*}
by Proposition \ref{prop:Heisenberg_constant_term}. As $N_\alpha$ is a maximal unipotent subgroup of $M$, Corollary \ref{cor:mini_theta_cuspidality} shows that this vanishes, so $\theta(\varphi,f)_U$ lies in $\mathcal{A}_{\cusp}(L)$.

Note that inclusion induces an isomorphism $N_\alpha\times N_\beta\ra^\sim V^{\mathrm{ab}}$. Hence every unitary character $\psi_\beta:[N_\beta]\ra S^1$ extends uniquely to a unitary character $\psi_V:[V^{\mathrm{ab}}]\ra S^1$ that is trivial on $[N_\alpha]$. We claim that, for every nontrivial such $\psi_\beta$, we have $\theta(\varphi,f)_{V,\psi_V}=0$. This would imply that
\begin{align*}
(\theta(\varphi,f)_U)_{N_\beta,\psi_\beta} = \theta(\varphi,f)_{V,\psi_V}=0,
\end{align*}
and as $\psi_\beta$ runs over all nontrivial unitary characters of $[N_\beta]$, the Fourier expansion of $\theta(\varphi,f)_U$ would show that $\theta(\varphi,f)_U$ vanishes.

Let us turn to the claim. Under our identifications in \S\ref{ss:roots} and \S\ref{ss:globaltheta}, the restriction of $\psi_V$ to $[N]$ corresponds to $(t,0,0,0)$ for some $t$ in $F^\times$. By replacing $\varphi$ with $g\cdot\varphi$, it suffices to evaluate $\theta(\varphi,f)_{V,\psi_V}$ at $g=1$. We get:
    \begin{align*}
    \theta(\varphi, f)_{V,\psi_V}(1) 
         &= \int_{[V]} \int_{[G'_J]} \psi_V(v)^{-1}\theta(\varphi)(vg') \overline{f(g')} \dd{g'}\dd{v} \\
         & = \int_{[G'_J]}\int_{[V/Z]}\psi_V(v)^{-1} \left(\theta(\varphi)_{\widetilde N}(vg') + \sum_{X \in \mathcal{O}_{\min}(F)}\theta(\varphi)_{\widetilde{N},\psi_X}(vg')  \right)\overline{f(g')}\dd{v}\dd{g'}\quad\text{ by Proposition \ref{prop:FE_of_min}} \\ 
          & = \int_{[G'_J]}\sum_{X\in\O_t(F)}\theta(\varphi)_{\widetilde{N},\psi_X}(g')\overline{f(g')}\dd{g'}.
\end{align*}
Because the Killing form is $G'_J$-invariant, we can unfold this integral to obtain
\begin{align*}
\sum_{X\in G'_J(F)\backslash\O_t(F)}\int_{(\Stab_{G'_J}X)(\A_F)\backslash G'_J(\A_F)}\int_{[\Stab_{G'_J}X]}\theta(\varphi)_{\widetilde{N},\psi_X}(v'g')\overline{f(v'g')}\dd{v'}\dd{g'}.
\end{align*}
For $X$ that corresponds to $x$ in $J$ under Lemma \ref{lemma:X^3=0orbits}.(1) satisfying $x\circ x=0$, Lemma \ref{lemma: square zero matrices} and the proof of Proposition \ref{prop:Heisenberg_constant_term} imply that the associated integral vanishes. So consider the case where $x\circ x\neq0$. Note that $\Stab_{G'_J}X$ lies in $V'$ as in Lemma \ref{lemma:X^3=0orbits}.(2), so $\Stab_{G'_J}X$ is unipotent. Therefore $\Stab_{G'_J}X$ satisfies weak approximation, so Lemma \ref{lem:stabilizerinvariance} implies that $\theta(\varphi)_{\widetilde{N},\psi_X}(v'g')=\theta(\varphi)_{\widetilde{N},\psi_X}(g')$. Hence the corresponding term in the above sum equals
\begin{align*}
\int_{(\Stab_{G'_J}X)(\A_F)\backslash G'_J(\A_F)}\theta(\varphi)_{\widetilde{N},\psi_X}(g')\int_{[\Stab_{G'_J}X]}\overline{f(v'g')}\dd{v'}\dd{g'}.
\end{align*}
Since $\psi_V$ is trivial on $[N_\alpha]$, Lemma \ref{lemma:X^3=0orbits}.(2) and $G'_J$-invariance of the Killing form imply that $\theta(\varphi)_{\widetilde{N},\psi_X}(g')$ is $[V']$-invariant. Therefore this term equals
\begin{align*}
\int_{V'(\A_F)\backslash G'_J(\A_F)}\theta(\varphi)_{\widetilde{N},\psi_X}(g')\int_{[V']}\overline{f(v'g')}\dd{v'}\dd{g'},
\end{align*}
which vanishes because $f$ is cuspidal.
\end{proof}
To conclude, note that Theorem \ref{thm:cuspidal} follows from Proposition \ref{prop:Heisenberg_constant_term} and Proposition \ref{prop:three-step_constant_term}.

\section{From $\PU_3$ to $\mathsf{G}_2$: non-vanishing}\label{sec:non-vanishing}
We begin this section by proving a criterion for the exceptional theta lift of cuspidal Howe--Piatetski-Shapiro representations $\sigma$ to not vanish. To do this, we rewrite the generic Fourier coefficients of $\theta(\sigma)$ along the Heisenberg parabolic in terms of torus periods of $\sigma$ along the subtori of $\PU_3$ considered in \S\ref{ss:subtori}. By using our local results from \S\ref{ss:U3HPSpadic} and \S\ref{ss:U3HPSarch}, we obtain the desired criterion.

Next, we spend the rest of this section studying precisely when this criterion is met. Using our global results from \S\ref{ss:PU3HPS}, this amounts to carefully analyzing certain local root numbers while simultaneously finding certain nonvanishing $L$-values. Our study lets us complete the proof of Theorem \ref{thmB}.

\subsection{Fourier coefficients and torus periods} Recall from \S\ref{ss:subtori} that $E$ is a cubic \'etale $F$-algebra, which induces a $2$-dimensional torus $T_E$ over $F$, and recall that any embedding $i:E\hookrightarrow J$ of $F$-algebras induces an injective morphism $i:T_E\hookrightarrow G'_J$ over $F$.

Fix a generator $t$ of $E$ over $F$. Write $f(x)=x^3-bx^2+cx-d$ for its normalized characteristic polynomial over $F$, and under our identifications in \S\ref{ss:roots}, also write $E$ for the element $(1,b,c,d)$ of $\X$.

\begin{lemma}\label{lem:G'orbitsaretori}
We have a natural $G'_J$-equivariant isomorphism of varieties over $F$
\begin{align*}
\{F\mbox{-algebra embeddings }E\hookrightarrow J\}\ra^\sim \O_{\min}\cap p^{-1}(E)
\end{align*}
given by $i\mapsto X\coloneqq(1,i(t),i(t)^\#,N_J(i(t)))$. Under this correspondence, $\Stab_{G'_J}X$ equals the image of $T_E$ under $i:T_E\hookrightarrow G'_J$.
\end{lemma}
\begin{proof}
The proof of Lemma \ref{lemma:l(x)circlstar(y)} shows that, for all $l$ in $L_J(F)$, we have $l^*(i(t)^\#)=l(i(t))^\#$. Since $l$ preserves $N_J$, Proposition \ref{prop:min_orbit} and Lemma \ref{lem:p} imply that our morphism is valued in $\O_{\min}\cap p^{-1}(E)$. Because $f$ is separable, we see that it is even an isomorphism.

Note that $J\otimes_FK$ is isomorphic to $\mathrm{M}_3(K)$ as $K$-algebras with involution, so $i$ induces a embedding $E\otimes_FK\hookrightarrow\mathrm{M}_3(K)$ of $K$-algebras with involution. Now $\Stab_{G'_J}i$ equals the image in $G'_J$ of the intersection of $\U_3$ with the commutator of $E\otimes_FK$ in $\mathrm{M}_3(K)$. By dimension counting, $E\otimes_FK$ is a maximal commutative $K$-subalgebra of $\mathrm{M}_3(K)$, so it equals its own commutator. Finally, since $E\otimes_FK\hookrightarrow\mathrm{M}_3(K)$ commutes with the involution, we see that $\Stab_{G'_J}i$ and hence $\Stab_{G'_J}X$ has the desired form.
\end{proof}

For the rest of this subsection, assume that $F$ is a number field, and assume that $K$ is a field.
\begin{prop}\label{prop:Fouriercoefficient1}
For any $f$ in $\mathcal{A}_{\cusp}(G'_J)$ and $\varphi$ in $\Omega$, we have
\begin{align*}
\theta(\varphi,f)_{N,\psi_E}(1)=\sum_i\int_{i(T_E)(\A_F)\backslash G'_J(\A_F)}\theta(\varphi)_{\widetilde{N},\psi_{X}}(g')\int_{[i(T_E)]}\overline{f(t'g')}\dd{t'}\dd{g'},
\end{align*}
where $i$ runs over $G'_J(F)$-conjugacy clases of $F$-algebra embeddings $E\hookrightarrow J$.
\end{prop}
\begin{proof}
By Proposition \ref{prop:FE_of_min}, we have
\begin{align*}
\theta(\varphi,f)_{N,\psi_E}(1) &= \int_{[N]}\int_{[G'_J]}\psi_E(n)^{-1}\theta(\varphi)(ng')\overline{f(g')}\dd{g'}\dd{n} \\
&= \int_{[G'_J]}\int_{[N/Z]}\psi_E(n)^{-1}\left(\theta(\varphi)_{\widetilde{N}}(ng')+\sum_{X\in\O_{\min}(F)}\theta(\varphi)_{\widetilde{N},\psi_X}(ng')\right)\overline{f(g')}\dd{n}\dd{g'} \\
&= \int_{[G'_J]}\sum_{X\in\O_{\min}(F)\cap p^{-1}(E)}\theta(\varphi)_{\widetilde{N},\psi_X}(g')\overline{f(g')}\dd{g'}
\end{align*}
since $E=(1,b,c,d)$ in $\X$ is nonzero. Because the Killing form is $G'_J$-invariant, we can unfold the above integral to obtain 
\begin{align*}
\sum_{X\in G'_J(F)\backslash(\O_{\min}(F)\cap p^{-1}(E))}\int_{(\Stab_{G'_J}X)(\A_F)\backslash G'_J(\A_F)}\int_{[\Stab_{G'_J}X]}\theta(\varphi)_{\widetilde{N},\psi_{X}}(t'g')\overline{f(t'g')}\dd{t'}\dd{g'} \\
\qquad = \sum_i\int_{i(T_E)(\A_F)\backslash G'_J(\A_F)}\int_{[i(T_E)]}\theta(\varphi)_{\widetilde{N},\psi_{X}}(t'g')\overline{f(t'g')}\dd{t'}\dd{g'} & \qquad \text{by Lemma \ref{lem:G'orbitsaretori}.}
\end{align*}
Since $T_E$ is a $2$-dimensional torus, it is $F$-rational \cite[Theorem 2]{Vos67} and hence satisfies weak approximation. Therefore Lemma \ref{lem:stabilizerinvariance} implies that $\theta(\varphi)_{\widetilde{N},\psi_{X}}(t'g')=\theta(\varphi)_{\widetilde{N},\psi_{X}}(g')$, so our expression finally becomes
\begin{gather*}
\sum_i\int_{i(T_E)(\A_F)\backslash G'_J(\A_F)}\theta(\varphi)_{\widetilde{N},\psi_{X}}(g')\int_{[i(T_E)]}\overline{f(t'g')}\dd{t'}\dd{g'}.\qedhere
\end{gather*}
\end{proof}
For any irreducible cuspidal automorphic representation $\sigma$ of $\PU_3(\A_F)$, recall from \S\ref{ss:PU3HPS} the $\C$-linear map $\mathcal{P}_i:\sigma\ra\C$. Also, since $J$ is not associated with a division algebra, recall that the natural map
\begin{align*}
G'_J(F)\backslash\{F\mbox{-algebra embeddings }E\hookrightarrow J\}\ra\sideset{}{_v'}\prod G'_J(F_v)\backslash\{F_
v\mbox{-algebra embeddings }E_v\hookrightarrow J_v\}
\end{align*}
is a bijection \cite[Lemma 15.5.(2)]{gan2021twisted}.
\begin{thm}\label{thm:vanishingcriterion}
Let $(\epsilon_v)_v$ be a sequence as in Theorem \ref{thm:globalHPSU3}.(1), write $\sigma$ for $\bigotimes'_v\sigma^{\epsilon_v}_v$, and assume that $\sigma$ is cuspidal. Then $\theta(\sigma)$ is nonzero if and only if there exists a cubic \'etale $F$-algebra $E$ such that $\mathcal{P}_i:\sigma\ra\C$ is nonzero, where $i:E\hookrightarrow J$ is the unique $G'_J(F)$-conjugacy class of $F$-algebra embeddings (if it exists) such that, for every place $v$ of $F$, our $i_v$ and $\epsilon_v$ satisfy the conditions in Proposition \ref{prop:HPStorus} or Proposition \ref{prop:archHPStorus}.
\end{thm}
\begin{proof}
For all $f$ in $\sigma$ and $\varphi$ in $\Omega$, Proposition \ref{prop:Fouriercoefficient1} yields
\begin{align*}
\theta(\varphi,f)_{N,\psi_E}(1)=\sum_i\int_{i(T_E)(\A_F)\backslash G'_J(\A_F)}\theta(\varphi)_{\widetilde{N},\psi_{X}}(g')\overline{\mathcal{P}_i(g'\cdot f)} \dd{g'},
\end{align*}
where $i$ runs over $G'_J(F)$-conjugacy classes of $F$-algebra embeddings $E\hookrightarrow J$. Note that $\mathcal{P}_i$ is an element of
\begin{align*}
\Hom_{G'_J(\A_F)}(\sigma,\one) = \textstyle\bigotimes_v\Hom_{G'_J(F_v)}(\sigma_v^{\epsilon_v},\one),
\end{align*}
which Proposition \ref{prop:HPStorus} or Proposition \ref{prop:archHPStorus} indicate is nonzero if and only if, for every place $v$ of $F$, our $i_v$ and $\epsilon_v$ satisfy the conditions therein. Therefore our expression becomes
\begin{align}
\theta(\varphi,f)_{N,\psi_E}(1)=\int_{i(T_E)(\A_F)\backslash G'_J(\A_F)}\theta(\varphi)_{\widetilde{N},\psi_{X}}(g')\overline{\mathcal{P}_i(g'\cdot f)}\dd{g'}.\label{eq:FouriercoefficientII}
\end{align}

Suppose that no $E$ as in Theorem \ref{thm:vanishingcriterion} exists. Corollary \ref{cor:thetacuspidal} shows that $\theta(\sigma)$ is cuspidal and hence semisimple, so it suffices to show that $\theta(\sigma)$ has no irreducible subrepresentations $\pi$. For such a $\pi$, the above discussion shows that $(-)_{N,\psi_E}$ vanishes on $\pi$ for all cubic \'etale $F$-algebras $E$. By \cite[Theorem 3.1]{Gan:Mult_formula_for_cubic}, this cannot happen, so we must have $\theta(\sigma)=0$.

Conversely, suppose that some $E$ as in Theorem \ref{thm:vanishingcriterion} exists. Then there exists $f$ in $\sigma$ of the form $\otimes'_vf_v$ such that $\mathcal{P}_i(f)\neq0$; let $S$ be a finite set of places of $F$ such that, for all $v$ not in $S$,
\begin{itemize}
    \item $v$ does not lie above $\{2,3,\infty\}$,
    \item the \'etale $F_v$-algebras $K_v$ and $E_v$ are unramified,
    \item $X$ lies in $\O_{\min}(\O_{F_v})$, 
    \item $\sigma_v^{\epsilon_v}$ is unramified, and $f_v$ is an unramified vector in $\sigma_v^{\epsilon_v}$.
\end{itemize}
 For all $v$ not in $S$, we claim that the intersection of the $G'_J(F_v)$-orbit of $X$ with $\widetilde\X(\O_{F_v})$ equals the $G'_J(\O_{F_v})$-orbit of $X$. We immediately have $G'_J(\O_{F_v})\cdot X\subseteq (G'_J(F_v)\cdot X)\cap\widetilde\X(\O_{F_v})$. For the reverse inclusion, note that the orbit $G'_J\cdot X$ is isomorphic to $G'_J/i(T_E)$ over $\O_{F_v}$ by Lemma \ref{lem:G'orbitsaretori} and flatness.\footnote{This also uses Lemma \ref{lem:G'orbitsaretori} over $\mathbb{F}_{q_v}$, which holds with the same proof and uses the fact that $v$ does not lie above $\{2,3\}$.} We have
 \begin{align*}
    (G'_J(F_v)\cdot X)\cap\widetilde\X(\O_{F_v})\subseteq(G'_J\cdot X)(F_v)\cap\widetilde\X(\O_{F_v}) = (G'_J\cdot X)(\O_{F_v})
 \end{align*}
since Lemma \ref{lem:G'orbitsaretori} shows that $G'_J\cdot X$ is a closed subscheme of $\widetilde\X$. Finally, Lang's lemma indicates that $(G'_J\cdot X)(\O_{F_v})$ equals $(G'_J)(\O_{F_v})\cdot X$, which proves the claim.

Note that $\varphi\mapsto\theta(\varphi)_{\widetilde{N},\psi_X}(1)$ lies in $\Hom_{\widetilde{N}(\A_F)}(\Omega,\psi_X)=\bigotimes_v'\Hom_{\widetilde{N}(F_v)}(\Omega_v,\psi_{v,X})$. For every nonarchimedean place $v$ of $F$, Proposition \ref{prop:minorbit} identifies $\Omega_{v,Z(F_v)}$ with an $\widetilde{M}(F_v)$-subrepresentation of $C^\infty(\O_{\min}(F_v))$, and under this identification, evaluation at $X$ spans $\Hom_{\widetilde{N}(F_v)}(\Omega_v,\psi_{v,X})$ by Proposition \ref{prop:minorbit}.(2). Therefore the claim and Proposition \ref{prop:minorbit}.(3) show that when $\varphi$ is of the form $\otimes'_v\varphi_v$, where $\varphi_v$ is an unramified vector in $\Omega_v$ for all $v$ not in $S$, the function $g'\mapsto\theta(\varphi)_{\widetilde{N},\psi_X}(g')$ is a $\C^\times$-multiple of the indicator function on $\prod_{v\notin S}G'_J(\O_{F_v})$ when restricted to $G'_J(\A_F^S)$. Because $f_v$ is unramified for all $v$ not in $S$, this implies that \eqref{eq:FouriercoefficientII} equals a $\C^\times$-multiple of
\begin{align}
\int_{i(T_E)(F_S)\backslash G'_J(F_S)}\theta(\varphi)_{\widetilde{N},\psi_X}(g')\overline{{\mathcal{P}_i(g'\cdot f)}}\dd{g'}\label{eq:FouriercoefficientIII},
\end{align}
where $F_S$ denotes $\prod_{v\in S}F_v$.

Next, let $\phi$ be a Bruhat--Schwartz function on $\widetilde{N}(F_S)$. If we use $\phi*\varphi$ in place of $\varphi$, we get
\begin{align*}
\theta(\phi*\varphi)_{\widetilde{N},\psi_X}(g') = \int_{[\widetilde{N}]}\theta(\phi*\varphi)(\widetilde{n}g')\psi_X^{-1}(\widetilde{n})\dd{\widetilde{n}} = \int_{[\widetilde{N}]}\int_{\widetilde{N}(F_S)}\phi(\widetilde{h})\theta(\varphi)(\widetilde{n}g'\widetilde{h})\psi_X^{-1}(\widetilde{n})\dd{\widetilde{h}}\dd{\widetilde{n}}.
\end{align*}
By replacing $\widetilde{n}$ with $\widetilde{n}g'\widetilde{h}^{-1}g'^{-1}$, invariance of the Killing form under $\widetilde{N}$ turns this integral into
\begin{align*}
\int_{[\widetilde{N}]}\int_{\widetilde{N}(F_S)}\phi(\widetilde{h})\theta(\varphi)(\widetilde{n}g')\psi_X^{-1}(\widetilde{n}g'\widetilde{h}^{-1}g'^{-1})\dd{\widetilde{h}}\dd{\widetilde{n}} &= \int_{[\widetilde{N}]}\theta(\varphi)(\widetilde{n}g')\dd{\widetilde{n}}\psi_X^{-1}(\widetilde{n})\int_{\widetilde{N}(F_S)}\phi(\widetilde{h})\psi_{g'^{-1}\cdot X}(\widetilde{h})\dd{\widetilde{h}} \\
&= \theta(\varphi)_{\widetilde{N},\psi_X}(g')\widehat{\phi_Z}(-g'^{-1}\cdot X),
\end{align*}
where $\phi_Z$ denotes the $Z(F_S)$-average of $\phi$, and $\widehat{(-)}$ denotes the Fourier transform on $(\widetilde{N}/Z)(F_S)$. Hence using $\phi*\varphi$ in place of $\varphi$ in \eqref{eq:FouriercoefficientIII} yields
\begin{align*}
\int_{i(T_E)(F_S)\backslash G'_J(F_S)}\theta(\varphi)_{\widetilde{N},\psi_X}(g')\widehat{\phi_Z}(-g'^{-1}\cdot X)\overline{\mathcal{P}_i(g'\cdot f)}  \dd{g'}.
\end{align*}
Note that $g'\mapsto\overline{\mathcal{P}_i(g'\cdot f)}$ is continuous, and it is nonzero at $g'=1$. Now we can choose $\otimes_{v\in S}\varphi_v$ such that $g'\mapsto\theta(\varphi)_{\widetilde{N},\psi_X}(g')$ is nonzero at $g'=1$, and we can choose $\phi$ such that $g'\mapsto\widehat{\phi_Z}(-g'^{-1}\cdot X)$ is a bump function supported on a small neighborhood of $g'=1$. With these choices, we see that the integral is nonzero, so $\theta(\sigma)$ is nonzero, as desired.
\end{proof}

\subsection{Some root numbers over $p$-adic fields}\label{ss:rootnumberspadic}
\label{ss:epsilon_factors}In this subsection, assume that $F$ is a local field, and assume that $K$ is a field. We will use Theorem \ref{thm:vanishingcriterion} to compute the (non-)vanishing of our global theta lift, so let us explicate the associated local conditions in our cases of interest. Let $F'$ be a quadratic \'etale algebra over $F$, and write $K'$ for $K\otimes F'$. 

\begin{lemma}\label{lem:criteria}
     When $E=F\times F'$, the condition on $\epsilon$ in Proposition \ref{prop:HPStorus} or Proposition \ref{prop:archHPStorus} is
    $$\epsilon = [-1]\cdot\epsilon(\textstyle\frac12,\chi^3,\psi(\tr_{K/F}(\delta-))) \cdot \Delta_{F'/F} \cdot \epsilon(\textstyle\frac12,\chi \cdot \omega_{K'/K},\psi(\tr_{K/F}(\delta-))).$$
\end{lemma}
\begin{proof}
    From the expressions in Proposition \ref{prop:HPStorus} or Proposition \ref{prop:archHPStorus} for $\lambda$ and $\epsilon$, we get 
    \begin{equation}\label{eqn:epsilon_case(3)}
        \epsilon = \epsilon(\textstyle\frac12,\chi,\psi(\tr_{K/F}(\delta-)))\cdot[-1]\cdot\epsilon(\frac12,\chi^3,\psi(\tr_{K/F}(\delta-)))\cdot\Delta_{F'/F}\cdot\epsilon(\frac12,\chi\circ\Nm_{K'/K},\psi(\tr_{K'/F}(\delta-))),
    \end{equation}
    where for $F'$ not in $\{F \times F, K\}$, we use the fact that $\Nm_{F'/F}$ induces an isomorphism 
    \begin{align*}
    F'^\times/\Nm_{K'/F'}(K'^\times)\ra^\sim F^\times/\Nm_{K/F}(K^\times),
    \end{align*}
    so the image of $F^\times$ in $F'^\times/\Nm_{K'/F'}(K'^\times)$ is trivial.

    By using inductivity in degree $0$ on the Galois side, we conclude that
    \begin{align*}
        \epsilon(\textstyle\frac12,\chi|_{K'} - \one,\psi(\tr_{K'/F}(\delta-))) & = \epsilon(\textstyle\frac12,\Ind_{K'}^K (\chi|_{K'} - \one),\psi(\tr_{K/F}(\delta-))) \\
        & = \epsilon(\textstyle\frac12,\chi + \chi\cdot\omega_{K'/K} - \one - \omega_{K'/K},\psi(\tr_{K/F}(\delta-))).
    \end{align*}
    Hence the rightmost factor in \eqref{eqn:epsilon_case(3)} equals
    $$\epsilon(\textstyle \frac12, \chi, \psi(\tr_{K/F}(\delta -))) \cdot \epsilon(\textstyle\frac12, \chi \cdot \omega_{K'/K}, \psi(\tr_{K/F}(\delta -))) \cdot \epsilon(\textstyle\frac 12, \omega_{K'/K}, \psi(\tr_{K/F}(\delta -))).$$
    To finish the proof, we just need to check that
    \begin{equation}\label{eqn:epsilonomegais1}
    \epsilon(\textstyle\frac12,\omega_{K'/K},\psi(\tr_{K/F}(\delta-)))= 1.
    \end{equation} 
    For $F'$ in $\{F \times F,K\}$, we have $\omega_{K'/K}=\one$, so \eqref{eqn:epsilonomegais1} holds. For other $F'$, we use the fact that $\Nm_{K/F}$ induces an isomorphism $K^\times/\Nm_{K'/K}(K'^\times)\ra^\sim F^\times/\Nm_{F'/F}(F'^\times)$, so the image of $F^\times$ in $K^\times/\Nm_{K'/K}(K'^\times)$ is trivial. Therefore $\omega_{K'/K}|_{F^\times}$ is trivial, so we can use \cite[Theorem 3]{FQ73} to deduce \eqref{eqn:epsilonomegais1}.
\end{proof}
The fact that $\chi$ is conjugate-symplectic and the equivariance of $\epsilon$-factors under scaling imply that the conditions in Lemma \ref{lem:criteria} are independent of $\delta$ and $\psi$. Hence we will choose whatever $\delta$ and $\psi$ are convenient for computations.

We now study which values of $\epsilon$ can occur, depending on $K$ and $\chi$. For the rest of this subsection, write $\epsilon(\chi, \psi) := \epsilon(\frac12, \chi, \psi)$ to simplify the notation, and assume that $F$ is nonarchimedean. Write $c(-)$ for the conductor of a character. Let us first collect some general lemmas.

\begin{lemma}\label{lem:discriminants}
Write $\mathfrak{D}\subseteq\O_F$ for the discriminant ideal of $F'/F$. Then $v_F(\mathfrak{D})$ equals $c(\omega_{F'/F})$, and it is also congruent to $v_F(\Delta_{F'/F})$ modulo $2$.
\end{lemma}
\begin{proof}
Because $K/F$ is quadratic, the first statement follows immediately from the conductor-discriminant formula. The second statement follows from using the $\O_F$-basis of $\O_K$ that computes $\mathfrak{D}$ as the $F$-basis of $K$ that computes $\Delta_{K/F}$.
\end{proof}

\begin{lemma}\label{lemma:comm_of_Hilbert}
    Assume that $F'$ is a field. Then the image of $\Delta_{F'/F}$ in $F^\times/\Nm_{K/F}(K^\times)\cong\{\pm1\}$ equals the image of $\Delta_{K/F}$ in $F^\times/\Nm_{F'/F}(F'^\times)\cong\{\pm1\}$.
\end{lemma}
\begin{proof}
The image of $\Delta_{F'/F}$ in $F^\times/\Nm_{K/F}(K^\times)\cong\{\pm1\}$ equals the Hilbert symbol $(\Delta_{F'/F},\Delta_{K/F})$, so the desired result follows from commutativity of the Hilbert symbol.
\end{proof}

\begin{lemma}\label{lemma:bound_on_cond}
Suppose that $t\coloneqq c(\omega)$ is at least $1$. Then $c(\chi)$ is at least $2t-1$, and if inequality holds, then $c(\chi)$ is even.
\end{lemma}
\begin{proof}
By assumption, there exists $x$ in $\O_F^\times$ such that $\chi(1+\varpi_F^{t-1}x)=\omega(1+\varpi_F^{t-1}x)\neq1$. Since $t\geq1$, our $K/F$ is ramified, so this implies that there exists $y$ in $\O_K^\times$ such that $\chi(1+\varpi_K^{2t-2}y)\neq1$. Hence $c(\chi)\geq2t-1$.

Now assume that $c\coloneqq c(\chi)$ is odd, and write $d\coloneqq\frac{c-1}2$. Then there exists $z$ in $\O_K^\times$ such that $\chi(1+\varpi_K^{c-1}z)\neq1$, so there exists $w$ in $\O_K^\times$ such that $\chi(1+\varpi_F^dw)\neq1$. Checking valuations shows that $\O_K=\O_F\oplus\O_F\varpi_K$, so $w=a+b\varpi_K$ for some $a$ and $b$ in $\O_F$. Because $w$ is a unit, we see that $a$ must also be a unit. Moreover,
$$1\neq\chi(1+\varpi_F^dw)=\chi(1+\varpi_F^da+\varpi_F^d\varpi_Kb)=\chi(1+\varpi_F^da)=\omega(1+\varpi_F^da)$$
because $1+\varpi_F^da$ is a unit and $v_K(\varpi_F^d\varpi_K)=2d+1=c$. Therefore $a$ witnesses the fact that $\textstyle\frac{c-1}2=d<t$ and consequently $c<2t+1$. Since $c$ is odd and at least $2t-1$, we must have $c=2t-1$, as desired.
\end{proof}

First, we observe that when $K/F$ and $\chi$ are unramified (and hence $\chi^2=1$), we have $\epsilon = +1$.

\begin{prop}\label{prop:unramifiedsigns}
Assume that $K/F$ and $\chi$ are unramified. Then we have $\epsilon=+1$ in Lemma \ref{lem:criteria}.
\end{prop}
\begin{proof}
Now $\Nm_{K/F}:\O_K^\times\ra\O_F^\times$ is surjective, so $[-1]$ is trivial in $F^\times/\Nm_{K/F}(K^\times)$.  Moreover, we may choose $\delta$ and $\psi$ such that $\psi(\tr_{K/F}(\delta-))$ has conductor $0$. Since $\chi^3$ is unramified, we get $\epsilon(\chi^3,\psi(\tr_{K/F}(\delta-)))=+1$. Next, since $K/F$ and $\chi$ are unramified, we see that
\begin{align*}
    \epsilon(\chi \cdot \omega_{K'/K}, \psi(\tr_{K/F}(\delta -))) & = \chi(\varpi_K)^{c(\omega_{K'/K})} \epsilon(\omega_{K'/K}, \psi(\tr_{K/F}(\delta -))) & \text{by \cite[(3.2.6.3)]{Tate:Number_theoretic_background} } \\
    & =  \chi(\varpi_K)^{c(\omega_{K'/K})} & \text{by \eqref{eqn:epsilonomegais1}} \\
    & = (-1)^{c(\omega_{K'/K})} & \text{since $\chi$ is conjugate-symplectic} \\
    & = (-1)^{c(\omega_{F'/F})} & \text{since $\omega_{F'/F}\circ\Nm_{K/F}=\omega_{K'/K}$}.
\end{align*}
To complete the proof, observe that
$$\Delta_{F'/F}\cdot(-1)^{c(\omega_{F'/F})} = +1$$
by applying Lemma \ref{lem:discriminants} and noting that $v_F\pmod{2}$ induces an isomorphism $F^\times/\Nm_{K/F}(K^\times)\ra^\sim\Z/2$.
\end{proof}

Next, when $\chi^2\neq1$, we expect $\epsilon$ to take any value in $\{\pm1\}$ as we vary $F'$. We prove this away from $2$.

\begin{prop}\label{prop:flippable}
Assume that the residue characteristic of $F$ is not $2$, and that $\chi^2\neq1$. For any $\epsilon_0$ in $\{\pm1\}$, there exists an $F'$ such that $\epsilon=\epsilon_0$ in Lemma \ref{lem:criteria}.
\end{prop}
\begin{proof}
If the image of $\Delta_{K/F}$ in $F^\times/\Nm_{K/F}(K^\times)$ equals $-1$, then taking $F' = F \times F$ and $F' 
= K$ in Lemma \ref{lem:criteria} yields different values of $\epsilon$, so one of them must equal $\epsilon_0$. So assume that the image of $\Delta_{K/F}$ in $F^\times/\Nm_{K/F}(K^\times)=+1$. Lemma \ref{lem:criteria} indicates that we must find a field $F'\neq K$ such that
    \begin{equation}\label{eqn:epsilon-1forF'/F}
        \Delta_{F'/F} \cdot \epsilon(\chi, \psi(\tr_{K/F}(\delta -))) \cdot \epsilon(\chi \cdot \omega_{K'/K}, \psi(\tr_{K/F}(\delta -))) = -1.
    \end{equation}

    \textbf{Case 1.} $K/F$ is unramified. We may choose $\delta$ and $\psi$ such that $\psi(\tr_{K/F}(\delta-))$ has conductor $0$. Then, for any conjugate-symplectic unitary character $\chi_0:K^\times\ra S^1$, \cite[Lemma 3.1]{GGP12} yields
    $$\epsilon(\chi_0, \psi(\tr_{K/F}(\delta -))) = (-1)^{c(\chi_0)}.$$
    Since $\omega_{K'/K}|_{F^\times}$ is trivial, $\chi$ and $\chi\cdot\omega_{K'/K}$ are conjugate-symplectic, so the left hand side of \eqref{eqn:epsilon-1forF'/F} equals
    $$\Delta_{F'/F} \cdot (-1)^{c(\chi)} \cdot (-1)^{c(\chi \cdot \omega_{K'/K})} = (-1)^{c(\omega_{F'/F})}  \cdot (-1)^{c(\chi)} \cdot (-1)^{c(\chi \cdot \omega_{K'/K})},$$
where we applied Lemma \ref{lem:discriminants} to $F'/F$. Hence it suffices to show that, for ramified $F'/F$, we have
    $$c(\omega_{F'/F}) + c(\chi) + c(\chi \cdot \omega_{K'/K}) \equiv 1 \pmod 2.$$
    Because the residue characteristic of $F$ is not $2$, we get $c(\omega_{F'/F}) = c(\omega_{K'/K}) = 1$. When $c(\chi) \geq 2$, we see that $c(\chi \cdot \omega_{K'/K}) = c(\chi)$, so we indeed have
    $$c(\omega_{F'/F}) + c(\chi) + c(\chi \cdot \omega_{K'/K}) = 1 + 2c(\chi) \equiv 1 \pmod 2.$$
    Since $K/F$ is unramified, $\chi^2 \neq 1$ implies that $\chi$ is ramified. Thus it remains to consider $c(\chi)=1$. Now $\varpi_F$ is a uniformizer for $\O_K$, and $\chi(\varpi_F)  = \omega(\varpi_F) = -1$,
    so $\chi^2 \neq1$ implies that $\chi^2|_{\O_K^\times} \neq 1$. Hence $(\chi \cdot \omega_{K'/K})^2|_{\O_K^\times} = \chi^2|_{\O_K^\times} \neq1$,
    which shows that $(\chi \cdot \omega_{K'/K})|_{\O_K^\times} \neq 1$. Because $\chi$ and $\omega_{K'/K}$ both have conductor $1$, this implies that their product also has conductor $1$, so altogether we have
    $$c(\omega_{F'/F}) + c(\chi \cdot \omega_{K'/K}) + c(\chi) \equiv 1 + 1 + 1 \equiv 1 \pmod 2.$$
    This concludes the verification of \eqref{eqn:epsilon-1forF'/F} for any ramified $F'/F$.
    
    \textbf{Case 2.} $K/F$ is ramified. Because the residue characteristic of $F$ is not $2$, our $K$ is obtained from $F$ by adjoining the square root of a uniformizer, and the different of $K/F$ has valuation $1$. Thus we may choose $\delta$ to satisfy $v_K(\delta)=-1$, and choosing $\psi$ to have conductor $0$ makes $\psi(\tr_{K/F}(\delta-))$ also have conductor $0$.
    
    Take $F'$ to be the unramified field extension of $F$. Then $K'/K$ is also unramified, so \cite[(3.2.6.3)]{Tate:Number_theoretic_background} yields
    $$\epsilon(\chi \cdot \omega_{K'/K}, \psi(\tr_{K/F}(\delta -))) = \omega_{K'/K}(\varpi_K)^{c(\chi)} \epsilon(\chi, \psi(\tr_{K/F}(\delta -))) = (-1)^{c(\chi)} \epsilon(\chi, \psi(\tr_{K/F}(\delta -))).$$
    Therefore the left hand side of \eqref{eqn:epsilon-1forF'/F} becomes $\Delta_{F'/F} \cdot (-1)^{c(\chi)}$. Because the residue characteristic of $F$ is not 2, we have $t\coloneqq c(\omega)=1$. Since $F'/F$ is unramified, Lemma \ref{lemma:comm_of_Hilbert} and Lemma \ref{lem:discriminants} show that
    $$\Delta_{F'/F} = \omega_{F'/F}( \Delta_{K/F}) = (-1)^t = -1.$$
    
    By Lemma \ref{lemma:bound_on_cond}, we see that $c(\chi)$ is either $1$ or even. Thus it remains to show that $c(\chi)$ is not $1$.

    We claim that, if $c(\chi) = 1$, then $\chi^2=1$, which is not the case for us. Indeed, $\chi|_{\O_K^\times}$ becomes a character of
    $$\O_K^\times/(1+\varpi_K \O_K) = (\O_K/\varpi_K)^\times = (\O_F/\varpi_F)^\times = \O_F^\times/(1 + \varpi_F \O_F),$$
    and because $\chi$ is conjugate-symplectic, we see that $\chi^2|_{\O_K^\times}=1$. Because the residue characteristic of $F$ is not $2$, we may choose $\varpi_F$ and $\varpi_K$ such that $\varpi_F=\varpi_K^2$, so we also get $\chi(\varpi_K)^2 = \chi(\varpi_F) = \omega(\varpi_F) = 1$, where the last equality follows from $[\varpi_F]=\Delta_{K/F}=+1$ in $F^\times/\Nm_{K/F}(K^\times)$. This completes the proof.
\end{proof}

At $2$, we prove this for unramified $K/F$.

\begin{prop}\label{prop:flippable_2_unramified}
Assume that $K/F$ is unramified, the residue characteristic of $F$ is $2$, and $\chi^2\neq1$. For any $\epsilon_0$ in $\{\pm1\}$, there exists an $F'$ such that $\epsilon=\epsilon_0$ in Lemma \ref{lem:criteria}.
\end{prop}
Because $K/F$ is unramified and $\chi$ is conjugate-symplectic, the $\chi^2\neq1$ assumption forces $\chi$ to be ramified. We begin with some lemmas. Write $e\coloneqq v_F(2)$ for the absolute ramification index of $F$.

\begin{lemma}\label{lemma:cond_of_quadratic}
   If the image of $\Delta_{F'/F}$ in $F^\times/\Nm_{K/F}(K^\times)$ is $-1$, then $c(\omega_{F'/F}) = 2e+1$.
 \end{lemma}
 \begin{proof}
Our assumption is that $v_F(\Delta_{F'/F})$ modulo $2$ is nontrivial, so there exists a uniformizer $\varpi_F$ of $\O_F$ such that $F' = F[\sqrt{\varpi_F}]$. Checking valuations shows that $\O_{F'}=\O_F\oplus\O_F\sqrt{\varpi_F}$, so the discriminant ideal of $F'/F$ is generated by $4\varpi_F$. Finally, applying Lemma \ref{lem:discriminants} to $F'/F$ yields the desired result.
 \end{proof}

 \begin{lemma}\label{lemma:c(chiomega<c(chi)}
     Assume that $c(\chi)$ is even, and $c(\chi)$ is at most $2e$. 
     \begin{enumerate}
         \item There exists a quadratic character $\omega':F^\times\ra\{\pm1\}$ such that
        $$c(\chi \cdot \omega' \circ \Nm_{K/F}) < c(\chi),$$
        \item If $\chi^2 \neq 1$, then there exists a quadratic character $\omega':F^\times\ra\{\pm1\}$ such that $c(\chi \cdot \omega' \circ \Nm_{K/F})$ is odd and less than $c(\chi)$.  
     \end{enumerate}
 \end{lemma}
 \begin{proof}
We start by proving part (1). Note that, for abelian groups $A\subseteq B$ satisfying $A\cap 2B=0$, any homomorphism $A\ra\{\pm1\}$ extends to a homomorphism $B\ra\{\pm1\}$. Write $a\coloneqq c(\chi)$, which is even and satisfies $0<a\leq2e$. Let us try to verify these conditions for
\begin{align*}
A\coloneqq\frac{1+\varpi_F^{a-1}\O_F}{1+\varpi_F^a\O_F}\mbox{ and }B\coloneqq\frac{1+\varpi_F\O_F}{1+\varpi_F^a\O_F}.
\end{align*}
Let $y$ be in $1+\varpi_F\O_F$, and write $y=1+\varpi_F^nu$ for some positive integer $n$ and $u$ in $\O_F^\times$. Then
\begin{align*}
y^2-1 = 2\varpi_F^nu+\varpi_F^{2n}u^2\mbox{ and hence } v_F(y^2-1)\geq\min\{e+n,2n\}.
\end{align*}
If $v_F(y^2-1)\leq a-1$, then $v_F(y^2-1)<2e$, which forces $n<e$. Consequently $v_F(y^2-1)=2n$ is even, and because $a-1$ is odd, we see that $y^2$ cannot lie in $1+\varpi_F^{a-1}\O_F$ without lying in $1+\varpi_F^a\O_F$. So $A\cap 2B=0$.
    
    Note that $\O_K/\varpi_F\O_K=\mathbb{F}_{q^2}$, and the restriction
    $$\chi|_{1 + \varpi^{a-1}_F\O_K} \colon \frac{1 + \varpi_F^{a-1}\O_K}{1 + \varpi_F^{a}\O_K} \iso \mathbb F_{q^2} \to S^1$$
    must be quadratic. Since $K/F$ is unramified, $\Nm_{K/F}$ induces a surjective homomorphism
    \begin{align*}
    \frac{1+\varpi_F^{a-1}\O_K}{1+\varpi_F^a\O_K}\twoheadrightarrow\frac{1+\varpi_F^{a-1}\O_F}{1+\varpi_F^a\O_F}
    \end{align*}
    that is identified with $\tr_{\mathbb{F}_{q^2}/\mathbb{F}_q}:\mathbb{F}_{q^2}\twoheadrightarrow\mathbb{F}_q$. Because the residue characteristic is $2$, we see that $\ker(\tr_{\mathbb{F}_{q^2}/\mathbb{F}_q})=\mathbb{F}_q$, and since $\chi$ is conjugate-symplectic, our $\chi|_{1+\varpi_F^{a-1}\O_K}$ is trivial on $\ker(\tr_{\mathbb{F}_{q^2}/\mathbb{F}_q})$. Therefore we obtain a homomorphism $\xi:A\ra\{\pm1\}$ such that $\xi\circ\Nm_{K/F}=\chi|_{1+\varpi_F^{a-1}\O_K}$.
    
    Our earlier discussion yields a homomorphism $\widetilde\xi:B\ra\{\pm1\}$ extending $\xi$. Then $\chi|_{1+\varpi_F^{a-1}\O_K}\cdot\widetilde\xi\circ\Nm_{K/F}$ is trivial, so setting $\omega'$ on $F^\times=\varpi_F^\Z\times\mathbb{F}_q^\times\times(1+\varpi_F\O_F)$ to be trivial on $\varpi_F^\Z\times\mathbb{F}_q^\times$ and to be $\widetilde\xi$ on $1+\varpi_F\O_F$ concludes the proof of part (1).

    We now turn to part (2). Let $\omega':F^\times\ra\{\pm1\}$ be a quadratic character with minimal $c(\chi\cdot\omega'\circ\Nm_{K/F})$. Since $\chi$ is conjugate-symplectic and $\chi^2\neq1$, we see that $\chi^2|_{\O_K^\times}\neq1$, so $c(\chi\cdot\omega'\circ\Nm_{K/F})\geq1$. If $c(\chi\cdot\omega'\circ\Nm_{K/F})$ were even, then applying part (1) to $\chi\cdot\omega'\circ\Nm_{K/F}$ would yield a quadratic character $\omega'':F^\times\ra\{\pm1\}$ such that $c(\chi\cdot(\omega'\omega'')\circ\Nm_{K/F})<c(\chi\cdot\omega'\circ\Nm_{K/F})$. This would violate the minimality of $c(\chi\cdot\omega'\circ\Nm_{K/F})$, so $c(\chi\cdot\omega'\circ\Nm_{K/F})$ must be odd. 
 \end{proof}

 \begin{proof}[Proof of Proposition~\ref{prop:flippable_2_unramified}]
    Case $1$ in the proof of Proposition \ref{prop:flippable} shows it suffices to find a ramified quadratic character $\omega':F^\times\ra\{\pm1\}$ such that
\begin{align*}
c(\omega')+c(\chi)+c(\chi\cdot\omega'\circ\Nm_{K/F})\equiv1\pmod2.
\end{align*}
     We divide our work into four cases, depending on $c(\chi)$.
     \begin{enumerate}
         \item If $c(\chi) \geq 2e+2$, take $\omega' = \omega_{F'/F}$ for $F'/F$ such that $\Delta_{F'/F}$ in $F^\times/\Nm_{K/F}(K^\times)$ is $-1$. Then
\begin{align*}
    c(\omega') + c(\chi) + c(\chi \cdot \omega' \circ \Nm_{K/F}) & = (2e + 1) + c(\chi) + c(\chi)& \text{by Lemma~\ref{lemma:cond_of_quadratic}} \\
             & \equiv 1 \pmod 2.
         \end{align*}

         \item If $c(\chi) \leq 2e$ is odd, take $\omega'$ as in case (1). Then
         \begin{align*}
             c(\omega') + c(\chi) + c(\chi \cdot \omega'  \circ \Nm_{K/F}) & = (2e + 1) + c(\chi) + (2e+1) & \text{by Lemma~\ref{lemma:cond_of_quadratic}} \\
             & \equiv 1 \pmod 2.
         \end{align*}

         \item If $c(\chi) \leq 2e$ is even, then by Lemma \ref{lemma:c(chiomega<c(chi)}.(2) there exists an $\omega'$ such that $c(\chi \cdot \omega' \circ \Nm_{K/F})$ is odd and less than $c(\chi)$. Hence we must have $c(\omega')=c(\omega' \circ \Nm_{K/F}) = c(\chi)$, so
        \begin{align*}
             c(\omega') + c(\chi) + c(\chi \cdot \omega' \circ \Nm_{K/F}) = c(\chi) + c(\chi) + c(\chi\cdot\omega'\circ\Nm_{K/F}) \equiv  1 \pmod 2.
         \end{align*}
         
         \item If $c(\chi) = 2e+1$, take $\omega'$ as in case (1). We divide our work further into two subcases:
         \begin{enumerate}
             \item If $c(\chi \cdot \omega' \circ \Nm_{K/F})$ is odd, then
             \begin{align*}
                c(\omega') + c(\chi) + c(\chi \cdot \omega' \circ \Nm_{K/F}) & = (2e+1) + (2e+1) + c(\chi \cdot \omega' \circ \Nm_{K/F}) & \text{by Lemma~\ref{lemma:cond_of_quadratic}} \\
                & \equiv 1 \pmod 2.
            \end{align*}
            
            \item If $c(\chi \cdot \omega' \circ \Nm_{K/F})$ is even, then it is at most $2e+1$ and hence at most $2e$. Applying Lemma \ref{lemma:c(chiomega<c(chi)}.(2) to $\chi\cdot\omega'\circ\Nm_{K/F}$ yields a quadratic character $\omega'':F^\times\ra\{\pm1\}$ such that $c(\chi \cdot (\omega'\omega'') \circ \Nm_{K/F})$ is odd and less than $2e$. Therefore
            \begin{align*}
             c(\omega' \omega'') + c(\chi) + c(\chi \cdot (\omega' \omega'') \circ \Nm_{K/F}) & = (2e+1) + (2e+1) + c(\chi \cdot (\omega' \omega'') \circ \Nm_{K/F})  \\ 
             & \equiv 1 \pmod 2.
            \end{align*}
            Hence replacing $\omega'$ with $\omega'\omega''$ yields the desired result.
         \end{enumerate}
     \end{enumerate}
     This concludes the proof of Proposition~\ref{prop:flippable_2_unramified}.
 \end{proof}

Finally, when $\chi^2=1$, we expect to get $\epsilon=+1$ for some $F'$. We prove this under our previous hypotheses.

\begin{prop}\label{prop:plusachievable}
Assume that $\chi^2=1$. If the residue characteristic of $F$ is $2$, also assume that $K/F$ is unramified. Then there exists an $F'$ such that $\epsilon=+1$ in  Lemma \ref{lem:criteria}.
\end{prop}
\begin{proof}
   First, assume that $[-1] = +1$. Taking $F' = F \times F$ yields $\Delta_{F'/F} = 1$, so
    \begin{align*}
        \epsilon & = \epsilon(\chi^3,\psi(\tr_{K/F}(\delta-))) \cdot \epsilon(\chi \cdot \omega_{K'/K},\psi(\tr_{K/F}(\delta-))) \\
        & = \epsilon(\chi,\psi(\tr_{K/F}(\delta-))) \cdot \epsilon(\chi,\psi(\tr_{K/F}(\delta-))) & \text{since $\omega_{K'/K} = 1$ and $\chi^2 = 1$} \\
        & = +1.
    \end{align*}

    Now assume that $[-1] = -1$. Then $K/F$ is ramified, so our hypothesis ensures the residue characteristic of $F$ is not $2$. Hence we can assume that $K=F(\sqrt{\varpi_F})$. We claim that there exist no conjugate-symplectic quadratic characters $\chi:K^\times\ra \{\pm1\}$, since such a $\chi$ yields
    $$1 = \chi^2(\sqrt{\varpi_F}) = \chi(\varpi_F) = \omega(\varpi_F) = -1,$$
    where the last equality follows from $\Nm_{K/F}(\sqrt{\varpi_F}) = - \varpi_F$ and $[-1] = -1$.
\end{proof}

\begin{remark}\label{rmk:assumption_at_2}
    We expect Proposition \ref{prop:flippable_2_unramified} and Proposition \ref{prop:plusachievable} to hold even without the assumption that $K/F$ is unramified when the residue characteristic of $F$ is $2$.
\end{remark}

\subsection{Some root numbers over $\RR$} In this subsection, assume that $F=\RR$, and assume that $K=\C$. 
\begin{prop}\label{prop:rootnumbersoverR}
Let $\epsilon$ be in $F^\times/\Nm_{K/F}(K^\times)$. Then there exists a cubic \'etale $F$-algebra $E$ and a $G'_J(F)$-conjugacy class of $F$-algebra embeddings $i:E\hookrightarrow J$ such that $i$ and $\epsilon$ satisfy the conditions in Proposition \ref{prop:archHPStorus} if and only if one of the following holds: 
\begin{itemize}
    \item $G'_J$ is quasi-split and $\epsilon=+1$,
    \item $G'_J$ is anisotropic and $\epsilon=-1$.
\end{itemize}
\end{prop}
\begin{proof}
Recall from \S\ref{ss:U3HPSarch} that $\chi_v(z)=(z/\sqrt{z\overline{z}})^N$ for some odd integer $N$. Note that the conditions in Proposition \ref{prop:archHPStorus} are independent of $\delta$ and $\psi$, so take $\delta=i$ and $\psi(x)=e^{-2\pi ix}$. Then $\epsilon(\textstyle\frac12,\chi,\psi(\tr_{K/F}(\delta-)))$ and $\epsilon(\textstyle\frac12,\chi^3,\psi(\tr_{K/F}(\delta-)))$ equal the sign of $N$ \cite[Proposition 2.1]{GGP12}.

There are only two cubic \'etale $F$-algebras. Both are isomorphic to $F\times F'$ for some quadratic \'etale $F$-algebra $F'$, so Lemma \ref{lem:criteria} shows that the conditions in Proposition \ref{prop:archHPStorus} force $\epsilon=-1\cdot\Delta_{F'/F}$. When $F'=\C$, the $3$-dimensional Hermitian space induced by $\lambda$ is isotropic and $\Delta_{F'/F}=-1$, so $G'_J$ is quasi-split and $\epsilon=+1$. When $F'=\RR^2$, all entries of $\lambda$ in $E^\times/\Nm_{L/E}(L^\times)\cong\{\pm1\}^3$ are equal to the sign of $N$, and $\Delta_{F'/F}=+1$. Therefore $G'_J$ is anisotropic and $\epsilon=-1$.
\end{proof}

\subsection{Global (non-)vanishing}
In this subsection, assume that $F$ is a number field, and assume that $K$ is a field. The following concludes the proof of Theorem \ref{thmB} and is the main result of this section.
\begin{thm}\label{cor:non-vanishing}
Let $(\epsilon_v)_v$ be a sequence as in Theorem \ref{thm:globalHPSU3}.(1), write $\sigma$ for $\bigotimes'_v\sigma^{\epsilon_v}_v$, and assume that $\sigma$ is cuspidal. Consider the following condition:
\begin{gather}\label{eqn:extracondition}
\begin{tabular}{c}
for all places $v$ of $F$, if $\chi_v^2=1$, then $\epsilon_v=+1$, and \\
when $v$ is archimedean, $G'_{J,F_v}$ is anisotropic if and only if $\epsilon_v=-1$.
\end{tabular}
\end{gather}
If \eqref{eqn:extracondition} is not satisfied, then $\theta(\sigma)$ is zero. If \eqref{eqn:extracondition} is satisfied and we also have
\begin{enumerate}
\item $L(\frac12,\chi)$ is nonzero,
\item for all places $v$ of $F$ above $2$, the \'etale $F_v$-algebra $K_v$ is unramified,
\end{enumerate}
then $\theta(\sigma)$ is nonzero. Moreover, if Conjecture \ref{conj:G2HPSarch} holds at all archimedean places of $F$, then $\theta(\sigma)$ is isomorphic to $\bigotimes'_v\pi_v^{\epsilon_v}$, where $\pi_v^{\epsilon_v}$ is the irreducible smooth representation of $G(F_v)$ as in \S\ref{ss:G2_packets_nonarch} or \S\ref{ss:G2_packets_arch}.

In particular, Proposition \ref{prop:G2HPSarchcases} shows that this holds unconditionally when
\begin{enumerate}
    \item[(3)] for all archimedean places $v$ of $F$, either $K_v=\RR\times\RR$ and $\chi_v:\RR^\times\rightarrow S^1$ satisfies $\chi_v(-1)=1$, or $K_v=\C$.
\end{enumerate}
\end{thm}
\begin{proof}
Corollary \ref{cor:thetacuspidal} shows that $\theta(\sigma)$ is cuspidal and hence semisimple. If \eqref{eqn:extracondition} is not satisfied, then one of the following holds:
\begin{itemize}
\item there exists an archimedean place $v$ of $F$ such that $G'_J$ is quasi-split and $\epsilon_v=-1$. Then Proposition \ref{prop:rootnumbersoverR} and Theorem \ref{thm:vanishingcriterion} imply that $\theta(\sigma)$ is zero.
\item there exists a nonarchimedean place $v$ of $F$ such that $\chi_v^2=1$ and $\epsilon_v=-1$. Then Theorem \ref{thm:groupB} and Proposition \ref{prop:local_global_comp} imply that $\theta(\sigma)$ has no irreducible subrepresentations, so $\theta(\sigma)$ is zero.
\end{itemize}
Conversely, suppose that \eqref{eqn:extracondition}, (1), and (2) are satisfied. Let $S$ be a finite set of places of $F$ such that, for all $v$ not in $S$,
\begin{itemize}
    \item $v$ is nonarchimedean,
    \item the \'etale $F_v$-algebra $K_v$ is unramified,
    \item $\chi_v$ is unramified.
\end{itemize}
Let $v$ in $S$ be not split in $K$. We claim that there exists a quadratic \'etale $F_v$-algebra $F_v'$ and a $G'_J(F_v)$-conjugacy class of $F_v$-algebra embeddings $i_v:F_v\times F_v'\hookrightarrow J_v$ such that $i_v$ and $\epsilon_v$ satisfy the conditions in Proposition \ref{prop:HPStorus} or Proposition \ref{prop:archHPStorus}:
\begin{itemize}
    \item when $v$ does not lie above $\{2,\infty\}$ and $\chi_v^2\neq1$, this is Proposition \ref{prop:flippable}.
    \item when $v$ lies above $2$ and $\chi_v^2\neq1$, this follows from (2) and Proposition \ref{prop:flippable_2_unramified}.
    \item when $v$ is nonarchimedean and $\chi_v^2=1$, this follows from \eqref{eqn:extracondition}, (2), and Proposition \ref{prop:plusachievable}.
    \item when $v$ is archimedean, this follows from \eqref{eqn:extracondition} and Proposition \ref{prop:rootnumbersoverR}.
\end{itemize}
Using Krasner's lemma, one can construct a quadratic \'etale $F$-algebra $F'$ such that, for all $v$ in $S$ that does not split in $K$, the completion of $F$ at $v$ is isomorphic to $F_v'$ over $F_v$. Take $E=F\times F'$, and write $K'$ for $K\otimes F'$. For all $v$ not in $S$ that do not split in $K$, \eqref{eqn:extracondition} and Proposition \ref{prop:unramifiedsigns} show that any $G'_J(F_v)$-conjugacy class of $F_v$-algebra embeddings $i_v:F_v\times F_v'\hookrightarrow J_v$ satisfies the conditions in Proposition \ref{prop:HPStorus} with $\epsilon_v$. Note that these conditions are automatic for all places $v$ of $F$ that split in $K$.

Recall from \S\ref{ss:dihedral} that $\tau$ is the irreducible cuspidal automorphic representation of $\PGL_2(\A_F)$ obtained from $\chi$ by automorphic induction. For every place $v$ of $F$, applying the projection formula on the Galois side yields $\tau_v\otimes\omega_{F'_v/F_v}=\Ind_{K_v}^{F_v}(\chi_v\cdot\omega_{K'_v/K_v})$, and inductivity in degree $0$ gives
\begin{align*}
\epsilon(\textstyle\frac12,\tau_v\otimes\omega_{F'_v/F_v}-\one-\omega_v,\psi_v) &= \epsilon(\textstyle\frac12,\Ind_{K_v}^{F_v}(\chi_v\cdot\omega_{K'_v/K_v}-\one),\psi_v) =\epsilon(\textstyle\frac12,\chi_v\cdot\omega_{K'_v/K_v}-\one,\psi_v\circ\tr_{K_v/F_v}).
\end{align*}
Taking the product over all $v$ yields $\epsilon(\frac12,\tau\otimes\omega_{F'/F})\epsilon(\frac12,\omega)=\epsilon(\frac12,\chi\cdot\omega_{K'/K})$. Now taking the product of Lemma \ref{lem:criteria} over all $v$ shows that $\epsilon(\textstyle\frac12,\chi^3) = \textstyle\prod_v\epsilon_v = \epsilon(\textstyle\frac12,\chi^3)\epsilon(\textstyle\frac12,\chi\cdot\omega_{K'/K})$, so $\epsilon(\frac12,\chi\cdot\omega_{K'/K})=1$. Finally, \eqref{eqn:epsilonomegais1} implies that $\epsilon(\frac12,\omega)=1$, so altogether we get $\epsilon(\frac12,\tau\otimes\omega_{F'/F})=1$. Therefore \cite[Theorem B, part (1)]{FH95} lets us assume, after replacing $F'$, that $F'\neq K$ and $L(\textstyle\frac12,\tau\otimes\omega_{F'/F})\neq0$. Hence $\R^1_{L/E}\G_m$ is anisotropic, and (1) implies that
\begin{align*}
L(\textstyle\frac12,\Ind^F_K\chi\otimes\Ind^F_E\one) = L(\textstyle\frac12,\tau)^2L(\textstyle\frac12,\tau\otimes\omega_{F'/F}) = L(\textstyle\frac12,\chi)^2L(\textstyle\frac12,\tau\otimes\omega_{F'/F})\neq0.
\end{align*}
Applying Theorem \ref{thm:vanishingcriterion} and Proposition \ref{prop:globalHPStorus} indicates that $\theta(\sigma)$ is nonzero, as desired.

For any irreducible subrepresentation $\pi$ of $\theta(\sigma)$, Proposition \ref{prop:local_global_comp} indicates that $\pi_v$ is a quotient of $\theta(\sigma_v^{\epsilon_v})$. Using \eqref{eqn:extracondition}, Theorem \ref{thm:groupB} or Conjecture \ref{conj:G2HPSarch} imply that $\theta(\sigma_v^{\epsilon_v})$ is the irreducible representation $\pi_v^{\epsilon_v}$, so it remains to see that $\pi$ appears in $\theta(\sigma)$ with multiplicity one. The proof of Proposition \ref{prop:local_global_comp} shows that the multiplicity space of $\pi$ in $\theta(\sigma)$ equals
\begin{align*}
\Hom_{(G\times G'_J)(\A_F)}(\Omega,\sigma\otimes\pi) = \textstyle\bigotimes_v\Hom_{(G\times G'_J)(F_v)}(\Omega_v,\sigma_v^{\epsilon_v}\otimes\pi_v^{\epsilon_v}).
\end{align*}
By \cite[Theorem 4.1 (ii)]{bakic2021howe} or \cite[Theorem 1.2(ii)]{gan2021howe} for nonarchimedean $v$ and Conjecture \ref{conj:G2HPSarch} for archimedean $v$, each factor in the right hand side is $1$-dimensional. Thus $\pi$ indeed appears in $\theta(\sigma)$ with multiplicity one.
\end{proof}

\section{From $\mathsf{G}_2$ to $\PU_3$}\label{s:liftback}
Our goal in this section is to prove Theorem \ref{thmC}, which concludes the proof of Theorem \ref{thmA}. We start by recalling the dual pair considered by Gan--Savin \cite{gan2021twisted}, which lies in a seesaw with our $\mathsf{G}_2\times\PU_3$ dual pair. (Strictly speaking, we replace $\PU_3$ with a disconnected group whose neutral component is $\PU_3$.) Using this seesaw, we prove a condition for our exceptional theta lift from $\mathsf{G}_2$ to $\PU_3$ to not vanish.

We then prove a criterion for our exceptional theta lift from $\mathsf{G}_2$ to $\PU_3$ to be cuspidal, in terms of the theta lift considered by Gan--Gurevich--Jiang \cite{gan2002cubic}. Using results of Gan \cite{Gan:Mult_formula_for_cubic}, we deduce that the exceptional theta lifts of our dihedral global long root $A$-packets are nonzero and cuspidal. We use this to prove Theorem \ref{thmC}, and we conclude by putting everything together to prove Theorem \ref{thmA}.

\subsection{Twisted composition algebras}\label{ss:TCA}
In \S\ref{sec:non-vanishing}, we proved that certain theta lifts from $G'_J$ to $G$ are nonzero by \emph{using} the non-vanishing of periods along some subtorus $i:T_E\hookrightarrow G'_J$. In this section, we will reverse this strategy and prove that certain theta lifts from $G$ to $G'_J$ are nonzero by \emph{showing} that their period along some subtorus $i:T_E\hookrightarrow G'_J$ is non-vanishing.

For any rank-$2$ $E$-twisted composition algebra $C$ in the sense of \cite[\S36]{KMRT}, write $H_C\coloneqq\underline{\Aut}(C)$ for its automorphism group over $F$. Recall that the Springer construction \cite[Theorem (38.6)]{KMRT} yields a bijection
\begin{align*}
\coprod_J\Aut(J)\backslash\{F\mbox{-algebra embeddings }i:E\hookrightarrow J\}\ra^\sim\{\mbox{rank-}2\mbox{ }E\mbox{-twisted composition algebras }C\},
\end{align*}
where $J$ runs over $9$-dimensional Freudenthal--Jordan algebras over $F$. We have a canonical injective morphism $H_C\hookrightarrow\underline{\Aut}(J)$ of groups over $F$, and this morphism induces an isomorphism on component groups. Under the above correspondence, the neutral component $H^\circ_C$ is identified with $T_E$ by the proof of Lemma \ref{lem:embeddingsGaloiscohom}, and $H_C\hookrightarrow\underline{\Aut}(J)$ recovers $i:T_E\hookrightarrow G'_J$ after restricting to $H_C^\circ=T_E$.

Recall from \S\ref{ss:E6} the Lie algebra $\widetilde\g_E$, which is the quasi-split form of $\mathsf{D}_4$ with respect to $E$ over $F$. Write $G_E$ for the associated simply connected group over $F$. Now \cite[Proposition 6.2]{gan2021twisted} shows that the inclusion $\widetilde\g_E\subseteq\widetilde\g$ induces an injective morphism $G_E\hookrightarrow\underline{\Aut}(\widetilde\g)$ of groups over $F$, and we see that the images of $H_C$ and $G_E$ in $\underline{\Aut}(\widetilde\g)$ commute. This yields a morphism $H_C\times G_E\rightarrow\underline{\Aut}(\widetilde\g)$ of groups over $F$, which is an important dual pair in this section.

Since $\g\subseteq\widetilde\g_E$ and hence $G\hookrightarrow G_E$, we have a seesaw of dual pairs in $\underline{\Aut}(\widetilde\g)$
\begin{center}
    \begin{tikzcd}
     G_E \ar[rd, no head] \ar[d, no head] & \underline{\Aut}(J) \ar[ld, no head] \ar[d, no head]\\
    G & H_C.
    \end{tikzcd}
\end{center}
\subsection{Non-vanishing}
For the rest of this section, assume that $F$ is a number field, and assume that $K$ is a field. We write $\theta_J(-)$ for the exceptional theta lift from Definition \ref{defn:globaltheta}, to emphasize its dependence on the $J$ from the end of Example \ref{exmp:J}.

Recall from \S\ref{ss:dihedral} that $\tau$ is the irreducible cuspidal automorphic representation of $\PGL_2(\A_F)$ obtained from $\chi$ by automorphic induction, and recall from \S\ref{ss:global_long_root} the $G(\A_F)$-subrepresentation $\mathcal{A}_\tau(G)$ of $\mathcal{A}_{\disc}(G)$.
\begin{thm}\label{thm:backnonzero}
Suppose that $L(\textstyle\frac12,\chi)\neq0$. Let $\pi$ be an irreducible subrepresentation of $\mathcal{A}_\tau(G)$, and assume that $\pi$ is cuspidal. Then there exists a $J$ of the form considered at the end of Example \ref{exmp:J} such that $\theta_J(\pi)\neq0$.
\end{thm}

By extending $\Omega$ to a representation of $\Aut(J)\widetilde{G}(\A_F)$ \cite[\S14.3]{gan2021twisted}, one can form a theta lift $\theta_C(-)$ from $H_C$ to $G_E$ \cite[\S14.4]{gan2021twisted}, where we use the fact that $K$ is a field to see that $[H_C]$ has finite measure. For any $f$ in $\mathcal{A}_{\cusp}(G)$, the constant function $1$ on $[H_C]$, and $\varphi$ in $\Omega$, the seesaw from \S\ref{ss:TCA} yields an identity
\begin{align}
\int_{[H_C]}\theta_J(\varphi,f)(h')\dd{h'} = \int_{[G]}\theta_C(\varphi,1)(g)\overline{f(g)}\dd{g}\label{eq:seesawsiegelweil},
\end{align}
where $\theta_J(-)$ denotes the extension of our exceptional theta lift to $G\times\underline{\Aut}(J)$. In other words, roughly speaking, the period of $\theta_J(\pi)$ along $i:T_E\hookrightarrow G_J'$ is controlled by the Petersson inner product with $\theta_C(\one)$.

This motivates us to prove the following analogue of the Siegel--Weil formula. Write $P_E$ for the Heisenberg parabolic of $G_E$, and write $M_E$ for its Levi quotient. Note that $M_E$ is isomorphic to
\begin{align*}
\ker\big(\det:\R_{E/F}\GL_2\ra(\R_{E/F}\G_m)/\G_m\big),
\end{align*}
so it has an algebraic character $\det:M_E\ra\G_m$, which we use to view homomorphisms $\A_F^\times\ra\C^\times$ as homomorphisms $M_E(\A_F)\ra\C^\times$.

Write $I_{P_E}(s,\omega)$ for the 
induced representation $\Ind^{G_E(\A_F)}_{P_E(\A_F)}(\omega\abs{-}^{s+\frac52})$. For any standard section $h$ of $I_{P_E}(s,\omega)$, write $\mathcal{E}(s,h)$ for the associated normalized Eisenstein series as in \cite[p.~2026]{segal2017family}. Now $\mathcal{E}(s,h)$ has a simple pole at $s=\textstyle\frac12$ \cite[Theorem 4.1]{segal2018degenerateEisenstein}, so we obtain a residue map $R:I_{P_E}(\textstyle\frac12,\omega)\ra\mathcal{A}(G_E)$ given by $h\mapsto\res_{s=\frac12}\mathcal{E}(s,h)$.
\begin{prop}\label{prop:siegelweil}
For any $h$ in $I_{P_E}(\textstyle\frac12,\omega)$, there exists finitely many rank-$2$ $E$-twisted composition algebras $C_j$ and $\varphi_j$ in $\Omega$ such that $\res_{s=\frac12}\mathcal{E}(s,h)=\sum_j\theta_{C_j}(\varphi_j,1)$.
\end{prop}
\begin{proof}
The image of $R$ is semisimple by \cite[Theorem 5.4.(4)]{segal2019degenerateresidual} or  \cite[Theorem 6.6.(2)]{segal2019degenerateresidual}, so $R$ factors through the maximal semisimple quotient of $I_{P_E}(\textstyle\frac12,\omega)$. For every place $v$ of $F$, the maximal semisimple quotient of $\Ind_{P_E(F_v)}^{G_E(F_v)}(\omega_v\abs{-}_v^{\frac12+\frac52})$ is isomorphic to $\bigoplus_{C_v}\theta_{C_v}(\one)$, where $C_v$ runs over rank-$2$ $E_v$-twisted composition algebras satisfying $K_{J_v}=K_v$: for nonarchimedean $v$ this is \cite[Corollary 12.11]{gan2021twisted}, and for archimedean $v$ this follows from \cite[Proposition A.1]{segal2019degenerateresidual} and \cite[\S13.2]{gan2021twisted}. Therefore the maximal semisimple quotient of $I_{P_E}(\textstyle\frac12,\omega)=\bigotimes_v'\Ind_{P_E(F_v)}^{G_E(F_v)}(\omega_v\abs{-}_v^{\frac12+\frac52})$ is isomorphic to
\begin{align*}
\bigoplus_{(C_v)_v}\bigotimes_v\!'\,\theta_{C_v}(\one),
\end{align*}
where $(C_v)_v$ runs over sequences of rank-$2$ $E_v$-twisted composition algebras such that $C_v$ is trivial for cofinitely many $v$. Every such $(C_v)_v$ arises from a rank-$2$ $E$-twisted composition algebra $C$ \cite[Lemma 15.5.(2)]{gan2021twisted}, so $\bigotimes_v'\theta_{C_v}(\one)$ is isomorphic to $\theta_C(\one)$ as $G_E(\A_F)$-representations. Finally, \cite[Theorem 16.8]{gan2021twisted} indicates that $\theta_C(\one)$ is an irreducible $G_E(\A_F)$-subrepresentation of $\mathcal{A}_{\disc}(G_E)$ that appears with multiplicity $1$, so we obtain the desired result.
\end{proof}
\begin{remark}
One can use results of Segal \cite{segal2019degenerateresidual} to determine which rank-$2$ $E$-twisted composition algebras $C_j$ yield nonzero terms in Proposition \ref{prop:siegelweil}. However, since this is logically unnecessary for our results, we do not pursue this here.
\end{remark}

Recall from \S\ref{ss:global_long_root} that $\widehat{G}$ is the connected split semisimple group of type $\mathsf{G}_2$ over $\C$, and write $\mathrm{Std}$ for the $7$-dimensional irreducible algebraic representation of $\widehat{G}$.
\begin{prop}\label{prop:residueintegral}
Let $\pi=\bigotimes_v'\pi_v$ be an irreducible cuspidal automorphic representation of $G(\A_F)$, and assume that $\pi_{N,\psi_E}\neq0$. Let $S$ be a finite set of places of $F$ such that, for all $v$ not in $S$,
\begin{itemize}
    \item $v$ is nonarchimedean,
    \item the \'etale $F_v$-algebras $K_v$ and $E_v$ are unramified,
    \item $\pi_v$ is unramified.
\end{itemize}
For all $v$ not in $S$, let $f_v$ be an unramified vector in $\pi_v$. Then there exist $f_S$ in $\bigotimes_{v\in S}\pi_v$, finitely many rank-$2$ $E$-twisted composition algebras $C_j$, and $\varphi_j$ in $\Omega$ such that, for $f=f_S\otimes\bigotimes_{v\notin S}f_v$, we have
\begin{align*}
\res_{s=1}L^S(s,\pi,\omega,\mathrm{Std}) = \sum_j\int_{[G]}\theta_{C_j}(\varphi_j,1)(g)\overline{f(g)}\dd{g},
\end{align*}
where $L^S(s,\pi,\omega,\mathrm{Std})$ denotes the partial $L$-function of $\pi$ with respect to $\mathrm{Std}$, twisted by $\omega$.
\end{prop}
\begin{proof}
By \cite[Theorem 3.1]{segal2017family}, there exist $f_S$ in $\bigotimes_{v\in S}\pi_v$ and a standard section $h$ of $I_{P_E}(s,\omega)$ such that
\begin{align*}
L^S(s+\textstyle\frac12,\pi,\chi,\mathrm{Std})\xi(s) = \displaystyle\int_{[G]}\mathcal{E}(s,h)(g)\overline{f(g)}\dd{g},
\end{align*}
where $\xi(s)$ is holomorphic in a neighborhood of $s=\textstyle\frac12$ and satisfies $\xi(\textstyle\frac12)=1$. After taking residues at $s=\textstyle\frac12$, Proposition \ref{prop:siegelweil} enables us to conclude that
\begin{gather*}
\res_{s=1}L^S(s,\pi,\omega,\mathrm{Std}) = \int_{[G]}\big(\res_{s=\frac12}\mathcal{E}(s,h)(g)\big)\overline{f(g)}\dd{g} = \sum_j\int_{[G]}\theta_{C_j}(\varphi_j,1)(g)\overline{f(g)}\dd{g}.\qedhere
\end{gather*}
\end{proof}
\begin{proof}[Proof of Theorem \ref{thm:backnonzero}]
By \cite[Theorem 3.1]{Gan:Mult_formula_for_cubic}, there exists a cubic \'etale $F$-algebra $E$ such that $\pi_{N,\psi_E}\neq0$. Let $S$ be a finite set of places of $F$ satisfying the conditions in Proposition \ref{prop:residueintegral} such that, for all $v$ not in $S$, we also have $\pi_v\cong\pi_v^+$. This implies that, for all $v$ not in $S$, the local $L$-factor $L_v(s,\pi_v,\omega_v,\mathrm{Std})$ equals
\begin{align*}
L_v(s+\textstyle\frac12,\tau_v)\cdot L_v(s-\textstyle\frac12,\tau_v)\cdot L_v(s,\tau_v,\omega_v,\mathrm{Sym}^2),
\end{align*}
as $\tau_v\otimes\omega_v\cong\tau_v$ because $\tau_v$ is dihedral. Writing $\rho$ for the automorphic induction of $\chi^2$, we get
\begin{align*}
L_v(s,\tau_v,\mathrm{Sym}^2)=L_v(s,\omega_v)\cdot L_v(s,\rho_v)\mbox{ and hence }L_v(s,\tau_v,\omega_v,\mathrm{Sym}^2) = L_v(s,\one)\cdot L_v(s,\rho_v),
\end{align*}
as $\rho_v\otimes\omega_v\cong\rho_v$ because $\rho_v$ is dihedral. Altogether, taking the product over all $v$ not in $S$ yields
\begin{align*}
L^S(s,\pi,\omega,\mathrm{Std}) &= L^S(s+\textstyle\frac12,\tau)\cdot L^S(s-\frac12,\tau)\cdot \zeta_F^S(s)\cdot L^S(s,\rho) \\
&= L^S(s+\textstyle\frac12,\chi)\cdot L^S(s-\frac12,\chi)\cdot \zeta_F^S(s)\cdot L^S(s,\chi^2).
\end{align*}
Because $L^S(\textstyle\frac12,\chi)\neq0$ and $\chi^2\neq1$, this has a simple pole at $s=1$. Combined with Proposition \ref{prop:residueintegral}, we see that there exist $f$ in $\pi$, finitely many rank-$2$ $E$-twisted composition algebras $C_j$, and $\varphi_j$ in $\Omega$ such that
\begin{align*}
\sum_j\int_{[G]}\theta_{C_j}(\varphi_j,1)(g)\overline{f(g)}\dd{g}
\end{align*}
is nonzero. In particular, one of the terms $\int_{[G]}\theta_{C_j}(\varphi_j,1)(g)\overline{f(g)}\dd{g}$ is nonzero, so \eqref{eq:seesawsiegelweil} shows that the corresponding $\int_{[H_{C_j}]}\theta_J(\varphi_j,f)(h')\dd{h'}$ is nonzero, where $J$ denotes the $9$-dimensional Freudenthal--Jordan algebra over $F$ corresponding to $C_j$. In particular, the extension of $\theta_J(\pi)$ to $\underline{\Aut}(J)$ is nonzero, so our original $\theta_J(\pi)$ on $G'_J$ is also nonzero.

We want to show that $J$ is of the form considered at the end of Example \ref{exmp:J}. If this were not the case, then $J$ must equal $B^{\iota=1}$ for some $9$-dimensional central division algebra $B$ over $K$ with an involution of the second kind $\iota:B\rightarrow B$ for $K/F$. The resulting $G'_J$ is anisotropic, so $\theta_J(\pi)$ is automatically cuspidal and hence semisimple.

Let $\sigma$ be an irreducible subrepresentation of $\theta_J(\pi)$. For every place $v$ of $F$, Proposition \ref{prop:local_global_comp} shows that $\sigma_v$ is a quotient of $\theta_J(\pi_v)$. When $v$ does not lie in $S$ and is not a ramified place of $B$, \cite[Theorem 4.1 (ii)]{bakic2021howe} or \cite[Theorem 1.2(ii)]{gan2021howe} indicate that $\theta_J(\pi_v)\cong\theta_J(\pi_v^+)$ is isomorphic to the $\sigma_v^+$ associated with $\chi_v$ from \S\ref{ss:U3HPSpadic} when $\chi^2_v=1$, and it is isomorphic to the sum of the $\sigma^+_v$ associated with $\chi_v$ and $\chi_v^{-1}$ when $\chi_v^2\neq1$. Therefore $\sigma$ satisfies the condition in Theorem \ref{thm:globalHPSU3}.(2), which Theorem \ref{thm:globalHPSU3}.(3) implies cannot happen.
\end{proof}

\subsection{Cuspidality}
The goal of this subsection is to prove the following criterion for $\theta_J(-)$ to be cuspidal.
\begin{theorem}\label{thm:liftbackcuspidal}
Let $\pi$ be an irreducible cuspidal automorphic representation of $G(\A_F)$. If $\pi$ is not in the near equivalence class from \cite[Main Theorem (ii)]{Gan:Mult_formula_for_cubic}, then $\theta_J(\pi)$ is cuspidal. In particular, this holds when $\pi$ is a subrepresentation of $\mathcal{A}_\tau(G)$.
\end{theorem}
This is automatic for anisotropic $G'_J$, so assume that $G'_J$ is quasi-split. Recall from \S\ref{ss:subtori} that $B'$ is a Borel subgroup of $G'_J$ containing the maximal subtorus $T'$, write $V'$ for the unipotent radical of $B'$, and write $Z'$ for the center of $V'$. Write $J_\uparrow\subseteq J$ for the subspace of matrices concentrated above the anti-diagonal, $J_\leftrightarrow\subseteq J$ for the anti-diagonal matrices, and $J_\downarrow\subseteq J$ for the matrices concentrated below the anti-diagonal. Recall that $\mathfrak{l}_J$ corresponds to $\R_{K/F}\mathfrak{sl}_3$, where $a$ in $\R_{K/F}\mathfrak{sl}_3$ acts on $J=\mathfrak{h}_3$ via $x\mapsto ax+x\prescript{t}{}{\overline{a}}$ \cite[p.~330]{bakic2021howe}, and write $\mathfrak{s}\subseteq\mathfrak{sl}_3$ for the subalgebra of diagonal matrices.

Write $\widetilde{B}'$ for the parabolic subgroup of $\widetilde{G}$ that admits a Levi factor $\widetilde{T}'$ with Lie algebra equal to
\begin{align*}
(\mathfrak{sl}_3\oplus\R_{K/F}\mathfrak{s})\oplus(V\otimes J_\leftrightarrow)\oplus(V^*\otimes J_\leftrightarrow^*)
\end{align*}
and whose unipotent radical $\widetilde{V}'$ has Lie algebra $\widetilde{\mathfrak{v}}'$ equal to $(\R_{K/F}\mathfrak{n}_3)\oplus(V\otimes J_\uparrow)\oplus(V^*\otimes J_\downarrow^*)$. By checking on Lie algebras, we see that $\widetilde{B}'\cap(G\times G'_J)= G\times B'$. Note that the center $\widetilde{Z}'$ of $\widetilde{V}'$ has abelian Lie algebra $\widetilde{\mathfrak{z}}'$, and the quotient $\widetilde{V}'/\widetilde{Z}'$ is also abelian \cite[p.~335]{bakic2021howe}. Hence the exponential morphism $\widetilde{V}'/\widetilde{Z}'\cong\widetilde{\mathfrak{v}}'/\widetilde{\mathfrak{z}}'$ is an isomorphism of groups over $F$. Finally, note that $\widetilde{V}'\cap G'_J=V'$ and $\widetilde{Z}'\cap G'_J=Z'$.

Write $\widetilde{\mathfrak{v}}'_{\mathrm{op}}$ for the nilradical of the opposite parabolic of $\widetilde{\mathfrak{b}}'$, and write $\widetilde{\mathfrak{z}}'_{\mathrm{op}}$ for the center of $\widetilde{\mathfrak{v}}'_{\mathrm{op}}$. We use the Killing form to identify $\widetilde{\mathfrak{z}}'^*$ with $\widetilde{\mathfrak{z}}'_{\mathrm{op}}$ and $(\widetilde{\mathfrak{v}}'/\widetilde{\mathfrak{z}}')^*$ with $\widetilde{\mathfrak{v}}'_{\mathrm{op}}/\widetilde{\mathfrak{z}}'_{\mathrm{op}}$. Note that $\widetilde{T}'$ naturally acts on $\widetilde{\mathfrak{v}}'$. For any $X$ in $\widetilde{\mathfrak{z}}'_{\mathrm{op}}$, write $\psi_X:[\widetilde{Z}']\rightarrow S^1$ for the unitary character given by $\widetilde{z}'\mapsto\psi(\ang{X,\widetilde{z}'})$, and for any $X$ in $\widetilde{\mathfrak{v}}'_{\mathrm{op}}/\widetilde{\mathfrak{z}}'_{\mathrm{op}}$, write $\psi_X:[\widetilde{V}'/\widetilde{Z}']\rightarrow S^1$ for the analogous unitary character. Write $\O$ for the minimal nonzero nilpotent $\widetilde{G}$-orbit in $\widetilde\g$.
\begin{lemma}\label{lem:Fourierback}
For any $\varphi$ in $\Omega$, we have $\theta(\varphi)=\theta(\varphi)_{\widetilde{V}'}+\displaystyle\sum_{X\in\O\cap\widetilde{\mathfrak{z}}'_{\mathrm{op}}}\theta(\varphi)_{\widetilde{Z}',\psi_X}+\sum_{X\in\O\cap\widetilde{\mathfrak{v}}'_{\mathrm{op}}/\widetilde{\mathfrak{z}}'_{\mathrm{op}}}\theta(\varphi)_{\widetilde{V}',\psi_X}$.
\end{lemma}
\begin{proof}
Apply \cite[Proposition 3.7]{gan2005minimal} and \cite[I.16]{moeglin1987modeles} to any nonarchimedean place of $F$ not lying above $2$.
\end{proof}
Write $\mathbb{O}$ for the split octonion algebra over $F$, and write $\mathbb{O}_0$ its trace-zero subspace. One can naturally identify the orthogonal complement of $\mathfrak{z}'$ in $\widetilde{\mathfrak{z}}'_{\mathrm{op}}$ with $\mathbb{O}_0$ and the orthogonal complement of $\mathfrak{v}'/\mathfrak{z}'$ in $\widetilde{\mathfrak{v}}'_{\mathrm{op}}/\widetilde{\mathfrak{z}}'_{\mathrm{op}}$ with $\mathbb{O}_0\otimes K$ \cite[p.~336]{bakic2021howe}. Under our identifications, $G\cong\underline{\Aut}(\mathbb{O})$ acts tautologically on $\mathbb{O}_0$, and $T'\cong\R_{K/F}\G_m$ acts on $F$ via $\Nm_{K/F}(-)^{-1}$ and on $K$ via $(-)^{-1}$.

Recall from \S\ref{ss:parabolics} the \emph{three-step} parabolic subgroup $Q=LU$ of $G$.
\begin{lemma}\label{lemma:minorbitV'}
Under our identifications, $\O\cap\mathbb{O}_0$ and $\O\cap(\mathbb{O}_0\otimes K)$ correspond to the $G\times T'$-orbits of 
\begin{align*}
\begin{bmatrix}
0 & (1,0,0) \\
(0,0,0) & 0
\end{bmatrix}\in\mathbb{O}_0
\end{align*}
in $\mathbb{O}_0$ and $\mathbb{O}_0\otimes K$, respectively. In particular, for any $X$ in $\O\cap\mathbb{O}_0$ or $\O\cap(\mathbb{O}_0\otimes K)$, the commutator subgroup of $\Stab_{G\times T'}X$ contains $U$. 
\end{lemma}
\begin{proof}
This follows from the proof of \cite[Lemma 2.3]{bakic2021howe}.
\end{proof}
\begin{corollary}\label{cor:trivUaction}
Let $\varphi$ be in $\Omega$. For any $X$ in $\O\cap\mathbb{O}_0$, we have $\theta(\varphi)_{\widetilde{Z}',\psi_X}(ug)=\theta(\varphi)_{\widetilde{Z}',\psi_X}(ug)$ for all $u$ in $U(\A_F)$ and $g$ in $\widetilde{G}(\A_F)$. For any $X$ in $\O\cap(\mathbb{O}_0\otimes K)$, the same holds for $\theta(\varphi)_{\widetilde{V}',\psi_X}$.
\end{corollary}
\begin{proof}
By replacing $\varphi$ with $g\cdot\varphi$, it suffices to consider the case $g=1$. Now \cite[Proposition 3.7]{gan2005minimal} and \cite[I.16]{moeglin1987modeles} show that $\Hom_{\widetilde{V}'(F_v)}(\Omega_v,\psi_{v,X})$ is $1$-dimensional for every nonarchimedean place $v$ of $F$ not lying above $2$, so the action of $(\Stab_{G\times T'}X)(F_v)$ on $\Hom_{\widetilde{V}'(F_v)}(\Omega_v,\psi_{v,X})$ factors through a character of $(\Stab_{G\times T'}X)(F_v)$. Lemma \ref{lemma:minorbitV'} implies that $U$ lies in the commutator subgroup of $\Stab_{G\times T'}X$, so $U(F_v)$ acts trivially on $\Hom_{\widetilde{V}'(F_v)}(\Omega_v,\psi_{v,X})$.

 Note that $\varphi\mapsto\theta(\varphi)_{\widetilde{V}',\psi_X}(1)$ yields an element $\ell$ of $\Hom_{\widetilde{V}'(\A_F)}(\Omega,\psi_X)$, and $u\cdot\ell$ equals the linear functional $\varphi\mapsto\theta_{\widetilde{V}',\psi_X}(u^{-1})$. So fixing $\varphi$ and varying $u$ yields a continuous function $\ell_\varphi:U(\A_F)\rightarrow\C$. The above shows that $\ell_\varphi$ is invariant under $U(\A_F^{\{2,\infty\}})$, and we see that $\ell_\varphi$ is also invariant under $U(F)$. Since $U$ satisfies weak approximation, this implies that $\ell_\varphi$ is constant, which yields the second statement.

 For the first statement, we argue similarly. Because $\psi_{v,X}$ factors through a continuous homomorphism $\widetilde{Z}'(F_v)\rightarrow F_v$, the Stone--von Neumann theorem yields a unique irreducible unitary representation $W_{\psi_{v,X}}$ of $\widetilde{V}'(F_v)$ with central character $\psi_{v,X}$. Now \cite[Proposition 3.7]{gan2005minimal} and \cite[I.16]{moeglin1987modeles} show that $\Hom_{\widetilde{V}'(F_v)}(\Omega_v,W_{\psi_{v,X}})$ is $1$-dimensional, and Lemma \ref{lemma:minorbitV'} implies that $U(F_v)$ acts trivially on $\Hom_{\widetilde{V}'(F_v)}(\Omega_v,W_{\psi_{v,X}})$.

 Note that $\varphi\mapsto(\widetilde{v}'\mapsto\theta(\varphi)_{\widetilde{Z}',\psi_X}(\widetilde{v}'))$ yields a $\widetilde{V}'(\A_F)$-equivariant map $\ell:\Omega\rightarrow L^2([\widetilde{V}'],\psi_X)$. Since $[\widetilde{V}']$ is compact, $L^2([\widetilde{V}'],\psi_X)$ is semisimple. For any irreducible subrepresentation $W$ of $L^2([\widetilde{V}'],\psi_X)$, uniqueness forces $W_v\cong W_{\psi_{v,X}}$. From here, arguing as above yields the first statement.
\end{proof}
\begin{proof}[Proof of Theorem \ref{thm:liftbackcuspidal}]
Let $f$ be in $\pi$, and let $\varphi$ be in $\Omega$. We compute the constant term of $\theta_J(\varphi,f)$ along $V'$ (the unipotent radical of the Borel of $G'_J$):
\begin{align*}
\theta_J(\varphi,f)_{V'}(1) &= \int_{[V']}\int_{[G]}\theta(\varphi)(gv')\overline{f(g)}\dd{g}\dd{v'} \\
&=\int_{[G]}\int_{[V']}\left(\theta(\varphi)_{\widetilde{V}'}(gv')+\displaystyle\sum_{X\in\O\cap\widetilde{\mathfrak{z}}'_{\mathrm{op}}}\theta(\varphi)_{\widetilde{Z}',\psi_X}(gv')+\sum_{X\in\O\cap\widetilde{\mathfrak{v}}'_{\mathrm{op}}/\widetilde{\mathfrak{z}}'_{\mathrm{op}}}\theta(\varphi)_{\widetilde{V}',\psi_X}(gv')\right)\overline{f(g)}\dd{v'}\dd{g} \\
&= \int_{[G]}\left(\theta(\varphi)_{\widetilde{V}'}(g)+\displaystyle\sum_{X\in\O\cap\mathbb{O}_0}\theta(\varphi)_{\widetilde{Z}',\psi_X}(g)+\sum_{X\in\O\cap(\mathbb{O}_0\otimes K)}\theta(\varphi)_{\widetilde{V}',\psi_X}(g)\right)\overline{f(g)}\dd{g}
\end{align*}
by Lemma \ref{lem:Fourierback}. Because the Killing form is $G$-invariant, we can unfold the middle and rightmost integrals to obtain
\begin{align*}
\int_{[G]}\theta(\varphi)_{\widetilde{V}'}(g)\overline{f(g)}\dd{g}&+\displaystyle\sum_{X\in G(F)\backslash(\O\cap\mathbb{O}_0)}\int_{(\Stab_GX)(F)\backslash G(\A_F)}\theta(\varphi)_{\widetilde{Z}',\psi_X}(g)\overline{f(g)}\dd{g} \\
&+\sum_{X\in G(F)\backslash(\O\cap(\mathbb{O}_0\otimes K))}\int_{(\Stab_GX)(F)\backslash G(\A_F)}\theta(\varphi)_{\widetilde{V}',\psi_X}(g)\overline{f(g)}\dd{g}.
\end{align*}
By Corollary \ref{cor:trivUaction}, the middle and rightmost integrals vanish because $f$ is cuspidal. As for the first integral, note that the restriction of $\Omega_{\widetilde{V}'}$ to $\widetilde{T}'^{\mathrm{der}}$ is isomorphic to the minimal representation as in \cite[Section 5]{gan2002cubic} for the cubic \'etale $F$-algebra $F\times K$. Because $\pi$ does not lie in the near equivalence class from \cite[Main theorem (ii)]{Gan:Mult_formula_for_cubic}, that result implies that the first integral vanishes, as desired.
\end{proof}

\subsection{Arthur's multiplicity formula}\label{ss:AMF}
Recall from \S\ref{ss:PU3HPS} the subrepresentation $\mathcal{A}_\chi(G'_J)$ of $\mathcal{A}_{\disc}(G'_J)$.
\begin{prop}\label{prop:allthetalifts}
Suppose that
\begin{enumerate}
\item $L(\frac12,\chi)$ is nonzero,
\item for all places $v$ of $F$ above $2$, the \'etale $F_v$-algebra $K_v$ is unramified.
\end{enumerate}
If Conjecture \ref{conj:G2HPSarch} holds at all archimedean places of $F$, then we have an isomorphism of $G(\A_F)$-representations
\begin{align*}
\sum_{(J,\sigma)}\theta_J(\sigma)\cong\bigoplus_{(\epsilon_v)_v}\bigotimes_v\!'\,\pi^{\epsilon_v}_v,
\end{align*}
where $J$ runs over Freudenthal algebras over $F$ of the form considered at the end of Example \ref{exmp:J}, $\sigma$ runs over irreducible cuspidal subrepresentations of $\mathcal{A}_\chi(G'_J)$, and $(\epsilon_v)_v$ runs over sequences as in Theorem \ref{thm:actualA}.(3) such that, if $L(\textstyle\frac12,\chi^3)$ is nonzero, then not every $\epsilon_v$ equals $+1$.

In particular, Proposition \ref{prop:G2HPSarchcases} shows that this holds unconditionally when
\begin{enumerate}
    \item[(3)] for all archimedean places $v$ of $F$, either $K_v=\RR\times\RR$ and $\chi_v:\RR^\times\rightarrow S^1$ satisfies $\chi_v(-1)=1$, or $K_v=\C$.
\end{enumerate}
\end{prop}
\begin{proof}
Fix such a $J$. Theorem \ref{thm:globalHPSU3}.(1) and Theorem \ref{thm:globalHPSU3}.(4) identify $\sigma\cong\bigotimes'_v\sigma_v^{\epsilon_v}$ for some sequence $(\epsilon_v)_v$ as in Theorem \ref{thm:globalHPSU3}.(1) such that, if $G'_J$ is quasi-split and $L(\frac12,\chi^3)$ is nonzero, then not every $\epsilon_v$ equals $+1$. Under our assumptions, Theorem \ref{cor:non-vanishing} shows that $\theta_J(\sigma)$ is nonzero if and only if \eqref{eqn:extracondition} is satisfied, in which case $\theta_J(\sigma)\cong\bigotimes'_v\pi_v^{\epsilon_v}$.

The first part of \eqref{eqn:extracondition} is equivalent to $(\epsilon_v)_v$ also satisfying the conditions in Theorem \ref{thm:actualA}.(3). Moreover, given a sequence $(\epsilon_v)_v$ as in Theorem \ref{thm:actualA}.(3), there exists a unique choice of $J$ such that $(\epsilon_v)_v$ satisfies the second part of \eqref{eqn:extracondition} with respect to $G'_J$. Therefore, as $J$ varies, we obtain the desired result.
\end{proof}

\begin{prop}\label{prop:thetabackfull}
Suppose that $L(\textstyle\frac12,\chi)\neq0$, and let $\pi$ be an irreducible cuspidal subrepresentation of $\mathcal{A}_\tau(G)$. Then $\pi$ lies in $\sum_{(J,\sigma)}\theta_J(\sigma)$, where $J$ runs over Freudenthal algebras over $F$ of the form considered at the end of Example \ref{exmp:J}, and $\sigma$ runs over irreducible cuspidal subrepresentations of $\mathcal{A}_\chi(G'_J)$.
\end{prop}
\begin{proof}
Let $S$ be a finite set of places of $F$ such that, for all $v$ not in $S$, our $v$ is nonarchimedean and $\pi_v\cong\pi_v^+$. Theorem \ref{thm:backnonzero} yields such a $J$ satisfying $\theta_J(\pi)\neq0$, and Theorem \ref{thm:liftbackcuspidal} shows that $\theta_J(\pi)$ is cuspidal and hence semisimple.

Let $\sigma$ be an irreducible subrepresentation of $\theta_J(\pi)$. For every place $v$ of $F$, Proposition \ref{prop:local_global_comp} shows that $\sigma_v$ is a quotient of $\theta_J(\pi_v)$. When $v$ does not lie in $S$, \cite[Theorem 4.1 (ii)]{bakic2021howe} or \cite[Theorem 1.2(ii)]{gan2021howe} indicate that the irreducible quotients of $\theta_J(\pi_v)\cong\theta_J(\pi_v^+)$ are the $\sigma_v^+$ associated with $\chi_v$ or $\chi_v^{-1}$, so Theorem \ref{thm:globalHPSU3}.(2) shows that $\sigma$ lies in $\mathcal{A}_{\chi'}(G'_J)$ for some conjugate-symplectic unitary character $\chi':K^\times\backslash\A_K^\times\ra S^1$ such that, for all $v$ not in $S$, our $\chi'_v$ equals either $\chi_v$ or $\chi_v^{-1}$. Then Lemma \ref{lem:invertedatplaces} below shows that $\chi'$ equals either $\chi$ or $\chi^{-1}$. Altogether, we see that $\theta_J(\pi)$ lies in $\sum_\sigma\sigma$, where $\sigma$ runs over irreducible cuspidal subrepresentations of $\mathcal{A}_\chi(G'_J)\oplus\mathcal{A}_{\chi^{-1}}(G'_J)$.

Because $\theta_J(\sigma)\neq0$, there exist $f$ in $\pi$ and $\varphi$ in $\Omega$ such that $f'\coloneqq\theta(\varphi,f)\neq0$. The above shows that $f'$ lies in $\sum_\sigma\sigma$. Now the Petersson inner product of $\theta(\varphi,f')$ and $f$ equals
\begin{align*}
\int_{[G]}\int_{[G'_J]}\theta(\varphi)(gg')\overline{f'(g')}\cdot\overline{f(g)}\dd{g'}\dd{g} = \int_{[G'_J]}\int_{[G]}\theta(\varphi)(gg')\overline{f(g)}\cdot\overline{f'(g')}\dd{g}\dd{g'} = \int_{[G'_J]}f'(g')\overline{f'(g')}\dd{g'}\neq0,
\end{align*}
so $\sum_\sigma\theta_J(\sigma)$ pairs nontrivially with $\pi$. Because $\pi$ is irreducible, we see that $\pi\subseteq\sum_\sigma\theta_J(\sigma)$.

Now the conjugate of $\chi$ equals $\chi^{-1}$, so any irreducible subrepresentation $\sigma$ of $\mathcal{A}_{\chi^{-1}}(G'_J)$ has $\sigma\circ c$ lying in $\mathcal{A}_{\chi}(G'_J)$. Using the extension of our theta lift to $G\times\underline{\Aut}(J)$, we see that $\theta_J(-\circ c)=\theta_J(-)$. Therefore when forming $\sum_\sigma\theta_J(\sigma)$, we only need irreducible cuspidal subrepresentations $\sigma$ of $\mathcal{A}_\chi(G'_J)$, as desired.
\end{proof}
\begin{lemma}\label{lem:invertedatplaces}
Let $\chi$ and $\chi'$ be two unitary characters $K^\times\backslash\A_K^\times\ra S^1$ such that, for cofinitely many places $v$ of $F$, our $\chi'_v$ equals either $\chi_v$ or $\chi_v^{-1}$. Then $\chi'$ equals either $\chi$ or $\chi^{-1}$.
\end{lemma}
\begin{proof}
Let $S$ be a finite set of places of $F$ such that, for all $v$ not in $S$, our $\chi'_v$ equals either $\chi_v$ or $\chi_v^{-1}$. Write $\tau$ and $\tau'$ for the isobaric sums $\chi\boxplus\chi^{-1}$ and $\chi'\boxplus\chi'^{-1}$, respectively, and note that $\tau_v$ is isomorphic to $\tau'_v$ for all $v$ not in $S$. Therefore we have
\begin{align*}
L^S(s,\chi^2)\cdot\zeta^S_K(s)^2\cdot L^S(s,\chi^{-2}) = L^S(s,\tau\times\tau)&=L^S(s,\tau\times\tau') \\
&= L^S(s,\chi\chi')\cdot L^S(s,\chi\chi'^{-1})\cdot L^S(s,\chi^{-1}\chi')\cdot L^S(s,\chi^{-1}\chi'^{-1}).
\end{align*}
Since the left hand side has a pole at $s=1$, so does the right hand side. Hence one of the four factors on the right hand side has a pole at $s=1$, which implies that one of the $\{\chi\chi',\chi\chi'^{-1},\chi^{-1}\chi',\chi^{-1}\chi'^{-1}\}$ is trivial.
\end{proof}

\begin{proof}[Proof of Theorem \ref{thm:actualA}.(3)]
Proposition \ref{prop:allthetalifts} and Proposition \ref{prop:thetabackfull} yield an isomorphism of $G(\A_F)$-representations
\begin{align*}
\mathcal{A}_{\cusp}(G)\cap\mathcal{A}_\tau(G)\cong\bigoplus_{(\epsilon_v)_v}\bigotimes_v\!'\,\pi^{\epsilon_v}_v,
\end{align*}
where $(\epsilon_v)_v$ runs over sequences as in Theorem \ref{thm:actualA}.(3) such that, if $L(\textstyle\frac12,\chi^3)$ is nonzero, then not every $\epsilon_v$ equals $+1$. The description of $\mathcal{A}_{\res}(G)$ given by H. Kim \cite[Theorem 5.1]{kim1996residual} and \v{Z}ampera \cite[Theorem 1.1]{zampera1997residual} imply that $\mathcal{A}_{\res}(G)\cap\mathcal{A}_\tau(G)=0$ if $L(\frac12,\chi^3)$ is zero and $\mathcal{A}_{\res}(G)\cap\mathcal{A}_\tau(G)\cong\bigotimes'_v\pi_v^+$ if $L(\frac12,\chi^3)$ is nonzero. Finally, we conclude by noting that $\mathcal{A}_\tau(G)$ equals $(\mathcal{A}_{\cusp}(G)\cap\mathcal{A}_\tau(G))\oplus(\mathcal{A}_{\res}(G)\cap\mathcal{A}_\tau(G))$.
\end{proof}

\bibliographystyle{amsplain}
\bibliography{bibliography}

\end{document}